\documentclass{article}
\usepackage {amsmath, amsfonts, amssymb, mathrsfs, amsthm,stmaryrd, sectsty,hyphenat,etoolbox,tabu,calc,graphicx,url}
\usepackage[usenames]{color}
\usepackage{palatino}
\usepackage[all]{xy}
\usepackage[toc,page]{appendix}
\usepackage{enumitem}
\def\mf#1{\mathfrak{#1}}
\def\mc#1{\mathcal{#1}}
\def\mb#1{\mathbb{#1}}
\def\tx#1{\textrm{#1}}

\def\R{\mathbb{R}}

\def\C{\mathbb{C}}
\def\Q{\mathbb{Q}}

\def\Z{\mathbb{Z}}
\def\N{\mathbb{N}}

\def\ul#1{\underline{#1}}

\def\hat{\widehat}

\def\from{\leftarrow}

\def\lrw{\longrightarrow}

\def\sm{\smallsetminus}

\def\<{\langle}
\def\>{\rangle}

\newenvironment{mytitle}
{\begin{center}\large\sc}
{\end{center}}

\newcommand{\qedwhite}{\hfill\ensuremath{\Box}}

\newtheorem*{thm*}{Theorem}

\theoremstyle{definition}
\newtheorem{thm}{Theorem}[subsection]
\newtheorem{lem}[thm]{Lemma}
\newtheorem{pro}[thm]{Proposition}
\newtheorem{cor}[thm]{Corollary}
\newtheorem{fct}[thm]{Fact}

\newtheorem{cnj}[thm]{Conjecture}

\newtheorem{dfn}[thm]{Definition}
\newtheorem{rem}[thm]{Remark}
\newtheorem{exa}[thm]{Example}
\newtheorem{cns}[thm]{Construction}

\newtheorem{prp}[thm]{Properties}

\AtEndEnvironment{thm}{\qedwhite}
\AtEndEnvironment{lem}{\qedwhite}
\AtEndEnvironment{pro}{\qedwhite}
\AtEndEnvironment{cor}{\qedwhite}
\AtEndEnvironment{fct}{\qedwhite}
\AtEndEnvironment{clm}{\qedwhite}
\AtEndEnvironment{cnd}{\qedwhite}
\AtEndEnvironment{cnj}{\qedwhite}
\AtEndEnvironment{exn}{\qedwhite}
\AtEndEnvironment{dfn}{\qedwhite}
\AtEndEnvironment{rem}{\qedwhite}
\AtEndEnvironment{exa}{\qedwhite}
\AtEndEnvironment{cns}{\qedwhite}
\AtEndEnvironment{asm}{\qedwhite}
\AtEndEnvironment{prp}{\qedwhite}

\sectionfont{\center\sc\normalsize}
\subsectionfont{\bf\normalsize}

\numberwithin{equation}{subsection}

\newlength{\sumcorr}

\begin{document}

\begin{mytitle}
Covers of reductive groups and functoriality
\end{mytitle}

\begin{center}
Tasho Kaletha
\end{center}

\begin{abstract}
For a quasi-split connected reductive group $G$ over a local field $F$ we define a compact abelian group $\tilde\pi_1(G)$ and an extension $1 \to \tilde\pi_1(G) \to G(F)_\infty \to G(F) \to 1$ of topological groups equipped with a splitting over $G_\tx{sc}(F)$. Any character $x : \tilde\pi_1(G) \to \mu_n(\C)$ leads to an $n$-fold cover $G(F)_x$ of $G(F)$ via pushout. We define an $L$-group $^LG_x$ for this cover that is generally a non-split extension of $\tx{Gal}(F^s/F)$ by $\hat G$. We prove a refined local Langlands correspondence for $G(F)_x$, assuming it is known for connected reductive groups with the same adjoint group as $G$.

Motivation for this construction comes from considerations of Langlands' functoriality conjecture, where subgroups $\mc{H} \subset {^LG}$ of the 
$L$-group of $G$ arise that need not be $L$-groups of other reductive groups. If such a subgroup is full and intersects $\hat G$ in a connected reductive subgroup of maximal rank, we construct a natural triple $(H,x,\xi)$ consisting of a quasi-split connected reductive group $H$, a double cover $H(F)_x$, and an $L$-embedding $\xi : {^LH}_x \to {^LG}$ that is an isomorphism onto $\mc{H}$. We expect that genuine representations of $H(F)_x$ transfer functorially to representations of $G(F)$.

In the special case of endoscopy, we show that the construction of transfer factors simplifies when the natural double cover $H(F)_x$ of the endoscopic group is used. The transfer factor becomes the product of two natural invariants that do not depend on auxiliary choices. One of them is closely related to Kottwitz's work \cite{Kot99} on transfer factors for Lie algebras. The other one is not specific to the case of endoscopy, and will likely play a role in general functoriality questions.

For $F=\R$, analogues of $\tilde\pi_1(G)$ and $G(F)_\infty$ were constructed by Adams and Vogan \cite{AV92} and our work is motivated by their ideas.

\end{abstract}

{\let\thefootnote\relax\footnotetext{This research is supported in part by NSF grant DMS-1801687 and a Simons Fellowship.}}

\tableofcontents

\section{Introduction}

Let $F$ be a local field and let $G$ be a connected reductive $F$-group. According to the Langlands conjectures, the representation theory of the topological group $G(F)$ should be governed by the $L$-group $^LG = \hat G \rtimes \Gamma$, where $\Gamma$ is the absolute Galois group of $F$ relative to a fixed separable closure $F^s$, and $\hat G$ is the dual group of $G$. The Langlands reciprocity conjecture predicts a correspondence between $L$-homomorphisms $L_F \to {^LG}$ and representations of $G(F)$, where $L_F$ is the local Langlands group of $F$, defined to be the Weil group $W_F$ when $F$ is archimedean, and the group $W_F \times \tx{SL}_2(\C)$ when $F$ is non-archimedean. The Langlands functoriality conjecture predicts a relationship between the representations of two groups $G_1(F)$ and $G_2(F)$ if there is an $L$-homomorphism $^LG_1 \to {^LG}_2$.

While reflecting on these questions one often encounters groups $\mc{G}$ which, like the group $^LG$, are extensions of $\Gamma$ by $\hat G$, but may not possess all properties of $^LG$. For example, the extension $\mc{G}$ may not be split. Or it may be split, but not equipped with a splitting. Or a splitting may be given, but the resulting action of $\Gamma$ on $\hat G$ may not preserve a pinning. The question then arises about the meaning of such groups $\mc{G}$ with relation to Langlands' conjectures.

The following are some examples of the occurrence of such groups. 
\begin{enumerate}
	\item Consider a discrete $L$-parameter $\varphi : W_F \to {^LG}$. If $F$ is non-archimedean, assume that $G$ splits over a tame extension of $F$ and the residual characteristic $p$ does not divide the order of the Weyl group of $G$. Then $\varphi$ determines a subgroup $\mc{S} \subset {^LG}$ that contains the image of $\varphi$ and is an extension $1 \to \hat S \to \mc{S} \to \Gamma \to 1$, where $\hat S$ is the dual of an elliptic maximal torus of $G$, cf. \cite[\S3.4]{She83} and \cite[\S4.1]{KalSLP}.

	When $F=\R$, $\mc{S}$ is usually not the $L$-group of $S$. When $F$ is non\-archimedean, $\mc{S}$ is non\-canonically isomorphic to $^LS$, but the set of possible isomorphisms is a torsor under the group of all characters of $S(F)$, so it is not clear how to obtain from the factored parameter $\varphi : W_F \to \mc{S}$ a character of $S(F)$ in a canonical way.
	
	\item Consider an endoscopic datum $(H,s,\mc{H},\eta)$ for $G$, as defined in \cite[\S2.1]{KS99}. The group $\mc{H}$ is an extension of $\Gamma$ by $\hat H$ that is split, but no splitting is given, and it is not assumed that there exists a splitting that gives an action of $\Gamma$ on $\hat H$ preserving a pinning. In fact, there are situations where such a splitting does not exist. Therefore, $\mc{H}$ need not be isomorphic to $^LH$.
	
	\item In the study of functoriality and questions of ``Beyond Endoscopy'', as for example in \cite{Lan04}, subgroups of $^LG$ denoted by $^\lambda H_\pi$ arise. These groups can often be assumed to be extensions of $\Gamma$ by a connected group $\hat H$, cf. \cite[\S1.2, \S1.4]{Lan04}, at least upon replacing $\Gamma$ by $\tx{Gal}(E/F)$ for a suitable finite Galois extension $E/F$. Similar groups occur in the discussion of \cite[\S2]{Art17}, where they are denoted by $\mc{G}'$ and are part of a ``beyond endoscopic datum''.
\end{enumerate}

The lack of an isomorphism $\mc{G} \to {^LG}$ can be resolved by a technical work-around. When $G$ is a torus, which is the setting 1. above (where now $S$ plays the role of $G$), using the Weil group to form $^LG$ in place of the Galois group leads to an extension of $W_F$ by $\hat G$, and such an extension is always split, cf. Proposition \ref{pro:cc_rajan}. However, there is no distinguished choice of splitting. When $G$ is a general connected reductive group, this procedure doesn't always guarantee that $\mc{G}$ will become isomorphic to $^LG$, and one needs to combine it with a second procedure, that of replacing $G$ by a central extension $G_1 \to G$ whose kernel is an induced torus. This is a so called $z$-extension, introduced by Langlands. There is always an $L$-embedding $\mc{G} \to {^LG_1}$, and one can now use $G_1$ in place of $G$ in the study of functoriality. However, the problem remains that there is no natural choice for a $z$-extension, nor for an $L$-embedding $\mc{G} \to {^LG_1}$. 

In this paper we propose a solution to this problem. Given a quasi-split connected reductive group $G$ we define (Definition \ref{dfn:tildepi}) a compact abelian group $\tilde\pi_1(G)$, functorial in $G$, and an extension (cf. \eqref{eq:uni}) of topological groups
\[ 1 \to \tilde\pi_1(G) \to G(F)_\infty \to G(F) \to 1. \]
This extension is equipped with a splitting over $G_\tx{sc}(F)$, i.e. a continuous homomorphism $G_\tx{sc}(F) \to G(F)_\infty$ that lifts the natural homomorphism $G_\tx{sc}(F) \to G(F)$. The definition of $\tilde\pi_1(G)$ is based on Borovoi's algebraic fundamental group \cite{Brv98}. The continuous characters of $\tilde\pi_1(G)$ are in bijection with the group $Z^2(\Gamma,Z(\hat G))$. If $x : \tilde\pi_1(G) \to \mb{S}^1$ is such a character, we obtain via pushout an extension 
\[ 1 \to \mb{S}^1 \to G(F)_x \to G(F) \to 1 \]
as well as an extension (cf. \eqref{eq:h-l1})
\[ 1 \to \hat G \to {^LG_x} \to \Gamma \to 1, \]
equipped with a set-theoretic splitting. The following is an imprecise version of Theorem \ref{thm:llc-h}.

\begin{thm*}
Assume that the refined local Langlands correspondence is known for $z$-extensions of $G$. Then the analogous formulation is known for genuine representations of the topological group $G(F)_x$ and $L$-parameters valued in $^LG_x$.	
\end{thm*}

When the character $x$ is of finite order, thus valued in $\mu_n(\C)$ for some $n$, then we have analogously the extension
\[ 1 \to \mu_n(\C) \to G(F)_{x,n} \to G(F) \to 1. \]
Thus $G(F)_{x,n}$ is a finite cover of $G(F)$. The above theorem holds for $G(F)_{x,n}$ as well.

The proof of this theorem relies on the fact that, when the derived subgroup of $G$ is simply connected, there exists a (non-canonical) $L$-isomorphism ${^LG} \to {^LG_x}$ between the Weil-forms of the $L$-groups of $G$ and $G(F)_x$, as well as a (non-canonical) genuine  character of $G(F)_x$, and the fact that these two objects can be chosen compatibly. Tensoring with the genuine character provides a bijection between the representations of $G(F)$ and the genuine representations of $G(F)_x$, and this bijection is the reflection of the $L$-isomorphism ${^LG} \to {^LG_x}$. In hindsight, composing the natural section $W_F \to {^LG}$ with this $L$-isomorphism provides the parameter for the genuine character. When the derived subgroup of $G$ is not simply connected there need not exist a genuine character of $G(F)_x$, or an $L$-isomorphism ${^LG} \to {^LG_x}$. One reduces the general case to the case when the derived subgroup of $G$ is simply connected by using a $z$-extension.

The application of this theorem to the problem we just sketched is the following result, proved in \S\ref{sec:l-covers}.

\begin{thm*}
Let $G$ be a connected reductive group and let $\mc{H} \subset {^LG}$ be a subgroup, given up to conjugation by $\hat G$, such that
\begin{enumerate}
	\item The projection $^LG \to \Gamma$ remains surjective upon restriction to $\mc{H}$.
	\item The intersection $\hat H = \mc{H} \cap \hat G$ is a connected reductive subgroup of maximal rank.
\end{enumerate}
There exists a tuple $(H,x,\xi)$ consisting of a quasi-split connected reductive group $H$, a character $x : \bar\pi_1(H) \to \{\pm 1\}$, and an $L$-embedding $^LH_x \to {^LG}$ that is an isomorphism onto $\mc{H}$. This triple is unique up to isomorphism.
\end{thm*}

\noindent Here $\bar\pi_1(H)$ is a slight variation of $\tilde\pi_1(H)$ that is introduced in \S\ref{sub:def}. The group $H(F)_x$ is a double cover of $H(F)$, and is not the group of $F$-points of a connected reductive group $H_x$.

This theorem answers the question as to which group is the partner to $G(F)$ in the conjectural functorial transfer that is embodied by the inclusion $\mc{H} \to {^LG}$. Examples of groups that satisfy the assumptions of the theorem are those in cases 1. and 2. discussed above, as well as some of those in case 3.

Let us examine case 1. The double cover $S(F)_x$ of the elliptic maximal torus $S(F)$ of $G(F)$ had already been constructed by different means in \cite{KalDC}. It was discussed there (\cite[\S4.2]{KalDC}) that a discrete $L$-parameter for $G$, with the proviso made in 1. when $F$ is non-archimedean, provides canonically a genuine character of the double cover $S(F)_x$, and that the stable character of the conjectural $L$-packet on $G(F)$ corresponding to that $L$-parameter can be written down explicitly in terms of that genuine character (\cite[\S4.4]{KalDC}).  The construction of $S(F)_x$ in \cite{KalDC} is quite different from the one given here. It has the virtue of being explicit, but only applies to tori and to very specific $x$. The construction given here applies to general $x$, but is not explicit.

Let us now examine case 2. Since the above theorem reveals the double cover $H(F)_x$ as the natural partner to $G(F)$ for endoscopic transfer, the question arises if one can formulate the Langlands--Shelstad--Kottwitz transfer factors in terms of $H(F)_x$ and state appropriate geometric and spectral transfer theorems or conjectures. Not only is that possible, but we find that the structure of the transfer factor becomes more transparent. In \cite[\S3]{LS87}, the transfer factor is defined as a product of 5 terms 
\[ \Delta = \Delta_I \cdot \Delta_{II} \cdot \Delta_{III_1} \cdot \Delta_{III_2} \cdot \Delta_{IV}. \]
The final term $\Delta_{IV}$ is a quotient of Weyl discriminants, and is thus of a simple nature, but the other terms are somewhat involved, and depend on auxiliary choices, called $a$-data and $\chi$-data. While each of these terms can be motivated in a certain way, the true nature of $\Delta_I$ and $\Delta_{II}$ may appear mysterious. Moreover, while the product $\Delta$ does not depend on $a$-data or $\chi$-data, it does depend on other auxiliary choices, namely that of a $z$-extension $H_1 \to H$ and $L$-embedding $\mc{H} \to {^LH_1}$.

On the other hand, we define in this paper (cf. \eqref{eq:tf_pin}) a transfer factor $\Delta_x$ that is just the product of three terms
\[ \<\tx{inv}(\tx{pin},-),s\>^{-1} \cdot \<\tx{inv}_\mc{H},-\> \cdot \Delta_{IV}. \]
In fact, following Waldspurger, we drop $\Delta_{IV}$ and agree to normalize orbital integrals and characters accordingly. Each of the terms $\<\tx{inv}(\tx{pin},-),s\>$ and $\<\tx{inv}_\mc{H},-\>$ has transparent structure, and does not depend on auxiliary data such as $a$-data or $\chi$-data. 

To explain the structure of $\<\tx{inv}(\tx{pin},-),s\>$ we recall Kottwitz's result \cite{Kot99} that, in the case of Lie algebras over a local field of characteristic zero, the transfer factor has the following very simple description: $\Delta_\tx{Lie}(X^H,X^G)=1$ if $G$ is quasi-split and $X^G \in \tx{Lie}(G)(F)$ meets the Kostant section for the fixed pinning of $G$, and in general is given by the usual $\kappa$-behavior in the variable $X^G$. Said slightly differently, there exists a cohomological invariant $\tx{inv}(\tx{pin},X^G)$ that measures the relative position of $X^G$ and the fixed pinning of $G$, and 
\[ \Delta_\tx{Lie}(X^H,X^G)=\<\tx{inv}(\tx{pin},X^G),s\>^{-1} \] 
is just the pairing of that cohomological invariant with the endoscopic element $s$ (one needs to transport $s$ to the right place, and this depends on $X^H$; moreover, the inverse we have placed here is not in \cite{Kot99}, but is dictated by the conventions of \cite{KS12} that we follow, see \S\ref{sub:endoreview}). 

It turns out that, in the group case, there is an analogous object. Namely, there is a cohomological invariant (see \S\ref{sub:rel1}) that measures the relative position of a pinning of $G$ and a strongly regular semi-simple element $\delta_{\pm\pm}$ of a certain cover $S(F)_{\pm\pm}$ of a maximal torus $S \subset G$, and the term $\<\tx{inv}(\tx{pin},\delta_{\pm\pm}),s\>$ is again simply the pairing of this cohomological invariant with the endoscopic element $s$. This invariant is a globalization (to the whole group) of the infinitesimal invariant in Kottwitz's work, in the sense that there is an open neighborhood $V \subset \tx{Lie}(S)(F)$ of $0$ such that the exponential map $\exp : V \to S(F)_{\pm\pm}$ converges and $\tx{inv}(\tx{pin},\exp(X^G))=\tx{inv}(\tx{pin},X^G)$, where on the left we have used the invariant for the group, and on the right the invariant for the Lie algebra, cf. Lemma \ref{lem:expinv}. Of course, not all elements in $S(F)_{\pm\pm}$ lie in the image of the exponential map.

We point out that, unlike for Lie algebras, the definition of the invariant $\tx{inv}(\tx{pin},-)$ in the group case requires the use of covers. We do not believe that it is possible to define such an invariant without using covers. We further point out that, by using rigid inner forms, this invariant can be defined for arbitrary connected reductive $G$, not necessarily quasi-split, and we work in this generality in \S\ref{sub:rel1}.

The term $\<\tx{inv}_\mc{H},-\>$ also has transparent structure. There is a diagram of $L$-embeddings
\[ \xymatrix{
	^LS_{x,H}\ar[r]&^LH_x\ar[d]\\
	^LS_G\ar[r]&^LG
}
\]
each of which is canonical (up to conjugation). This diagram leads naturally (cf. \S\ref{sub:rel2}) to a genuine character of a cover of $S(F)$ that arises as a Baer sum $S(F)_{G/H} \oplus S(F)_x$, and $\tx{inv}_\mc{H}$ is just this character. The product $\<\tx{inv}(\tx{pin},-),s\>^{-1} \cdot \<\tx{inv}_\mc{H},-\>$ becomes a function of $H(F)_x \times G(F)$ that is genuine on $H(F)_x$.

In order to obtain a balanced theory, this discussion needs to be performed not just with the group $G$ and its $L$-group $^LG$, but also with any cover of $G$ and its associated $L$-group. This does not require much additional effort and is done in \S\ref{sub:l2-covers} and \S\ref{sub:tf2}.

It is clear from the construction that the transfer factor $\Delta_x$ is compatible with the Lie algebra transfer factor (the term $\tx{inv}_\mc{H}$, being a smooth character of a cover of a torus, vanishes near the identity). The transfer factor $\Delta_x$ is also compatible with the original transfer factor of Langlands--Shelstad--Kottwitz, as is verified in \S\ref{sub:ks_comp}. Via this compatibility we are able to derive the following two basic results from their classical analogs. 
\begin{enumerate}
	\item Every anti-genuine function $f^G \in \mc{C}^\infty_c(G(F)_{x_G})$ has a matching anti\-genuine function $f^H \in \mc{C}^\infty_c(H(F)_{x_H})$, cf. Theorem \ref{thm:orbit}.
	\item Assume that endoscopic character identities hold for connected reductive groups with the same derived subgroup of $G$ and their endoscopic groups corresponding to $H$. Then they also hold for $G(F)_{x_G}$ and $H(F)_{x_H}$. More precisely, 
	\[ \Theta_{\varphi_G}^{s,\mf{w}}(f^G)=S\Theta_{\varphi_H}(f^H), \]
	for any tempered $L$-parameter $\varphi_H$ valued in $^LH_{x_H}$ and its composition $\varphi_G$ with the canonical $L$-embedding $^LH_{x_H} \to {^LG}_{x_G}$, cf. Theorem \ref{thm:charid}.
\end{enumerate}
It may be worth pointing out how this discussion relates to Langlands' result \cite{Lan79} that, when the derived subgroup of $G$ is simply connected, there exists an $L$-embedding $^LH \to {^LG}$ between the Weil forms of the $L$-groups. From our current point of view this result may be interpreted as the existence of an isomorphism $^LH \to {^LH_x}$ between the Weil-forms of the $L$-groups. This isomorphism, together with the canonical $L$-embedding ${^LH_x} \to {^LG}$, recovers Langlands' $L$-embedding.

In the more general setting of twisted endoscopy, an $L$-embedding $^LH \to {^LG}$ may fail to exist even when the derived subgroup of $G$ is simply connected. This is why \cite[\S2.2]{KS99} introduced the notion of a $z$-pair, which consists of a $z$-extension $H_1 \to H$ and an $L$-embedding $^LH \to {^LH_1}$, and phrased the construction of the transfer factor in terms of that choice. But when the canonical double cover $H(F)_x$ and the transfer factor $\Delta_x$ is used, the need for such auxiliary choices disappears.

This ends our discussion of 2. We do not have anything serious to say about 3., but can offer some vague remarks. It is expected that there should be ``beyond endoscopy transfer factors'', as for example discussed in \cite{Lan13}. With this in mind, one can ask if some parts of these may be related to parts of the endoscopic transfer factors. A related question may be, which parts of the endoscopic transfer factors are of truly endoscopic nature, and which are more general. The appearance of the term $\Delta_{II}^\tx{abs}$ in the character formula for supercuspidal representations in \cite[\S4]{KalRSP} suggested that some parts of the endoscopic transfer factor are intrinsic to the group $G$, and thus not of endoscopic nature. The construction of the transfer factor $\Delta_x$ discussed above offers some insight into this question. It is clear that the term $\tx{inv}_\mc{H}$ is not of endoscopic nature, and will likely play a role in more general considerations of functoriality. The term $\<\tx{inv}(\tx{pin},-),s\>$ arises as the pairing of a cohomological invariant that is again intrinsic to $G$ with the endoscopic element $s$. Thus, while this term itself is of endoscopic nature, the cohomological invariant might again play a role in more general considerations of functoriality. For example, it is implicit in \cite{KalRSP} that this invariant is closely related to the genericity of supercuspidal representations of $G(F)$.

The idea that non-split extensions of $\Gamma$ by $\hat G$ are to be viewed as $L$-groups of covers of $G(F)$ is not new. Adams and Vogan proposed this approach in \cite{AV92} for real reductive groups, and our work is inspired by theirs. Given such a group $G/\R$ they define a finitely generated abelian group $\pi_1(G)(\R)$ and an extension
\[ 1 \to \pi_1(G)(\R) \to G(\R)^\sim \to G(\R) \to 1, \]
cf. \cite[(5.3)(c),(7.11)(a,c)]{AV92}. They also state a basic form of the local Langlands correspondence for the resulting covers (\cite[Theorem 10.7]{AV92}).

The constructions of Adams and Vogan use the particularities of the real numbers, most notably that the Galois group is finite of order $2$, and in particular its cohomology is cyclic with period $2$. They also use the fact that $G(\R)$ is a real Lie group. This makes it unclear if these constructions can be carried out for non-archimedean fields. Our construction is different, but we are able to compare it with that of Adams and Vogan in \S\ref{sub:avcomp}. The rough result is the following.

\begin{thm*}\ \\[-15pt]
\begin{enumerate}
	\item There exists a natural homomorphism $\pi_1(G)(\R) \to \tilde\pi_1(G)$, which is injective and realizes the component group $\pi_0(\tilde\pi_1(G))$ as the profinite completion of $\pi_1(G)(\R)$. 
	\item The extension $G(\R)_\infty$ is the pushout of $G(\R)^\sim$ under the above homomorphism.
\end{enumerate}
\end{thm*}
According to this theorem, there is a natural bijection between the \emph{finite-order} characters of $\pi_1(G)(\R)$ and $\tilde\pi_1(G)$. The resulting covers of $G(\R)$ are canonically isomorphic. It is not clear if the covers obtained from infinite-order characters can be compared.

We close this introduction with a brief remark about the relationship between this paper and other works on covering groups and their $L$-groups, such as \cite{Del96}, \cite{BD01}, \cite{Weissman18a}, \cite{Weissman18b}, and \cite{Zhao22}. The purpose of those papers is to obtain a class of covers of reductive groups that is as large as possible. These constructions can be rather complicated, and establishing Langlands reciprocity or endoscopic functoriality for them in any generality appears currently well out of reach. Indeed, already the experience with the special case of the mateplectic group through the work of Wen-Wei Li \cite{WenWeiLi11}, \cite{WenWeiLi19}, \cite{WenWeiLi20} has shown that this problem is deep. Even the case of covers of general tori is complicated and not fully known; the case of covers of split tori is treated in \cite{Weissman18b}.

Our purpose here is rather orthogonal. Our focus is on linear reductive groups, and the covers are brought in to clarify notions and constructions involved in the study of functorliality for linear groups, such as endoscopic groups and transfer factors, as well as more general cases of functoriality. For this we have isolated a small class of covers that offers a natural setting for studying these phenomena while also keeping to a minimum the required technology and effort. We have established in \S\ref{sub:llc} and \S\ref{sec:endo} both reciprocity and endoscopic functoriality for these covers, assuming that it holds for closely related linear groups. This allows these covers to be used effectively in the study of linear groups, for example through the approach via induction on dimension due to Langlands and Arthur. 

It would be fruitful to compare the constructions of covers given here with these other works and to contemplate to what extent the discussions of \S\ref{sub:llc} and \S\ref{sec:endo} may be extended to a more general setting. 

\newpage

\section{The pseudo-fundamental group} \label{sec:fund}

Let $F$ be a local field, $F^s$ a fixed separable closure, and $\Gamma=\tx{Gal}(F^s/F)$. Let $G$ be a quasi-split connected reductive group over a field $F$.

\subsection{Definition of the pseudo-fundamental group} \label{sub:def}

Let $\pi_1(G)$ be Borovoi's algebraic fundamental group, cf. Appendix \ref{app:borovoi}. 

\begin{dfn} \label{dfn:tildepi}
Let $\tilde\pi_1(G)$ be the compact abelian group that is Pontryagin dual to the discrete abelian group $Z^2(\Gamma,\tx{Hom}_\Z(\pi_1(G),\C^\times))$.
\end{dfn}

\begin{fct}
The assignment $G \mapsto \tilde\pi_1(G)$ is a covariant functor from the category of connected reductive $F$-groups to the category of compact topological groups. If $\pi_1(G) \to \pi_1(H)$ is surjective, then so is $\tilde\pi_1(G) \to \tilde\pi_1(H)$.
\end{fct}

In terms of the dual group $\hat G$, the compact group $\tilde\pi_1(G)$ is the Pontryagin dual of the discrete abelian group $Z^2(\Gamma,Z(\hat G))$. In the following subsections we will define a central extension of locally compact groups
\begin{equation} \label{eq:uni}
1 \to \tilde\pi_1(G) \to G(F)_\infty \to G(F) \to 1
\end{equation}
equipped with a splitting over $G_\tx{sc}(F)$, i.e. with a homomorphism $s : G_\tx{sc}(F) \to G(F)_\infty$ that lifts the natural map $G_\tx{sc}(F) \to G(F)$. We will call \eqref{eq:uni} the \emph{pseudo-universal extension}, and $s$ a \emph{splitting over $G_\tx{sc}$.}

From any character $t : \tilde\pi_1(G) \to \mb{S}^1$ we obtain from \eqref{eq:uni} via pushout an extension $G(F)_t$ of $G(F)$ by $\mb{S}^1$. If the character is of finite order $n$, thus valued in $\mu_n(\C)$, we obtain in the same way an extension $G(F)_{t,n}$ of $G(F)$ by $\mu_n(\C)$. We can push-out the latter under the inclusion $\mu_n(\C) \to \mu_\infty(\C)$ to obtain an extension $G(F)_{t,\infty}$; here we are taking the discrete topology on $\mu_\infty(\C)$.

\begin{rem}
As reviewed in Appendix \ref{app:cag}, we have the exact sequence
\[ 1 \to \tilde\pi_1(G)^\circ \to \tilde \pi_1(G) \to \pi_0(\tilde \pi_1(G)) \to 1, \]
and a character of $\tilde\pi_1(G)$ has finite order if and only if it factors through the quotient $\pi_0(\tilde \pi_1(G))$, which is profinite. Throughout most of the paper we will only be concerned with characters of $\tilde\pi_1(G)$ that are of finite order. A character $t : \tilde\pi_1(G) \to \mb{S}^1$ has finite order if and only if it takes values in $\mu_\infty(\C)$, i.e. if and only if each of its values has finite order, cf. Lemma \ref{lem:surjfin}.
\end{rem}

\begin{dfn} \label{dfn:gen}
Let $X$ be a set with an action of $\C^\times$. A function $f : G(F)_{t,n} \to X$ is called \emph{genuine} if $f(\epsilon \cdot g)=\epsilon \cdot f(g)$ for $g \in G(F)_{t,n}$ and $\epsilon \in \mu_n(\C)$. It is called \emph{anti-genuine} if instead $f(\epsilon \cdot g)=\epsilon^{-1} \cdot f(g)$.
\end{dfn}

The main cases we have in mind are $X=\C$, $X=\C^\times$, $X=\mu_n(\C)$, $X=\mb{S}^1$, or $X=\tx{Aut}(V)$ for a complex vector space $V$.

\begin{rem} \label{rem:antigen}
The identity automorphism of $G(F)_\infty$ induces an isomorphism of topological groups $G(F)_{t,n} \to G(F)_{t^{-1},n}$ whose restriction to $\mu_n(\C)$ is the inversion automorphism. In this way, anti-genuine functions of $G(F)_{t,n}$ become identified with genuine functions of $G(F)_{t^{-1},n}$.
\end{rem}

\begin{rem}
We recall the following two basic constructions of extensions. Given $t_1,t_2 : \tilde\pi_1(G) \to \mu_n(\C)$, the \emph{Baer sum} $G(F)_{t_1,n} \oplus G(F)_{t_2,n}$ is defined as the pushout along the multiplication map $\mu_n(\C) \times \mu_n(\C) \to \mu_n(\C)$ of the fiber product $G(F)_{t_1,n} \times_{G(F)} G(F)_{t_2,n}$. Given $t : \tilde\pi_1(G) \to \mu_n(\C)$ the \emph{Baer inverse} of $G(F)_{t,n}$ is the pushout of $G(F)_{t,n}$ under the inversion map $\mu_n(\C) \to \mu_n(\C)$.

In our situation, these constructions behave as follows. The Baer inverse of $G(F)_{t,n}$ is naturally identified with $G(F)_{t^{-1},n}$. The homomorphism $G(F)_\infty \to G(F)_{t_1,n} \oplus G(F)_{t_2,n}$ induced by the diagonal embedding $G(F)_\infty \to G(F)_\infty \times G(F)_\infty$ factors through an isomorphism $G(F)_{t_1t_2,n} \to G(F)_{t_1,n} \oplus G(F)_{t_2,n}$.
\end{rem}

We now make some remarks about the dependence of the cover $G(F)_{t,n}$ on $n$.

\begin{rem}
Consider a multiple $m$ of $n$. Pushing out the central extension $G(F)_{t,n}$ of $G(F)$ by $\mu_n(\C)$ along the natural inclusion $\mu_n(\C) \to \mu_m(\C)$ we obtain the central extension $G(F)_{t,m}$. It is possible that for $t,t' : \tilde\pi_1(G) \to \mu_n(\C)$ the extensions $G(F)_{t,n}$ and $G(F)_{t',n}$ are not isomorphic, but the extensions $G(F)_{t,m}$ and $G(F)_{t',m}$ are isomorphic, cf. Remark \ref{rem:isofuse}.
\end{rem}

\begin{rem} \label{rem:n-m}
If $\pi$ is a genuine representation of $G(F)_{t,n}$, then the representation $\pi\boxtimes\tx{id}$ of $G(F)_{t,n} \times \mu_m(\C)$ transforms trivially under the antidiagonal embedding of $\mu_n$ into this product, and hence descends to a genuine representation of $G(F)_{t,m}$. This provides a functor between the categories of genuine representations of $G(F)_{t,n}$ and $G(F)_{t,m}$, and this functor is an equivalence, its inverse being restriction along $G(F)_{t,n} \to G(F)_{t,m}$. The same argument applies when $G(F)_{t,n}$ is replaced by $G(F)_t$.
\end{rem}

We now introduce a variation that will be quite useful. Let $T$ be the universal maximal torus of $G$ and let $\hat T$ be its dual. There is a natural embedding $Z(\hat G) \to \hat T$ and we set $\hat T_\tx{ad} = \hat T/Z(\hat G)$. Consider the complex of tori $[\hat T \to \hat T_\tx{ad}]$, where $\hat T$ is placed in degree $0$ and $\hat T_\tx{ad}$ is placed in degree $1$. We have the group of hypercocycles $Z^2(\Gamma,\hat T \to \hat T_\tx{ad})$. Our conventions are those of \cite[Appendix A]{KS99}. Thus this group consists of $z \in Z^2(\Gamma,\hat T)$ and $c \in C^1(\Gamma,\hat T_\tx{ad})$ such that $\partial c=\bar z$, where $\bar z \in Z^2(\Gamma,\hat T_\tx{ad})$ is the image of $z$.

\begin{dfn} \label{dfn:barpi}
Let $\bar\pi_1(G)$ be the compact abelian group that is Pontryagin dual to the discrete abelian group $Z^2(\Gamma,\hat T \to \hat T_\tx{ad})$. 
\end{dfn}

From the injection $Z^2(\Gamma,Z(\hat G)) \to Z^2(\Gamma,\hat T \to \hat T_\tx{ad})$ we obtain the natural surjective map $\bar\pi_1(G) \to \tilde\pi_1(G)$. The advantage of $\bar\pi_1(G)$ is that the extension \eqref{eq:uni} is obtained as a push-out of an analogous extension 
\begin{equation} \label{eq:baruni}
1 \to \bar\pi_1(G) \to G(F)_\infty \to G(F) \to 1.
\end{equation}
Any character $t : \tilde\pi_1(G) \to \mu_n(\C)$ pulls back to a character $t : \bar\pi_1(G) \to \mu_n(\C)$ and the corresponding extension $G(F)_{t,n}$ obtained from \eqref{eq:uni} is also obtainable from \eqref{eq:baruni}. In that sense, $\bar\pi_1(G)$ is the more primordial object. We will see that the theory of endoscopy leads to covers of $G(F)$ that most naturally present themselves via characters of $\bar\pi_1(G)$.

The disadvantage of $\bar\pi_1(G)$ is that its functoriality in $G$ rather limited (it respects isomorphisms, and embeddings of Levi subgroups, for example). Therefore, for some purposes, $\tilde\pi_1(G)$ is better behaved. We will see however in \S\ref{sub:elementary} that in many important cases functorial properties involving $\bar\pi_1(G)$ may be obtained from those of $\tilde\pi_1(G)$.

We will further see in \S\ref{sub:elementary} that, while not every character $t : \bar\pi_1(G) \to \mu_n(\C)$ descends to $\tilde\pi_1(G)$, we can always find a character $t' : \tilde\pi_1(G) \to \mu_m(\C)$ such that the extensions $G(F)_{t,m}$ and $G(F)_{t',m}$ are isomorphic. The problem is that $m$ is usually not equal to $n$, but rather a multiple of $n$, and the isomorphism between $G(F)_{t,m}$ and $G(F)_{t',m}$ is not unique, so one must keep track of it.

\subsection{Covers of tori} \label{sub:covers_tori}

We begin the construction and study of \eqref{eq:uni} in the special case that $G$ is an $F$-torus, which we shall denote by $S$ for emphasis. Then $\tilde\pi_1(S)=\bar\pi_1(S)$ is the Pontryagin dual of $Z^2(\Gamma,\hat S)$, where $\hat S$ is the complex dual torus of $S$. We write $C^1_u(W_F,\hat S)$ for the group of continuous 1-cochains $c:  W_F \to \hat S$ whose image is bounded. Note that $B^1(W_F,\hat S) = B^1(\Gamma,\hat S) \subset C^1(\Gamma,\hat S) \subset C^1_u(W_F,\hat S)$.

\begin{lem} \label{lem:s-key}
Let $t \in Z^2(\Gamma,\hat S)$. There exists $c \in C^1_u(W_F,\hat S)$ such that $\partial c=t$.
\end{lem}
\begin{proof}
If we only require that $c$ lie in $C^1(W_F,\hat S)$, rather than the subgroup $C^1_u(W_F,\hat S)$, then this is \cite[Lemma 4]{Lan79}. The proof given in loc. cit. rests on the surjective homomorphism $H^1_c(W_{K/F},S_1) \to H^1_c(W_{K/F},S_2)$ marked by (1), where $H^1_c$ is the notation used for cohomology classes represented by continuous cocycles; in this paper we are not using a subscript for this. If in this proof we replace that homomorphism with the homomorphism $H^1_u(W_{K/F},S_1) \to H^1_u(W_{K/F},S_2)$, which is still surjective by Pontryagin duality, we obtain the desired result.
\end{proof}

\noindent Define
\begin{eqnarray*}
\tilde Z^1(W_F,X)&=&\{ c \in C^1(W_F,X)\,|\, \partial c \in Z^2(\Gamma,X)\}\\
\tilde Z^1_u(W_F,X)&=&\tilde Z^1(W_F,X) \cap C^1_u(W_F,X)\\
\tilde H^1(W_F,X)&=&\tilde Z^1(W_F,X)/B^1(W_F,X)\\
\tilde H^1_u(W_F,X)&=&\tilde Z^1_u(W_F,X)/B^1(W_F,X)\\
\tilde H^i(\Gamma,X)&=&C^i(\Gamma,X)/B^i(\Gamma,X),
\end{eqnarray*}
where $X$ is any $\Gamma$-module in the first, third, and fifth lines, and a complex torus in the second and fourth lines. The differential $\partial : \tilde Z^1(W_F,\hat S) \to Z^2(\Gamma,\hat S)$ descends to the group $\tilde H^1(W_F,\hat S)$ and induces a group homomorphism that is surjective, in fact remains surjective even after restriction to $\tilde H^1_u(W_F,\hat S)$ according to Lemma \ref{lem:s-key}.

The Langlands correspondence (cf. Appendix \ref{app:langtori}) for $S$ identifies the abelian group $H^1_u(W_F,\hat S)$ with the Pontryagin dual $\tx{Hom}_\tx{cts}(S(F),\mb{S}^1)$ of $S(F)$. The latter has a natural topology, namely the compact-open topology. We endow $H^1_u(W_F,\hat S)$ with the transport of that topology via the Langlands isomorphism. We endow $\tilde H^1_u(W_F,\hat S)$ with the unique topology for which $H^1_u(W_F,\hat S)$ is an open subgroup. Then the exact sequence
\begin{equation}
1 \to H^1_u(W_F,\hat S) \to \tilde H^1_u(W_F,\hat S) \stackrel{-\partial}{\lrw} Z^2(\Gamma,\hat S) \to 1
\end{equation}
of abstract groups becomes an exact sequence of locally compact topological groups, with $Z^2(\Gamma,\hat S)$ discrete. Note that we have used the \emph{negative} of the differential. Pontryagin duality turns this exact sequence into the desired exact sequence \eqref{eq:uni} for $G=S$.

We will now discuss the Langlands duality statement for the cover $S(F)_\infty$, as well as for covers obtained from it via characters of $\tilde\pi_1(S)$.

\begin{pro} \label{pro:llc-s}
\textit{(The local correspondence):}
\begin{enumerate}
	\item There is a natural isomorphism between the group $\tilde H^1(W_F,\hat S)$ and the group of continuous characters $S(F)_\infty \to \C^\times$, under which the subgroup $\tilde H^1_u(W_F,\hat S)$ is identified with the subgroup of unitary characters.
	\item If $\chi_\infty : S(F)_\infty \to \C^\times$ corresponds to $[c] \in \tilde H^1(W_F,\hat S)$, the restriction of $\chi_\infty$ to $\tilde\pi_1(S) \to \C^\times$ equals the character $-\partial c \in Z^2(\Gamma,\hat S)$.
	\item If $\chi : S(F) \to \C^\times$ corresponds to $[z] \in H^1(W_F,\hat S)$, the inflation of $\chi$ to $S(F)_\infty$ corresponds to $[z] \in \tilde H^1(W_F,\hat S)$.
\end{enumerate}
\end{pro}
\begin{proof}
All statements hold by construction if we replace $\tilde H^1$ by $\tilde H^1_u$. Lemma \ref{lem:s-key} implies that $\tilde H^1(W_F,\hat S)$ is the amalgamation of $\tilde H^1_u(W_F,\hat S)$ and $H^1(W_F,\hat S)$ over $H^1_u(W_F,\hat S)$, and the proposition follows.
\end{proof}

\begin{pro} \label{pro:func-s}
\textit{(Functoriality of pseudo-universal covers):} 
\begin{enumerate}
	\item A homomorphism $f : S \to T$ of $F$-tori induces a continuous homomorphism $f_\infty : S(F)_\infty \to T(F)_\infty$ lifting $f: S(F) \to T(F)$ and restricting to $\tilde\pi_1(f): \tilde\pi_1(S) \to \tilde\pi_1(T)$.
	\item The correspondence between continuous characters of $S(F)_\infty$ and elements of $\tilde H^1(W_F,\hat S)$ is functorial in $S$.
\end{enumerate}
\end{pro}
\begin{proof}
The dual homomorphism $\hat f : \hat T \to \hat S$ induces a homomorphism of abstract groups $\tilde H^1_u(W_F,\hat T) \to \tilde H^1_u(W_F,\hat S)$. This homomorphism is continuous, because it restricts to a homomorphism $H^1_u(W_F,\hat T) \to H^1_u(W_F,\hat S)$ that is continuous, being the dual under Langlands duality of the continuous homomorphism $S(F) \to T(F)$. Dualizing again we obtain the desired continuous homomorphism $S(F)_\infty \to T(F)_\infty$, proving (1). To prove (2) we note that by construction the homomorphism $S(F)_\infty \to T(F)_\infty$ respects the homomorphism $\tilde H^1_u(W_F,\hat T) \to \tilde H^1_u(W_F,\hat S)$. On the other hand, the homomorphism $S(F) \to T(F)$ respects the homomorphism $H^1(W_F,\hat T) \to H^1(W_F,\hat S)$ by functoriality of Langlands duality. The argument of the proof of Proposition \ref{pro:llc-s} completes (2).
\end{proof}

Every element $t \in Z^2(\Gamma,\hat S)$ gives a continuous character $t : \tilde\pi_1(S) \to \mb{S}^1$ and we obtain the corresponding extension $S(F)_t$. When $t$ is of finite order $n$, we also have the extensions $S(F)_{t,n}$ and $S(F)_{t,\infty}$.

We define the $L$-group
\begin{equation} \label{eq:s-l}
^LS_t = \hat S \boxtimes_t \Gamma.
\end{equation}
Multiplication in this group is given by $s\boxtimes \sigma \cdot s' \boxtimes \sigma' = s\sigma(s')t(\sigma,\sigma')\boxtimes\sigma\sigma'$.

Propositions \ref{pro:llc-s} and \ref{pro:func-s} imply the following. 
\begin{cor}\ \\[-15pt] \label{cor:llc-s}
\begin{enumerate}
	\item The group of continuous characters of $S(F)_t$ is identified with the fiber product $\tilde H^1(W_F,\hat S) \times_{Z^2(\Gamma,\hat S)} \Z$, where the left map is the negative differential and the right map is $k \mapsto t^k$. When $t^n=1$, the analogous statement holds for $S(F)_{t,n}$, resp. $S(F)_{t,\infty}$ upon replacing $\Z$ by $\Z/n\Z$, resp. $\hat\Z$.
	\item Restriction along $S(F)_{t,n} \to S(F)_t$ resp. $S(F)_{t,n} \to S(F)_{t,\infty}$ is dual to the projection induced by the natural projection $\Z \to \Z/n\Z$ resp $\hat\Z \to \Z/n\Z$.
	\item The coset of genuine characters is the fiber over $1$ of the second projection, i.e. $\tilde H^1(W_F,\hat S)_t = \{ [c] \in \tilde H^1(W_F,\hat S)\,|\, (-\partial) c=t\}$. 
	\item The coset of genuine characters is also in bijection with the set of $\hat S$-conjugacy classes of $L$-homomorphisms $W_F \to {^LS_t}$. 
	\item Given a homomorphism $f : T \to S$ of $F$-tori the homomorphism $f_\infty : T(F)_\infty \to S(F)_\infty$ of Proposition \ref{pro:func-s}(1) induces a continuous homomorphism $f_{t,n} : T(F)_{t\circ f,n} \to S(F)_{t,n}$ that covers $f : T(F) \to S(F)$ and restricts to the identity on $\mu_n(\C)$. This homomorphism identifies $T(F)_{t \circ f,n}$ as the pull-back of $T(F) \to S(F) \from S(F)_{t,n}$. The analogous statement holds for the covers $S(F)_t$ and $S(F)_{t,\infty}$.
	\item The correspondence of (3) relates the homomorphism $f_{t,n}$ of (5) with the homomorphism 
	\[ ^Lf_t : \hat S \boxtimes_t \Gamma \to \hat T \boxtimes_{\hat f(t)} \Gamma,\qquad s\boxtimes \sigma \mapsto \hat f(s) \boxtimes \sigma. \]
\end{enumerate}
\end{cor}

\begin{cor} \label{cor:dc_pull}
Let $f:S \to T$ be a homomorphism of $F$-tori whose kernel $A$ is connected, and hence a torus. Let $t : \tilde\pi_1(T) \to \mu_n(\C)$. In the commutative diagram 
\[ \xymatrix{
1\ar[r]&A(F)\ar@{=}[d]\ar[r]&S(F)_{t,n}\ar[d]\ar[r]&T(F)_{t,n}\ar[d]\\
1\ar[r]&A(F)\ar[r]&S(F)\ar[r]&T(F)
}\]
that results from the pull-back property stated in Corollary \ref{cor:llc-s}(5), the natural map $A(F) \to S(F)_{t,n}$ is dual to the map $\hat S\boxtimes_t\Gamma \to \hat A\boxtimes_t\Gamma = \hat A \rtimes \Gamma$ where we have denoted by $t$ also the images of $t$ in $Z^2(\Gamma,\hat S)$ and $Z^2(\Gamma,\hat T)$, and the identification $\hat A\boxtimes_t\Gamma = \hat A \rtimes \Gamma$ arises from the fact that $t=1$ in $Z^2(\Gamma,\hat A)$. The analogous statement holds for $S(F)_t$ and $S(F)_{t,\infty}$.
\end{cor}
\begin{proof}
An equivalent way to obtain the left square is the following. We form the pull-back of $A(F) \to S(F) \from S(F)_{t,n}$. According to Corollary \ref{cor:llc-s}(5) this pull-back is identified with the cover obtained from the character $\tilde\pi_1(A) \to \C^\times$ that is the restriction of $t$. Since that character is trivial, this pull-back is canonically split. Composing the splitting $A(F) \to A(F)_{t,n}$ with the pull-back map $A(F)_{t,n} \to S(F)_{t,n}$ leads to the map $A(F) \to S(F)_{t,n}$. The map $A(F)_{t,n} \to S(F)_{t,n}$ is dual to the map $\hat S \boxtimes_t \Gamma \to \hat A \boxtimes_t \Gamma$ according to Corollary \ref{cor:llc-s}(6). 

On the other hand, the canonical splitting $A(F) \to A(F)_{t,n}$ can be interpreted as coming from the fact that $A(F)_{t,n}$ is the pushout of the trivial cover $A(F)_{t,1}$ of $A(F)$ by $\mu_1(\C)=\{1\}$ under the trivial embedding $\mu_1(\C) \to \mu_n(\C)$. According to Corollary \ref{cor:llc-s}(2) applied once for $n$ and once for $n=1$, this splitting is therefore dual to the identification $\hat S \boxtimes_t \Gamma = \hat A \rtimes \Gamma$.
\end{proof}

\begin{fct}
The $L$-group of the Baer sum $S(F)_{t_1,n} \oplus S(F)_{t_2,n}$ is naturally identified with the Baer sum of the $L$-groups $^LS_{t_1} \oplus {^LS}_{t_2}$.
\end{fct}

\begin{lem} \label{lem:dc_split}
The following data are equivalent
\begin{enumerate}
	\item A homomorphic section $s : S(F) \to S(F)_{t,n}$.
	\item A genuine character $\alpha : S(F)_{t,n}\to \mu_n(\C)$.
	\item A 1-cochain $c \in C^1(W_F,\hat S)$ satisfying $\partial c=t$ and $c^n \in B^1(W_F,\hat S)$.
\end{enumerate}
The same holds for the cover $S(F)_t$, where we replace $\mu_n(\C)$ by $\mb{S}^1$ and replace the condition $c^n \in B^1(W_F,\hat S)$ by $c \in C^1_u(W_F,\hat S)$. Finally, it also holds for the cover $S(F)_{t,\infty}$, where we replace $\mu_n(\C)$ by $\mu_\infty(\C)$ and impose the existence of $n \in \N$ such that $c^n \in B^1(W_F,\hat S)$.
\end{lem}
\begin{proof}
In all three cases the equivalence of (1) and (2) is given by the formula $\alpha(x)s(\bar x)=x$ for all $x \in S(F)_{t,n}$, where $\bar x \in S(F)$ is the image of $x$. For the equivalence of (2) and (3), the case of $S(F)_t$ is immediate from Proposition \ref{pro:llc-s}. The case of $S(F)_{t,n}$ follows from the multiplicativity of the local correspondence for $S(F)_\infty$, which implies that the order of a genuine character $\alpha : S(F)_{t,n} \to \C^\times$ is equal to the order of its parameter $c$ in the group $\tilde H^1(W_F,\hat S)$. Finally, in the case of $S(F)_{t,\infty}$ we use that a character $t : \tilde\pi_1(S) \to \mu_\infty(\C)$ has finite order $n$ due to Lemma \ref{lem:surjfin}. Therefore $S(F)_{t,\infty}$ is the push-out of $S(F)_{t,n}$ along the inclusion $\mu_n(\C) \to \mu_\infty(\C)$, hence genuine characters $S(F)_{t,\infty} \to \mu_\infty(\C)$ are in bijection with genuine characters $S(F)_{t,n} \to \mu_\infty(\C)$. The group $S(F)_{t,n}$, being a finite cover of $S(F)$, satisfies the assumption of Lemma \ref{lem:torchartor}, so every genuine character $S(F)_{t,n} \to \mu_\infty(\C)$ has finite order.
\end{proof}

We next turn to the analysis of the isomorphisms between the various covers $S(F)_{t,n}$. We begin with automorphisms.

\begin{lem} \label{lem:[n]}
Let $i$ and $n$ be positive integers.
\begin{enumerate}
 	\item The map $\tilde H^i(\Gamma,\hat S[n]) \to \tilde H^i(\Gamma,\hat S)[n]$ is surjective.
 	\item The map $\tilde H^1(\Gamma,\hat S)[n] \to \tilde H^1(W_F,\hat S)[n]$ is bijective.
 	\item If $S$ is induced, then the map $\tilde H^1(\Gamma,\hat S[n]) \to \tilde H^1(\Gamma,\hat S)[n]$ is bijective.
\end{enumerate} 
\end{lem} 
\begin{proof}
(1) Since $\hat S$ is a divisible group there exists a set-theoretic section $s : \hat S \to \hat S$ of the $n$-th power map. For the cohomology of $\Gamma$ we are considering $\hat S$ with the discrete topology, so this section is trivially continuous. Given $x \in C^i(\Gamma,\hat S)$ with the property $x^n=\partial y$ for $y \in C^{i-1}(\Gamma,\hat S)$ we take $z=s \circ y \in C^{i-1}(\Gamma,\hat S)$. Then $z^n=y$ and $x \cdot \partial z^{-1} \in C^i(\Gamma,\hat S[n])$ represents the same class as $x$. This proves (1).

(2) Consider $x \in C^1(W_F,\hat S)$ such that $x^n = \partial y$ for $y \in B^1(W_F,\hat S)=B^1(\Gamma,\hat S)$ and $\partial x \in Z^2(\Gamma,\hat S)$. Let $E/F$ be a finite Galois extension that splits $S$ and such that $\partial x$ and $y$ are inflated from $\Gamma_{E/F}$. The restriction of $x$ to $W_E$ lies in $Z^1(W_E,\hat S[n])=\tx{Hom}(E^\times,\hat S[n])$. Since $\hat S[n]$ is finite, the kernel of $x|_{E^\times}$ is a closed subgroup of $E^\times$ of finite index. If $E=\C$ the only possibility is $\C^\times$. If $E$ is non-archimedean, this subgroup contains the norms from a larger Galois extension of $F$ that contains $E$. Enlarging $E$ as necessary we may assume that $x|_{E^\times}$ is trivial. But this means that $x$ is inflated from $\Gamma_{E/F}$, i.e. $x \in \tilde H^1(\Gamma,\hat S)[n]$, proving surjectivity. 

To prove injectivity we note that the restriction map $C^1(\Gamma,\hat S) \to C^1(W_F,\hat S)$ is injective due to $W_F$ being dense in $\Gamma$, while $B^1(W_F,\hat S)=B^1(\Gamma,\hat S)$.

(3) Surjectivity has already been proved in (1), so it remains to prove injectivity. An element of $\tilde H^1(\Gamma,\hat S[n])$ with trivial image in $\tilde H^1(\Gamma,\hat S)[n]$ lies in $H^1(\Gamma,\hat S[n])$. Thus we need to prove the injectivity of $H^1(\Gamma,\hat S[n]) \to H^1(\Gamma,\hat S)$. We consider the exact sequence $1 \to \hat S[n] \to \hat S \to \hat S \to 1$, where the second map is the $n$-th power map. To this exact sequence we apply \cite[Corollary 2.3]{Kot84} to obtain the exact sequence $\pi_0(\hat S^\Gamma) \to H^1(\Gamma,\hat S[n]) \to H^1(\Gamma,\hat S)$. Since $S$ is induced, $\hat S^\Gamma$ is connected, and the claim follows.
\end{proof} 

For a central extension of abelian topological groups $1 \to A \to B \to C \to 1$ write $\tx{Aut}(A \to B \to C)$ for the group of continuous automorphisms of $B$ that preserve $A$ and induce the identity automorphisms on $A$ and $C$. The identity on $B$ is a distinguished element in this group, and the group of continuous homomorphisms $C \to A$ acts simply transitively on $\tx{Aut}(A \to B \to C)$ by multiplication. Taking the orbit of this action through the identity on $B$ provides an isomorphism $\tx{Aut}(A \to B \to C) = \tx{Hom}_\tx{cts}(C,A)$. From Lemmas \ref{lem:dc_split} and \ref{lem:[n]} we obtain

\begin{cor}\ \\[-15pt] \label{cor:s-auto}
\begin{enumerate}
	\item $\tx{Aut}(\mu_n(\C) \to S(F)_{t,n} \to S(F)) = H^1(\Gamma,\hat S)[n]$.
	\item $\tx{Aut}(\mu_\infty(\C) \to S(F)_{t,\infty} \to S(F)) = H^1(\Gamma,\hat S)[\infty]$.
	\item $\tx{Aut}(\mb{S}^1 \to S(F)_{t} \to S(F)) = H^1_u(W_F,\hat S)$.
\end{enumerate}
\end{cor}

To handle isomorphisms between different extensions, we define for $t \in Z^2(\Gamma,\hat S)$ 

\begin{eqnarray*}
C^1(\Gamma,\hat S)_{t}&=&\{ c \in C^1(\Gamma,\hat S)\,|\, (-\partial) c=t\},\\
\tilde H^1(\Gamma,\hat S)_{t}&=&C^1(\Gamma,\hat S)_{t}/B^1(\Gamma,\hat S),\\
C^1_u(W_F,\hat S)_{t}&=&\{ c \in C^1_u(W_F,\hat S)\,|\, (-\partial) c=t\},\\
\tilde H^1_u(W_F,\hat S)_{t}&=&C^1_u(W_F,\hat S)_{t}/B^1(\Gamma,\hat S).\\
\end{eqnarray*}

\begin{cns} \label{cns:iso-t}
Assume given $t,t' \in Z^2(\Gamma,\hat S)$ and $c \in C^1(\Gamma,\hat S)_{t/t'}$ and write $h=[c]$. We construct isomorphisms
\begin{equation} \label{eq:iso-t}
\xi_h : S(F)_{t,*} \to S(F)_{t',*}
\end{equation}
of extensions of $S(F)$, functorial in $S$, and multiplicative in $h$. Here $*$ is either void, or $n$, or $\infty$, and in the latter two cases we further assume given $n \in \N$ with $t^n=1$, $t'^n=1$, $h^n=1$.

The construction in all cases is the same, so we describe only the $\mb{S}^1$ case. Multiplication by $h^k$ on $\tilde H^1(W_F,\hat S)_{t'^k}$ for $k \in \Z$ and the identity on $\Z$ splice to an isomorphism 
\[ \tilde H^1_u(W_F,\hat S) \times_{Z^2(\Gamma,\hat S),t'} \Z \to \tilde H^1_u(W_F,\hat S) \times_{Z^2(\Gamma,\hat S), t} \Z. \] 
Pontryagin duality induces the isomorphism $\xi_h$. Given a homomorphism $f : T \to S$ we can map $t$ and $t'$ to $Z^2(\Gamma,\hat T)$ and $h$ to $\tilde H^1(\Gamma,\hat T)$ via $\hat f$. The compatibility of $\xi_h$ with $f$ is then immediate. The multiplicativity statement $\xi_{hh'} = \xi_h \circ \xi_{h'}$ is also immediate from the construction.

The corresponding isomorphism on the level of $L$-groups is given by 
\begin{equation} \label{eq:iso-l-s}
\hat S \boxtimes_{t'} \Gamma \to \hat S \boxtimes_{t} \Gamma,\qquad s \boxtimes \sigma \mapsto sc(\sigma) \boxtimes \sigma.	
\end{equation}

While the isomorphism $\xi_h$ only depends on the class $h$ of $c$, the isomorphism of $L$-groups depends on $c$ itself. But multiplying $c$ by a coboundary conjugates the isomorphism by an element of $\hat S$. Therefore the resulting bijections between the sets of conjugacy classes of $L$-parameters are the same.

Finally, we note that the same construction works with $C^1_u(W_F,\hat S)_{t/t'}$ in place of $C^1(\Gamma,\hat S)_{t/t'}$, but one has to use the Weil-form of the $L$-groups to write the corresponding dual isomorphism.
\end{cns}

\begin{lem}\ \\[-15pt] \label{lem:s-iso}
\begin{enumerate}
	\item The map sending $h \in \tilde H^1(\Gamma,\hat S)_{t/t'}[n]$ to the isomorphism of covers $\xi_h : S(F)_{t,n} \to S(F)_{t',n}$ induces a bijection between the set $\tilde H^1(\Gamma,\hat S)_{t/t'}[n]$ and the set of such isomorphisms. 
	\item The map sending $h \in \tilde H^1(\Gamma,\hat S)_{t/t'}$ to the isomorphism of covers $\xi_h : S(F)_{t,\infty} \to S(F)_{t',\infty}$ induces a bijection between the set $\tilde H^1(\Gamma,\hat S)_{t/t'}$ and the set of such isomorphisms. 
	\item The map sending $h \in \tilde H^1_u(W_F,\hat S)_{t/t'}$ to the isomorphism of covers $\xi_h : S(F)_{t} \to S(F)_{t'}$ induces a bijection between the set $\tilde H^1_u(W_F,\hat S)_{t/t'}$ and the set of such isomorphisms. 
\end{enumerate}
\end{lem}
\begin{proof}
We have already constructed a map from $\tilde H^1(\Gamma,\hat S)_{t/t'}[n]$ to the set of such isomorphisms. To show that it is bijective, note that the set of isomorphisms is a torsor under the group of automorphisms of $S(F)_{t,n}$. This group equals $H^1(\Gamma,\hat S)[n]$ by Corollary \ref{cor:s-auto}. On the other hand, $\tilde H^1(\Gamma,\hat S)_{t/t'}[n]$ is also a torsor under $H^1(\Gamma,\hat S)[n]$, and the two torsor structures agree by the multiplicativity of Construction \ref{cns:iso-t}. This proves (1). The arguments for (2) and (3) are analogous. For (2) note that if there exists a positive integer $n$ such that $t^n=1=t'^n$, then for any $c \in C^1(\Gamma,\hat S)_{t/t'}$ we have $c^n \in Z^1(\Gamma,\hat S)$, so the cohomology class of $c^n$ is torsion, so the cohomology class of $c$ is torsion (i.e. the corresponding element of $C^1/B^1$).
\end{proof}

\begin{cor}\ \\[-15pt] \label{cor:s-iso}
\begin{enumerate}
	\item The set of isomorphism classes of covers of $S(F)$ by $\mu_n(\C)$ that are obtained from the extension $S(F)_\infty$ via continuous characters $t : \tilde\pi_1(S) \to \mu_n(\C)$ is in bijection with $H^2(\Gamma,\hat S[n])$.
	\item The set of isomorphism classes of covers of $S(F)$ by $\mu_\infty(\C)$ that are obtained from the extension $S(F)_\infty$ via continuous characters $t : \tilde\pi_1(S) \to \mu_\infty(\C)$ is in bijection with $H^2(\Gamma,\hat S)$.
	\item All covers of $S(F)$ by $\mb{S}^1$ that are obtained from the extension $S(F)_\infty$ via continuous characters $t : \tilde\pi_1(S) \to \mb{S}^1$ are isomorphic.
\end{enumerate}
\end{cor}
\begin{proof}
(1) According to Lemma \ref{lem:s-iso}(1), given $t,t' \in Z^2(\Gamma,\hat S[n])$, the covers $S(F)_{t,n}$ and $S(F)_{t',n}$ are isomorphic if and only if there exists $h \in \tilde H^1(\Gamma,\hat S)[n]$ with $\partial h=t/t'$. According to Lemma \ref{lem:[n]}(1) the class $h$ is representable by an element of $C^1(\Gamma,\hat S[n])$.

(2) is immediate from Lemma \ref{lem:s-iso}(2).

(3) is immediate from Lemmas \ref{lem:s-iso}(3) and \ref{lem:s-key}.
\end{proof}

\begin{lem} \label{lem:uniqdiv}
Let $S$ be an $F$-torus.
\begin{enumerate}
	\item The natural map $\varinjlim_m H^i(\Gamma,\hat S[m]) \to H^i(\Gamma,\hat S)$ is an isomorphism for $i>1$, and surjective for $i=1$. 
	\item If $S$ is induced, then this map is an isomorphism also for $i=1$.
\end{enumerate}
\end{lem}
\begin{proof}
Let $\hat S[\infty]$ denote the torsion subgroup of $\hat S$, so that $\hat S[\infty]=\varinjlim_m \hat S[m]$. Then $\varinjlim_m H^i(\Gamma,\hat S[m])=H^i(\Gamma,\varinjlim_m \hat S[m])=H^i(\Gamma,\hat S[\infty])$. Since $\hat S$ is divisible, $\hat S/\hat S[\infty]$ is uniquely divisible, and hence cohomologically trivial. It follows that the map $H^i(\Gamma,\hat S[\infty]) \to H^i(\Gamma,\hat S)$ induced by the inclusion $\hat S[\infty] \to \hat S$ is an isomorphism for $i>1$ and surjective for $i=1$.

Assume now that $S$ is induced. Then Shapiro's lemma reduces the claim to the bijectivity of $H^1(\Gamma,\C^\times[\infty]) \to H^1(\Gamma,\C^\times)$, where $\Gamma$ acts trivially on $\C^\times$. Both of these groups are just the groups of continuous homomorphisms of $\Gamma$ into $\C^\times[\infty]$ and $\C^\times$, respectively, and the claim follows.
\end{proof}

The following will be used in the discussion of general reductive groups, and follows at once from Lemma \ref{lem:uniqdiv} and the 5-lemma.
\begin{cor} \label{cor:uniqdiv}
Let $T \to S$ be a complex of $F$-tori in degrees $0$ and $1$. 
\begin{enumerate}
	\item The natural map 
	$\varinjlim H^i(\Gamma,\hat S[m] \to \hat T[m]) \to H^i(\Gamma,\hat S \to \hat T)$ 
	is an isomorphism for $i>2$ and surjective for $i=2$. 
	\item If $T$ is induced, then this map is an isomorphism for $i=2$.
\end{enumerate}
\end{cor}

\begin{cor} \label{cor:s-minf}
Let $t,t' : \tilde\pi_1(S) \to \mu_n(\C)$ be continuous characters. The following are equivalent.
\begin{enumerate}
	\item There exists an integer multiple $m$ of $n$ such that the extensions $S(F)_{t,m}$ and $S(F)_{t',m}$ are isomorphic.
	\item The extensions $S(F)_{t,\infty}$ and $S(F)_{t',\infty}$ are isomorphic.
	\item $t$ and $t'$ become cohomologous in $H^2(\Gamma,\hat S)$.
\end{enumerate}
\end{cor}
\begin{proof}
According to Corollary \ref{cor:s-iso}, the extensions $S(F)_{t,m}$ and $S(F)_{t',m}$ are isomorphic if and only if $t$ and $t'$ become cohomologous in $H^2(\Gamma,\hat S[m])$. Thus, the existence of this $m$ is equivalent to the equality of the classes of $t$ and $t'$ in $\varinjlim_m H^2(\Gamma,\hat S[m])$. The equivalence of (1) and (3) follows from Lemma \ref{lem:uniqdiv}. The equivalence of (2) and (3) is Corollary \ref{cor:s-iso}(2).
\end{proof}

\begin{rem} \label{rem:isofuse}
Let $m$ be a multiple of $n$. Since the natural map $H^2(\Gamma,\hat S[n]) \to H^2(\Gamma,\hat S[m])$ need not be injective, two extensions $S(F)_{t,n}$ and $S(F)_{t',n}$ may be non-isomorphic, but the extensions $S(F)_{t,m}$ and $S(F)_{t',m}$ may become isomorphic.

When $F$ is non-archimedean then the connecting homomorphism in the exponential sequence leads to the isomorphism $H^2(\Gamma,\hat S) \to H^3(\Gamma,X_*(\hat S))$, and the latter is trivial because $F$ has strict cohomological dimension 2. Thus, any two covers become (non-canonically) isomorphic after pushing out to some $\mu_m$. On the other hand, the group $H^2(\Gamma,\hat S[n])$ need not vanish. For example, when $S$ is a 1-dimensional anisotropic torus, then $H^2(\Gamma,\hat S[2])$ has two elements. Thus, there do exist non-isomorphic extensions $S(F)_{t,n}$.

When $F=\R$, the group $H^2(\Gamma,\hat S)$ need not vanish, so there can be covers $S(F)_{t,n}$ and $S(F)_{t',n}$ that do not become isomorphic after any enlargement of $n$.
\end{rem}

\begin{rem}
The automorphism group of the cover $S(F)_{t,n}$ does not depend on $t$. Note further that it need not vanish. Therefore, the datum of an isomorphism class of extensions of $S(F)$ by $\mu_n(\C)$ coming from $S(F)_\infty$, i.e. the corresponding element of $H^2(\Gamma,\hat S[n])$, usually contains less information than an actual such extension, i.e. an element of $Z^2(\Gamma,\hat S[n])$ representing its cohomology class.
\end{rem}

\subsection{Covers of quasi-split connected reductive groups} \label{sub:covers_red}

Let $G$ be a quasi-split connected reductive $F$-group. In this subsection we will construct the extension \eqref{eq:baruni}, and obtain \eqref{eq:uni} via push-out along $\bar\pi(G) \to \tilde\pi(G)$. All the construction in this subsection will depend only on the quotient $Z^2(\Gamma,\hat T \to \hat T_\tx{ad})/B^1(\Gamma,\hat T_\tx{ad})$, so we can freely replace $\bar\pi_1(G)$ with the corresponding subgroup (the resulting extension can then be pushed out to the full $\bar\pi_1(G)$).

Let $T$ be the universal maximal torus of $G$. By functoriality we obtain a surjective (automatically open) homomorphism $\tilde\pi_1(T) \to \tilde\pi_1(G)$. It factors through a not necessarily surjective homomorphism $\tilde\pi_1(T) \to \bar\pi_1(G)$. We denote the push-out of the universal cover $T(F)_\infty$ along $\tilde\pi_1(T) \to \bar \pi_1(G)$ by $T(F)_\infty^G$.

\begin{cns} \label{cns:h-split}
We construct a splitting of $T(F)_\infty^G$ over $T_\tx{sc}(F)$, i.e. a homomorphism $s : T(F)_\tx{sc} \to T(F)_\infty^G$ that lifts the natural homomorphism $T_\tx{sc}(F) \to T(F)$.

We do the construction on the dual side. The Pontryagin dual of $T(F)_\infty^G$ is the pull-back of $\tilde H^1_u(W_F,\hat T) \to Z^2(\Gamma,\hat T) \from Z^2(\Gamma,\hat T \to \hat T_\tx{ad})/B^1(\Gamma,\hat T_\tx{ad})$. This pull-back is the subgroup of $\tilde H^1_u(W_F,\hat T) \times C^1(\Gamma,\hat T_\tx{ad})/B^1(\Gamma,\hat T_\tx{ad})$ consisting of those pairs $([c_1],[c_2])$ with $c_1 \in C^1_u(W_F,\hat T)$ and $c_2 \in C^1(\Gamma,\hat T_\tx{ad})$ that satisfy the properties $\partial c_1 \in Z^2(\Gamma,\hat T)$ and $(-\partial) \bar c_1=\partial c_2$. We define a homomorphism $\hat s$ from this subgroup to $H^1_u(W_F,\hat T_\tx{ad})$ to send $([c_1],[c_2])$ to $[\bar c_1 \cdot c_2]$. The restriction of $\hat s$ to the subgroup $H^1_u(W_F,\hat T)$ is the natural homomorphism $H^1_u(W_F,\hat T) \to H^1_u(W_F,\hat T_\tx{ad})$. 
\end{cns}

The choice of an $F$-Borel pair $(T_0,B_0)$ of $G$ provides an isomorphism $T \to T_0$, hence an embedding $T \to G$. Such an embedding will be called admissible. Since all $F$-Borel pairs are conjugate under $G_\tx{sc}(F)$, the same is true for all admissible embeddings $T \to G$.

\begin{lem} \label{lem:present}
Choose an admissible embedding $j : T \to G$ and let $T$ operate on $G_\tx{sc}$ by conjugation via this embedding. The sequence
\[ 1 \to T_\tx{sc}(F) \to G_\tx{sc}(F) \rtimes T(F) \to G(F) \to 1, \]
where the first map is $t_\tx{sc} \mapsto (j(t_\tx{sc}),\bar t_\tx{sc}^{-1})$ and the second is $(g_\tx{sc},t) \mapsto \bar g_\tx{sc} \cdot j(t)$, is exact.
\end{lem}
\begin{proof}
Exactness is immediate before taking $F$-points. The sequence remains exact after taking $F$-points since $T_\tx{sc}$ is an induced torus.
\end{proof}

\begin{cns} \label{cns:h-univ}
We now construct the extension \eqref{eq:baruni} as follows. Choose an admissible embedding $T \to G$. Let $T(F)_\infty^G$ act on $G_\tx{sc}(F)$ through the map $T(F)_\infty^G \to T(F)$. Then we obtain the diagram with exact rows and columns
\[\xymatrix{
&&1\ar[d]&1\ar[d]\\
&&T_\tx{sc}(F)\ar[d]\ar@{=}[r]&T_\tx{sc}(F)\ar[d]\\
1\ar[r]&\bar\pi_1(G)\ar[r]\ar@{=}[d]& G_\tx{sc}(F) \rtimes T(F)_\infty^G\ar[r]\ar[d]&G_\tx{sc}(F) \rtimes T(F)\ar[r]\ar[d]&1\\
1\ar[r]&\bar\pi_1(G)\ar[r]& G(F)_\infty\ar[r]\ar[d]&G(F)\ar[r]\ar[d]&1\\
&&1&1
} \]
Here the right column is Lemma \ref{lem:present}. The map $T_\tx{sc}(F) \to G_\tx{sc}(F) \rtimes T(F)_\infty^G$ is given by $t_\tx{sc} \mapsto (j(t_\tx{sc}),s(t_\tx{sc})^{-1})$ and $G(F)_\infty$ is defined as its cokernel.

The splitting $s : T_\tx{sc}(F) \to T(F)_\infty^G$ extends to a splitting $G_\tx{sc}(F) \to G(F)_\infty$ in the evident way -- by mapping $g_\tx{sc} \in G_\tx{sc}(F)$ to $(g,1) \in G_\tx{sc}(F) \rtimes T(F)_\infty^G$ and then projecting onto $G(F)_\infty$.

Since any two admissible embeddings $j : T \to G$ are conjugate under $G_\tx{sc}(F)$, the extension $G(F)_\infty \to G(F)$ is well-defined up to conjugation by $G_\tx{sc}(F)$. More precisely, if $j_i,i=1,2$ are two admissible embeddings leading to the two extensions $G(F)_{\infty,i} \to G(F)$ there exists $h \in G_\tx{sc}(F)$ such that
\[ \xymatrix{
	G(F)_{\infty,1}\ar[r]^{\tx{Ad}(h)}\ar[d]&G(F)_{\infty,2}\ar[d]\\
	G(F)\ar[r]^{\tx{Ad}(h)}&G(F)
}
\]
\end{cns} 

\begin{rem} \label{rem:func-h}
An isomorphism of $F$-groups $f : G_1 \to G_2$ induces an isomorphism $f_\infty : G_1(F)_\infty \to G_2(F)_\infty$ that lifts the isomorphism $f : G_1(F) \to G_2(F)$ and restricts to the isomorphism $\bar\pi_1(f)  : \bar\pi_1(G_1) \to \bar\pi_2(G_2)$. Indeed, $f$ induces an isomorphism $f : T_1 \to T_2$ between the universal maximal tori, hence by  Proposition \ref{pro:func-s} an isomorphism $f_\infty : T_1(F)_\infty \to T_2(F)_\infty$ of covers. This isomorphism descends to an isomorphism $T_1(F)_\infty^{G_1} \to T_2(F)_\infty^{G_2}$ between the pushouts along $\tilde\pi_1(T_i) \to \bar\pi_1(G_i)$. The isomorphism $f_\infty : T_1(F)_\infty^{G_1} \to T_2(F)_\infty^{G_2}$ is compatible with the isomorphism $f : T_{1,\tx{sc}}(F) \to T_{2,\tx{sc}}(F)$ and the splittings of Construction \ref{cns:h-split}, therefore induces the desired isomorphism $f_\infty : G_1(F)_\infty \to G_2(F)_\infty$.
\end{rem}

Given $t \in Z^2(\Gamma,\hat T \to \hat T_\tx{ad})/B^1(\Gamma,\hat T_\tx{ad})$ we can form $G(F)_t$ via pushout of of \eqref{eq:baruni} along $t : \bar\pi_1(G) \to \mb{S}^1$. If in addition $t^n=1$ for some and $n \in \N$ we can form $G(F)_{t,n}$ and $G(F)_{t,\infty}$ via the push-out along $t : \bar\pi_1(G) \to \mu_n(\C)$ or $t : \bar\pi_1(G) \to \mu_\infty(\C)$.

\begin{fct} \label{fct:gtn}
Write $t=(z,[c])$ with $z \in Z^2(\Gamma,\hat T)$ and $[c] \in C^1(\Gamma,\hat T_\tx{ad})/B^1(\Gamma,\hat T_\tx{ad})$. For $*$ being void, $n$, or $\infty$, the following statements hold.
\begin{enumerate}
	\item The inclusion $T(F)_\infty^G \to G(F)_\infty$ composed with the projection $G(F)_\infty \to G(F)_{t,*}$ factors through $T(F)_{z,*}$ and provides an isomorphism between $T(F)_{z,*}$ and the pull back of $T(F) \to G(F) \from G(F)_{t,*}$.
	\item The resulting map $T(F)_{z,*} \to G(F)_{t,*}$, together with the section $G_\tx{sc}(F) \to G(F)_{t,*}$, lead to the exact sequence
	\[ 1 \to T_\tx{sc}(F) \to G_\tx{sc}(F) \rtimes T(F)_{z,*} \to G(F)_{t,*} \to 1. \]
\end{enumerate}
\end{fct}

\begin{rem}
The section $T_\tx{sc}(F) \to T(F)_{z,*}$ that is the second component of the first map in Fact \ref{fct:gtn}(2) equals the composition of the section $T_\tx{sc}(F) \to T(F)_\infty^G$ of Construction \ref{cns:h-split} and the projection $T(F)_\infty^G \to T(F)_{z,*}$. Unwinding the definitions, we note that this section depends on the second component $[c]$ of $t=(z,[c])$. This helps clarify the roles of the two components of an element of $Z^2(\Gamma,\hat T \to \hat T_\tx{ad})/B^1(\Gamma,\hat T_\tx{ad})$: the first component leads to a cover of $T(F)$, and the second component leads to a section of that cover over $T_\tx{sc}(F)$.
\end{rem}

We will now study the isomorphisms between the extensions $G(F)_{t,*}$. Let $t \in Z^2(\Gamma,\hat T \to \hat  T_\tx{ad})/B^1(\Gamma,\hat T_\tx{ad})$ and write $t=(z,[c])$. Define
\begin{eqnarray*}
\tilde H^1(\Gamma, \hat T)_{t}&=&\{ h \in \tilde H^1(\Gamma, \hat T)\,|\, \partial h = z^{-1}, \bar h=[c]^{-1}\},\\
\tilde H^1_u(W_F, \hat T)_{t}&=&\{ h \in \tilde H^1_u(W_F, \hat T)\,|\, \partial h = z^{-1}, \bar h=[c]^{-1}\},\\
\end{eqnarray*}
Note that $\tilde H^1(\Gamma, \hat T)_t \subset \tilde H^1(\Gamma, \hat T)_z$, cf. Construction \ref{cns:iso-t}. 

\begin{cns} \label{cns:iso-h}
Let $t,t' \in Z^2(\Gamma,\hat T \to \hat  T_\tx{ad})/B^1(\Gamma,\hat T_\tx{ad})$ and write $t=(z,[c])$ and $t'=(z',[c'])$. For any $h \in \tilde H^1_u(W_F, \hat T)_{t/t'}$ we construct an isomorphism
\begin{equation} \label{eq:iso-h}
\xi_h : G(F)_{t,*} \to G(F)_{t',*}
\end{equation}
of extensions of $G(F)$, where again $*$ is void, $n$, or $\infty$, and in the last two cases we assume that $n$ is given such that $t^n=1$, $t'^n=1$, and $h^n=1$.

The element $h$ and the identity $(-\partial) h=z/z'$ lead, via Construction \ref{cns:iso-t}, to the isomorphism $\xi_h : T(F)_{z,*} \to T(F)_{z',*}$. The identity $\bar h=c'/c$ implies that this isomorphism respects the sections of $T_\tx{sc}(F)$ into these two covers, as one sees by inspecting Constructions \ref{cns:iso-t} and \ref{cns:h-split}. According to Fact \ref{fct:gtn}(2) the extension $G(F)_{t,*}$ is equal to the cokernel of $T_\tx{sc} \to G_\tx{sc}(F) \rtimes T(F)_{t,*}$, and the analogous statement holds for $G(F)_{t',*}$. Therefore the identity on $G_\tx{sc}$ and $\xi_h$ glue to an isomorphism $\xi_h : G(F)_{t,*} \to G(F)_{t',*}$ as required.
\end{cns}

\begin{lem} \label{lem:kneser-tits}
Any isomorphism $\xi : G(F)_{t,*} \to G(F)_{t',*}$ of extensions automatically respect the sections over $G_\tx{sc}(F)$.
\end{lem}
\begin{proof}
Let $s : G_\tx{sc}(F) \to G(F)_{t,*}$ and $s' : G_\tx{sc}(F) \to G(F)_{t',*}$ be the two natural sections. Then $\xi \circ s$ and $s'$ are two sections $G_\tx{sc}(F) \to G(F)_{t',*}$, therefore there exists $\chi : G_\tx{sc}(F) \to \mb{S}^1$ such that $\xi\circ s = \chi \cdot s'$. But $\chi$ must be trivial due to the Kneser--Tits conjecture, which is obvious for $F=\C$, proved by Cartan for $F=\R$ and by Platonov for non-archimedean $F$, cf. \cite[\S7.2]{PR94}.
\end{proof}

\begin{lem} \label{lem:h-iso}
Let $t,t' \in Z^2(\Gamma,\hat T \to \hat T_\tx{ad})/B^1(\Gamma,\hat T_\tx{ad})$. Every isomorphism of extensions $G(F)_{t,*} \to G(F)_{t',*}$ arises from Construction \ref{cns:iso-h} via a unique 
\[ h \in 
\begin{cases}
\tilde H^1_u(W_F,\hat T)_{t/t'},&*=\tx{void}\\
\tilde H^1(\Gamma,\hat T)_{t/t'}[n],&*=n\\
\tilde H^1(\Gamma,\hat T)_{t/t'},&*=\infty
\end{cases}
 \]
In the latter two cases we are assuming $t^n=1=t'^n$.
\end{lem}
\begin{proof}
Consider an isomorphism of extensions $\xi : G(F)_{t,*} \to G(F)_{t',*}$. According to Fact \ref{fct:gtn}(1), $\xi$ restricts to an isomorphism $\xi_T : T(F)_{z,*} \to T(F)_{z',*}$, where we have again written $t=(z,[c])$ and $t'=(z',[c'])$. According to Lemma \ref{lem:s-iso} there exists $h$ so that $\xi_T=\xi_h$ and $h$ satisfies all required properties, except possibly $\bar h=[c']/[c]$. 

According to Lemma \ref{lem:kneser-tits}, $\xi$ respects the sections over $G_\tx{sc}(F)$. This implies two things. First, we obtain the commutative diagram
\[ \xymatrix{
	G_\tx{sc}(F) \rtimes T(F)_{z,*} \ar[r]^{\tx{id} \times \xi_h}\ar[d]&G_\tx{sc}(F) \rtimes T(F)_{z',*}\ar[d]\\
	G(F)_{t,*}\ar[r]^\xi&G(F)_{t',*}
}
\]
Second, the isomorphism $\xi_h$ respects the sections of $T(F)_{z,*}$ and $T(F)_{z',*}$ over $T_\tx{sc}(F)$. This implies the desired equality $\bar h=[c']/[c]$. We conclude that $\xi$ indeed comes from Construction \ref{cns:iso-h}. The uniqueness statement follows from the uniqueness statement in Lemma \ref{lem:s-iso}.
\end{proof}

\begin{lem} \label{lem:zh}
The natural maps $H^1(\Gamma,Z(\hat G)) \to \tx{ker}(H^1(\Gamma,\hat T) \to H^1(\Gamma,\hat T_\tx{ad}))$ and $\tilde H^1(\Gamma,Z(\hat G)) \to \tx{ker}(\tilde H^1(\Gamma,\hat T) \to \tilde H^1(\Gamma,\hat T_\tx{ad}))$ are isomorphisms.
\end{lem}
\begin{proof}
(1) We have the exact sequence 
\[ \pi_0((\hat T_\tx{ad})^\Gamma) \to H^1(\Gamma,Z(\hat G)) \to H^1(\Gamma,\hat T) \to H^1(\Gamma,\hat T_\tx{ad}) \]
of \cite[Corollary 2.3]{Kot84}. Since $\hat T_\tx{ad}$ is an induced torus, $(\hat T_\tx{ad})^\Gamma$ is connected.	

(2) Let $c \in C^1(\Gamma,\hat T)$ become cohomologically trivial in $C^1(\Gamma,\hat T_\tx{ad})$. After modification by a coboundary we have $c \in C^1(\Gamma,Z(\hat G))$. If $c \in C^1(\Gamma,Z(\hat G))$ is cohomologically trivial in $C^1(\Gamma,\hat T)$, then it lies in $Z^1(\Gamma,\hat T) \cap C^1(\Gamma,Z(\hat G)) = Z^1(\Gamma,Z(\hat G))$ and (1) implies that it is in fact an element of $B^1(\Gamma,Z(\hat G))$.
\end{proof}

\begin{cor}\ \\[-15pt] \label{cor:h-iso}
\begin{enumerate}
	\item $(Z^2(\Gamma,\hat T \to \hat T_\tx{ad})/B^1(\Gamma,\hat T_\tx{ad}))[n] = Z^2(\Gamma,\hat T[n] \to \hat T_\tx{ad}[n])/B^1(\Gamma,\hat T_\tx{ad}[n])$.
	\item The set of isomorphism classes of covers of $G(F)$ by $\mu_n(\C)$ that are obtained from $G(F)_\infty$ via continuous characters $t : \bar\pi_1(G) \to \mu_n(\C)$ is in bijection with $H^2(\Gamma,\hat T[n] \to \hat T_\tx{ad}[n])$.
	\item The group of automorphisms of such a cover is $H^1(\Gamma,Z(\hat G))[n]$.
	\item The set of isomorphism classes of covers of $G(F)$ by $\mu_\infty(\C)$ that are obtained from $G(F)_\infty$ via continuous characters $t : \bar\pi_1(G) \to \mu_\infty(\C)$ is in bijection with $H^2(\Gamma,\hat T \to \hat T_\tx{ad})=H^2(\Gamma,Z(\hat G))$.
\end{enumerate}
\end{cor}
\begin{proof}
(1) follows at once from Lemma \ref{lem:[n]}(3).

(2) Let $t=(z,[c]), t'=(z',[c']) \in Z^2(\Gamma,\hat T[n] \to \hat T_\tx{ad}[n])/B^1(\Gamma,\hat T_\tx{ad}[n])$. According to Lemma \ref{lem:h-iso} the corresponding covers are isomorphic if and only if there exists $h \in \tilde H^1(\Gamma,\hat T)_{t/t'}[n]$. Lemma \ref{lem:[n]}(1) implies that $h$ can be represented by an element $x \in C^1(\Gamma,\hat T[n])$. This element has the properties $\partial x=z'/z$ and $[\bar x]=[c']/[c]$. The latter equation holds in $C^1(\Gamma,\hat T_\tx{ad})/B^1(\Gamma,\hat T_\tx{ad})$, but all three elements $x,c,c'$ belong to $C^1(\Gamma,\hat T_\tx{ad}[n])$, therefore Lemma \ref{lem:[n]}(3) implies that this equation holds already in $C^1(\Gamma,\hat T_\tx{ad}[n])/B^1(\Gamma,\hat T_\tx{ad}[n])$.

(3) According to Lemma \ref{lem:h-iso} the group of isomorphism of the cover associated to $t=(z,[c])$ is the group of $h \in \tx{ker}(H^1(\Gamma,\hat T)[n] \to H^1(\Gamma,\hat T_\tx{ad})[n]$, which, according to Lemma \ref{lem:zh}(1), equals $H^1(\Gamma,Z(\hat G))[n]$.

(4) Follows at once from Lemma \ref{lem:h-iso}.
\end{proof}

\begin{cor} \label{cor:h-minf}
Let $t,t' : \bar\pi_1(G) \to \mu_n(\C)$. The following statements are equivalent.
\begin{enumerate}
	\item There exists a multiple $m$ of $n$ such that the covers $G(F)_{t,m}$ and $G(F)_{t',m}$ become isomorphic.
	\item The covers $G(F)_{t,\infty}$ and $G(F)_{t',\infty}$ are isomorphic.
	\item The classes of $t$ and $t'$ in $H^2(\Gamma,\hat T \to \hat T_\tx{ad})=H^2(\Gamma,Z(\hat G))$ are equal.
\end{enumerate}
\end{cor}
\begin{proof}
According to Corollary \ref{cor:h-iso}(2), (1) holds if and only if the classes of $t$ and $t'$ in $\varinjlim_m H^2(\Gamma,\hat T[m] \to \hat T_\tx{ad}[m])$ become equal. According to Corollary \ref{cor:uniqdiv} the natural map from this group to $H^2(\Gamma,\hat T \to \hat T_\tx{ad})$ is an isomorphism. This shows the equivalence of (1) and (3). The equivalence of (2) and (3) is Corollary \ref{cor:h-iso}(4).
\end{proof}

\begin{rem}
We may also be interested in covers of $G(F)$ obtained from continuous characters $t : \tilde\pi_1(G) \to \mu_n(\C)$. These are the elements of the subgroup $Z^2(\Gamma,Z(\hat G)[n])$ of $Z^2(\Gamma,\hat T[n] \to \hat T_\tx{ad}[n])/B^1(\Gamma,\hat T_\tx{ad}[n])$. Lemmas \ref{lem:h-iso} and \ref{lem:zh} show that the covers for $t,t' \in Z^2(\Gamma,Z(\hat G)[n])$ are isomorphic if and only if there exists $h \in \tilde H^1(\Gamma,Z(\hat G))[n]$ with $\partial h=t'/t$. Thus the set of equivalence classes of such covers is a certain quotient of $H^2(\Gamma,Z(\hat G)[n])$ that maps to $H^2(\Gamma,Z(\hat G))[n]$. Unfortunately, due to the possibly non-trivial $n$-torsion of $\pi_0(Z(\hat G))$, the group $Z(\hat G)$ need not be $n$-divisible, so this quotient need not equal $H^2(\Gamma,Z(\hat G)[n])$.	
\end{rem}

\begin{cor}
Let $t : \bar\pi_1(G) \to \mu_n(\C)$. There exists a multiple $m$ of $n$ and $t' : \tilde\pi_1(G) \to \mu_m(\C)$ such that the covers $G(F)_{t,m}$ and $G(F)_{t',m}$ are isomorphic.
\end{cor}
\begin{proof}
Immediate from Corollary \ref{cor:h-minf} and the isomorphism $H^2(\Gamma,Z(\hat G)) \to H^2(\Gamma,\hat T \to \hat T_\tx{ad})$ induced by the inclusion $Z(\hat G) \to \hat T$.
\end{proof}

\subsection{Descent to the abelianization}

In this subsection we will discuss the following question: Do the covers \eqref{eq:uni} or \eqref{eq:baruni} descend to a cover of $\tx{cok}(G_\tx{sc}(F) \to G(F))$? The latter group is abelian, and in fact closely related to the hypercohomology group $H^1(F,G_\tx{sc} \to G)$. This makes such a descent statement useful. 

It will turn out that such a descent is not always possible. We will give a criterion for when this happens, and use it to obtain the compatibility of the covers \eqref{eq:uni} for $G$ and its maximal tori.

Write $K=\tx{ker}(G_\tx{sc} \to G)$. Then $K$ is a finite multiplicative group and the section $s : G_\tx{sc}(F) \to G(F)_\infty$ restricts to a group homomorphism $s : K(F) \to \bar\pi_1(G)$. More precisely:

\begin{fct} \label{fct:k-cart}
The square
\[ \xymatrix{
K(F)\ar[r]\ar[d]&G_\tx{sc}(F)\ar[d]\\
\bar\pi_1(G)\ar[r]&G(F)_\infty
}\]
is Cartesian.
\end{fct}

Write $C=\tx{cok}(G_\tx{sc}(F) \to G(F))$.

\begin{pro} \label{pro:gsc_abelian}
The following statements are equivalent.
\begin{enumerate}
	\item The homomorphism $s : K(F) \to \bar\pi_1(G)$ is trivial.
	\item $G(F)_\infty$ is the pullback of an extension of $C$ by $\bar\pi_1(G)$.
	\item There exists a genuine character $G(F)_t \to \C^\times$ for any $t : \bar\pi_1(G) \to \mb{S}^1$.
\end{enumerate}
Moreover, when these statements hold, then $C' := \tx{cok}(s : G_\tx{sc}(F) \to G(F)_\infty)$ is an extension of $C$ by $\bar\pi_1(G)$ and $G(F)_\infty$ is the pull-back of this extension along the projection map $G(F)_\infty \to C'$.
\end{pro}
\begin{proof}
$(1) \Rightarrow (2)$: The composed map $G(F)_\infty \to G(F) \to C$ descends to a surjective map $C' \to C$ whose kernel is equal to the image of $\bar\pi_1(G) \to G(F)_\infty \to C'$, hence to  $\tx{cok}(s : K(F) \to \bar\pi(G))$ by Fact \ref{fct:k-cart}.  Thus $C'$ is an extension of $C$ by $\bar\pi_1(G)$ if and only if $s|_{K(F)}$ is trivial. In that case, it is clear that the natural map $G(F)_\infty \to C'$ realizes $G(F)_\infty$ as the pull-back of $C' \to C \from G(F)$.

$(2) \Rightarrow (1)$: Conversely assume that $G(F)_\infty$ is a pull-back of an extension $C''$ of $C$ by $\bar\pi_1(G)$. Then the natural map $G_\tx{sc}(F) \to G(F)$ together with the trivial map $G_\tx{sc}(F) \to C''$ provide a section $s' : G_\tx{sc}(F) \to G(F)_\infty$. It differs from the section $s$ by a homomorphism $\alpha : G_\tx{sc}(F) \to \bar\pi_1(G)$. The composition of this homomorphism with any character of $\bar\pi_1(G)$ is a character of $G_\tx{sc}(F)$, hence trivial by the Kneser--Tits conjecture \cite[\S7.2]{PR94}. Therefore $s'=s$. But by construction $s'$ kills $K(F)$.

$(2) \Rightarrow (3)$: $G(F)_t$ is the pull-back of the pushout $C'_t$ of $C'$ under $t$. By Pontryagin duality $t$ extends to a character $\chi : C'_t \to \mb{S}^1$, which by construction is genuine. The pull-back of $\chi$ to $G(F)_t$ is a genuine character.

$(3) \Rightarrow (1)$. According to the Kneser--Tits conjecture, the composition of any genuine character $\chi : G(F)_t \to \C^\times$ with the section $s : G_\tx{sc}(F) \to G(F)_t$ is trivial. Therefore, the homomorphism $s : K(F) \to \bar\pi_1(G)$ composes trivially with any character of $\bar\pi_1(G)$ and must therefore be trivial.
\end{proof}

\begin{cor} \label{cor:gsc_abelian}
If the derived subgroup of $G$ is simply connected, then $G(F)_\infty$ is the pull-back of the cover $C' \to C$.
\end{cor}

\begin{rem} \label{rem:gsc_abelian}
According to Lemma \ref{lem:present} and Construction \ref{cns:h-univ} we have the identities
\begin{eqnarray*}
C'&=&\tx{cok}(G_\tx{sc}(F) \to G(F)_\infty) = \tx{cok}(T_\tx{sc}(F) \to T(F)_\infty^G),\\
C&=&\tx{cok}(G_\tx{sc}(F) \to G(F))=\tx{cok}(T_\tx{sc}(F) \to T(F)).
\end{eqnarray*}
If we assume that the derived subgroup of $G$ is simply connected and define $D=\tx{cok}(T_\tx{sc} \to T)$, then $C=D(F)$ and $\tilde\pi_1(D)=\tilde\pi_1(G)$. Functoriality of the pseudo-universal cover for a torus (Proposition \ref{pro:func-s}) provides a homomorphism $T(F)_\infty \to D(F)_\infty$ that descends to a homomorphism $T(F)_\infty^G \to D(F)_\infty$, and in turn provides a homomorphism $G(F)_\infty \to D(F)_\infty$. These homomorphisms identify the push-out of $C'$ along $\bar\pi_1(G) \to \tilde\pi_1(G)$ with $D(F)_\infty$, and realize the extension \eqref{eq:uni} for $G$ as the pull-back of the extension \eqref{eq:uni} for the torus $D$ along the natural map $G(F) \to D(F)$.
\end{rem}

\begin{pro} \label{pro:section-k-z}
The Pontryagin dual of the finite abelian group $K(F)$ is $H^2(W_F,Z(\hat G))$ and the Pontryagin dual of $s: K(F) \to \bar\pi_1(G)$ is the composition of the natural map $Z^2(\Gamma,\hat T \to \hat T_\tx{ad})/B^1(\Gamma,\hat T_\tx{ad}) \to H^2(\Gamma,\hat T \to \hat T_\tx{ad})$ with the identification $H^2(\Gamma,\hat T \to \hat T_\tx{ad})=H^2(\Gamma, Z(\hat G))$ provided by the quasi-isomorphism $[Z(\hat G) \to 1] \to [\hat T \to \hat T_\tx{ad}]$ and the restriction map $H^2(\Gamma,Z(\hat G)) \to H^2(W_F,Z(\hat G))$.
\end{pro}
\begin{proof} 
We have $K(F)=H^0(F,T_\tx{sc} \to T)$, where the complex $T_\tx{sc} \to T$ is placed in degrees $0$ and $1$. According to \cite[Lemma A.3.A]{KS99} the Pontryagin dual of this group is identified with $H^2(W_F,\hat T \to \hat T_\tx{ad})$, with $\hat T \to \hat T_\tx{ad}$ again placed in degrees $0$ and $1$. The latter complex id quasi-isomorphic to $Z(\hat G)$ placed in degree $0$, hence the first claim.

The map $s : K(F) \to \bar\pi_1(G)$ can be obtained by first restricting $s$ to $T_\tx{sc}$, where it equals the section $s : T_\tx{sc}(F) \to T(F)_\infty^G$ of Construction \ref{cns:h-split}, and then further restricting to $K(F)$. The claim follows by examining Construction \ref{cns:h-split}.
\end{proof}

\begin{rem} %
It follows from Proposition \ref{pro:section-k-z} and Lemma \ref{lem:cc_surj_na} that the homomorphism $s : K(F) \to \bar\pi_1(G)$ is injective when $F$ is non-archimedean. Therefore, when $F$ is non-archimedean and of characteristic zero, the statement of Corollary \ref{cor:gsc_abelian} can be strengthened to an if-and-only-if statement.
 
As discussed in Remark \ref{rem:cc_nsurj_a}, this need not be the case when $F$ is archimedean.
\end{rem}

The following result is obtained by the same analysis. We will give a brief indication.

\begin{pro} \label{pro:gsc_abelian1}
Let $t : \bar\pi_1(G) \to \mu_n(\C)$. The following statements are equivalent.
\begin{enumerate}
	\item The class of $t$ in $H^2(W_F,\hat T \to \hat T_\tx{ad})=H^2(W_F,Z(\hat G))$ is trivial.
	\item The extension $G(F)_{t,n}$ is a pull-back from $\tx{cok}(G_\tx{sc}(F) \to G(F))$.
	\item There exists a genuine character $G(F)_{t,n} \to \C^\times$.
\end{enumerate}
\end{pro}
\begin{proof}
$(1) \Rightarrow (2)$: According to Proposition \ref{pro:section-k-z} the homomorphism $t\circ s : K(F) \to \mu_n(\C)$ is trivial, hence $\tx{cok}(T_\tx{sc}(F) \to T(F)_{z,n}) \to \tx{cok}(T_\tx{sc}(F) \to T(F))$ is a cover with kernel $\mu_n(\C)$, and $G(F)_{t,n}$ is the pull-back of this cover.

$(2) \Rightarrow (3)$: The character $t : \bar\pi_1(G) \to \mu_n(\C)$ can be extended by Pontryagin duality to a genuine character of the cover of $\tx{cok}(G_\tx{sc}(F) \to G(F))$ and then pulled back to $G(F)_{t,n}$.

$(3) \Rightarrow (1)$: The composition of the genuine character with $s : G_\tx{sc}(F) \to G(F)_{t,n}$ is trivial, hence the character $K(F) \to \mu_n(\C)$ obtained from $t \circ s$ is trivial. Apply Proposition \ref{pro:section-k-z}.
\end{proof}

\begin{rem}
An example where the class of $t$ in $H^2(W_F,Z(\hat G))$ vanishes is when $G$ is a twisted Levi subgroup of a connected reductive group and $t$ is the corresponding Tits cocycle: as shown in \cite[Lemma 6.6]{KalDC}, this cocycle takes values in $Z(\hat G)^\circ$, and Proposition \ref{pro:cc_rajan} implies that $H^2(W_F,Z(\hat G)^\circ)=0$. 

On the other hand, when $Z(\hat G)$ is finite and $F$ is non-archimedean then Lemma \ref{lem:cc_surj_na} implies that the natural map $H^2(\Gamma,Z(\hat G)) \to H^2(W_F,Z(\hat G))$ is an isomorphism. Therefore, the class of $t$ in $H^2(W_F,Z(\hat G))$ is trivial if and only if its class in $H^2(\Gamma,\hat T \to \hat T_\tx{ad})=H^2(\Gamma,Z(\hat G))$ is trivial. According to Corollary \ref{cor:h-minf} this implies that the cover $G(F)_{t,m}$ is split for some multiple $m$ of $n$.
\end{rem}

\subsection{Elementary categorical observations} \label{sub:elementary}

A number of statements involving $\tilde\pi_1(G)$ can be extended to $\bar\pi_1(G)$ in a formal way. Rather than performing the same construction multiple times, we formalize it here in an abstract elementary language, so that it can be easily applied when needed.

Let $\mc{C}$ be a category. For our purposes we may and will assume that $\mc{C}$ is a small groupoid. Recall that a subcategory $\mc{C}' \subset \mc{C}$ is called \emph{full}, if for any two objects $x,y$ of $\mc{C}'$ we have $\tx{Hom}_{\mc{C}'}(x,y)=\tx{Hom}_\mc{C}(x,y)$. Recall further that $\mc{C}'$ is called \emph{essentially wide} if every object of $\mc{C}$ is isomorphic to an object of $\mc{C}'$. Thus, the inclusion functor $\mc{C}' \to \mc{C}$ is an equivalence of categories. In this situation, for any $x \in \mc{C}$ and $y \in \mc{C'}$, the set $\tx{Isom}_{\mc{C}}(x,y)$ is a torsor under $\tx{Aut}_{\mc{C}'}(y)$, and for any $x$ there exists $y$ such that this torsor is non-empty.

Given further a category $\mc{D}$, any functor $\mc{F}' : \mc{C'} \to \mc{D}$ extends uniquely to a functor $\mc{F} : \mc{C} \to \mc{D}$. In concise language, given $x \in \mc{C}$ we consider the comma category (i.e. the fiber product) $\mc{C}'_x$ of $\{x\} \to \mc{C} \from \mc{C'}$, where $\{x\}$ is the sub category of $\mc{C}$ with one object $x$ and identity morphism, and both functors are the natural embeddings. Then $\mc{F}'$ provides an obvious functor $\mc{F}'_x : \mc{C}'_x \to \mc{D}$ and we define $\mc{F}(x)=\lim \mc{F}'_x$. Note that the limit automatically exists, because it is represented by $\mc{F}'(x')$ for any $(x',f) \in \mc{C}'_x$.

Spelled out, for $x \in \mc{C}$ one considers the following compatible system of objects in $\mc{C}'$ and isomorphisms. The indexing set is the set of isomorphisms in $\mc{C}$ with source $x$ and target belonging to $\mc{C}'$. The objects in the system are the targets of those isomorphisms. Given two indices $f_1 : x \to y_1$ and $f_2 : x \to y_2$, we declare the unique morphism $y_1 \to y_2$ in the system to be $f_2 \circ f_1^{-1}$. Then we define $\mc{F}(x)=\varprojlim_f \mc{F}'(y_f)$, where $y_f$ denotes the target of $f$. Any morphism $x_1 \to x_2$ in $\mc{C}$ induces an obvious morphism between the two inverse systems, hence a morphism $\mc{F}(x_1) \to \mc{F}(x_2)$.

The same procedure works for natural transformations. That is, given a category $\mc{D}$, two functors $\mc{F}'_1,\mc{F}'_2 : \mc{C}' \to \mc{D}$, and a natural transformation $\eta' : \mc{F}'_1 \to \mc{F}'_2$, there exists a unique natural transformation $\eta : \mc{F}_1 \to \mc{F}_2$ extending $\eta'$. Indeed, for any $x \in \mc{C}$ we obtain from $\eta'$ a natural transformation $\mc{F}'_{1,x} \to \mc{F}'_{2,x}$ and hence a morphism $\eta'_x : \mc{F}_1(x) \to \mc{F}_2(x)$ that is functorial in $x$.

Spelled out, the morphisms $\eta'_{y_f} : \mc{F}_1'(y_f) \to \mc{F}_2'(y_f)$, indexed by $f : x \to y$, are compatible in $f$ and hence provide a morphism of compatible systems, thus a morphism between their limits, which we define to be $\eta_x : \mc{F}_1(x) \to \mc{F}_2(x)$. This morphism is functorial in $x$.

In our applications, the role of the category $\mc{C}$ will be played by the category $\mc{Z}^2(\Gamma,\hat T \to \hat T_\tx{ad})$ whose objects are the elements of $Z^2(\Gamma,\hat T \to \hat T_\tx{ad})$. For $t_1,t_2 \in Z^2(\Gamma,\hat T \to \hat T_\tx{ad})$, the set of morphisms $t_1 \to t_2$ will be the set 
\[ C^1(\Gamma,\hat T)_{t_1/t_2} = \{ x \in C^1(\Gamma,\hat T)\,|\, \partial x = z_2/z_1, \bar x=c_2/c_1 \}, \]
where we have written $t_i=(z_i,c_i)$. 

The role of the category $\mc{C}'$ will be played by the category $\mc{Z}^2(\Gamma,Z(\hat G))$ whose objects are the set $Z^2(\Gamma,Z(\hat G))$ and whose morphisms $z_1 \to z_2$ are the set
\[ C^1(\Gamma,Z(\hat G))_{z_1/z_2} = \{ x \in C^1(\Gamma,Z(\hat G))\,|\, \partial x = z_2/z_1 \}. \] 

\begin{lem} \label{lem:fullwide1}
The category $\mc{Z}^2(\Gamma,Z(\hat G))$ is a full essentially wide subcategory of $\mc{Z}^2(\Gamma,\hat T \to \hat T_\tx{ad})$. The set of isomorphism classes of either is $H^2(\Gamma,Z(\hat G))$.
\end{lem}
\begin{proof}
Fullness is immediately checked from the definitions of the homomorphism sets. It is clear that the set of isomorphism classes in $\mc{Z}^2(\Gamma,Z(\hat G))$ equals $H^2(\Gamma,Z(\hat G))$. We claim that the set of isomorphism classes in $\mc{Z}^2(\Gamma,\hat T \to \hat T_\tx{ad})$ equals $H^2(\Gamma,\hat T \to \hat T_\tx{ad})$. Given $t_1,t_2 \in Z^2(\Gamma,\hat T \to \hat T_\tx{ad})$ it is clear that elements of $C^1(\Gamma,\hat T)_{t_1/t_2}$ are coboundaries between $t_1$ and $t_2$. Conversely, a general coboundary is of the form $(\partial y_1,\bar y_1 \cdot \partial y_2)$ for $y_1 \in C^1(\Gamma,\hat T)$ and $y_2 \in C^0(\Gamma,\hat T_\tx{ad})$. Let $\dot y_2 \in \hat T$ be a lift of $y_2$. If we define $y=y_1 \cdot \partial \dot y_2$ , then this coboundary becomes $(\partial y,\bar y)$, but that is a morphism in $\mc{Z}^2(\Gamma,\hat T \to \hat T_\tx{ad})$. This proves the claim. Essential wideness now follows from the fact that the morphism of complexes $[Z(\hat G) \to 1] \to [\hat T \to \hat T_\tx{ad}]$ is a quasi-isomorphism, and hence induces an isomorphism $H^2(\Gamma,Z(\hat G)) \to H^2(\Gamma,\hat T \to \hat T_\tx{ad})$.
\end{proof}

Since we are often interested in finite order covers, we will now introduce a variation of these constructions that works with torsion cocycles.

Consider the category $\mc{Z}^2_\infty(\Gamma,\hat T \to \hat T_\tx{ad})$ whose objects are the elements of $Z^2(\Gamma,\hat T[\infty] \to \hat T_\tx{ad}[\infty])$. For $t_1,t_2 \in Z^2(\Gamma,\hat T[\infty] \to \hat T_\tx{ad}[\infty])$, the set of morphisms $t_1 \to t_2$ will be the set 
\[ C^1(\Gamma,\hat T[\infty])_{t_1/t_2} := \{ x \in C^1(\Gamma,\hat T[\infty])\,|\, \partial x = z_2/z_1, \bar x=c_2/c_1 \}, \]
where we have written $t_i=(z_i,c_i)$. Consider also the category $\mc{Z}^2_\infty(\Gamma,Z(\hat G))$ whose objects are the set $Z^2(\Gamma,Z(\hat G)[\infty])$ and whose morphisms $z_1 \to z_2$ are the set
\[ C^1(\Gamma,Z(\hat G)[\infty])_{z_1/z_2} = \{ x \in C^1(\Gamma,Z(\hat G)[\infty])\,|\, \partial x = z_2/z_1 \}. \] 
The analog of Lemma \ref{lem:fullwide1} holds in this context as well, but needs some preparation.

\begin{lem}\ \\[-15pt] \label{lem:torsionstuff}
\begin{enumerate}
	\item The sequence $1 \to Z(\hat G)[\infty] \to \hat T[\infty] \to \hat T_\tx{ad}[\infty] \to 1$ is exact.
	\item Given $t \in Z^2(\Gamma,\hat T[\infty] \to \hat T_\tx{ad}[\infty])$, we have $\tilde H^1(\Gamma,\hat T[\infty])_t=\tilde H^1(\Gamma,\hat T)_t$.
	\item Given $z \in Z^2(\Gamma,Z(\hat G)[\infty])$, we have $\tilde H^1(\Gamma,Z(\hat G)[\infty])_z=\tilde H^1(\Gamma,Z(\hat G))_z$.
\end{enumerate}
\end{lem}
\begin{proof}
(1) To see surjectivity of $\hat T[\infty] \to \hat T_\tx{ad}[\infty]$, pick $\bar x \in \hat T_\tx{ad}[\infty]$. Let $n \in \N$ be such that $\bar x^n=1$ and let $m=|\pi_0(Z(\hat G))|$. Let $x \in \hat T$ be a lift of $\bar x$. Then $x^{nm}$ lies in the identity component $Z(\hat G)^\circ$, which is a complex torus, hence divisible, so there exists $y \in Z(\hat G)^\circ$ with $y^{nm}=x^{nm}$. Then $y^{-1}x \in \hat T[\infty]$ also lifts $\bar x$.

(2) Let $c \in C^1(\Gamma,\hat T)_t$ and let $n \in \N$ be such that $t^n=1$. Then $c^{nm} \in Z^1(\Gamma,Z(\hat G)^\circ)$. Since $H^1(\Gamma,Z(\hat G)^\circ)$ is torsion and $Z(\hat G)^\circ$ is divisible, after possibly enlarging $m$ we find $b \in Z(\hat G)^\circ$ such that $b^{-nm}\sigma(b^{nm})=c^{nm}$. Then $(b^{-1}\sigma(b))^{-1}c \in C^1(\Gamma,\hat T[\infty])_t$ represents the same coset modulo $B(\Gamma,\hat T)$ as $c$.

Consider now $b^{-1}\sigma(b) \in B^1(\Gamma,\hat T) \cap C^1(\Gamma,\hat T[\infty])$. Let $n \in \N$ be such that $b^{-n}\sigma(b^n)=1$. Thus $b^n \in \hat T^\Gamma$ and upon enlarging $n$ we find $b^n \in \hat T^{\Gamma,0}$. The divisibility of $\hat T^{\Gamma,0}$ provides $t \in \hat T^{\Gamma,0}$ such that $t^n=b^n$. Then $t^{-1}b \in \hat T[\infty]$ and $(t^{-1}b)^{-1}\sigma(t^{-1}b)=b^{-1}\sigma(b)$.

(3) The proof is the same as for (2).
\end{proof}

\begin{lem} \label{lem:fullwide2}
The category $\mc{Z}^2_\infty(\Gamma,Z(\hat G))$ is a full essentially wide subcategory of $\mc{Z}^2_\infty(\Gamma,\hat T \to \hat T_\tx{ad})$. The set of isomorphism classes of either is $H^2(\Gamma,Z(\hat G))$.
\end{lem}
\begin{proof}
Fullness is again immediately checked from the definitions of the homomorphism sets. It is clear that the set of isomorphism classes of $\mc{Z}^2_\infty(\Gamma,Z(\hat G))$ is $H^2(\Gamma,Z(\hat G)[\infty])$. That the set of isomorphism classes of $\mc{Z}^2_\infty(\Gamma,\hat T \to \hat T_\tx{ad})$ is $H^2(\Gamma,\hat T[\infty] \to \hat T_\tx{ad}[\infty])$ follows from the same argument as in the proof of Lemma \ref{lem:fullwide1}, except that for the existence of a lift $\dot y_2 \in \hat T[\infty]$ of $y_2 \in \hat T_\tx{ad}[\infty]$ we need to appeal to Lemma \ref{lem:torsionstuff}(1). Applying Lemma \ref{lem:torsionstuff}(1) again we see that those two sets of isomorphism classes are equal, and essential wideness follows. Finally, Corollary \ref{cor:uniqdiv} shows that those sets of isomorphism classes equal $H^2(\Gamma,\hat T \to \hat T_\tx{ad})$.
\end{proof}

Thus we now have the following diagram
\[ \xymatrix{
	\mc{Z}^2_\infty(\Gamma,Z(\hat G))\ar[r]\ar[d]&\mc{Z}^2(\Gamma,Z(\hat G))\ar[d]\\
	\mc{Z}^2_\infty(\Gamma,\hat T \to \hat T_\tx{ad})\ar[r]&\mc{Z}^2_\infty(\Gamma,\hat T \to \hat T_\tx{ad})
}
\]
where the vertical maps are full essentially wide embeddings, while the horizontal maps are essentially wide, but not full.

As further variations, we have the category $\mc{\bar Z}^2(\Gamma,\hat T \to \hat T_\tx{ad})$ where the set of objects is $Z^2(\Gamma,\hat T \to \hat T_\tx{ad})/B^1(\Gamma,\hat T_\tx{ad})$ and the set of morphisms $\bar z_1 \to \bar z_2$ is the quotient $C^1(\Gamma,\hat T)_{z_1/z_2}/B^1(\Gamma,\hat T)$. We also have the torsion analog $\mc{\bar Z}^2_\infty(\Gamma,\hat T \to \hat T_\tx{ad})$ where the set of objects is $Z^2(\Gamma,\hat T[\infty] \to \hat T_\tx{ad}[\infty])/B^1(\Gamma,\hat T_\tx{ad}[\infty])$ and the set of morphisms $\bar z_1 \to \bar z_2$ is the quotient $C^1(\Gamma,\hat T[\infty])_{z_1/z_2}/B^1(\Gamma,\hat T[\infty])$.

We have the evident functors $\mc{Z}^2(\Gamma,\hat T \to \hat T_\tx{ad}) \to \mc{\bar Z}^2(\Gamma,\hat T \to \hat T_\tx{ad})$ and $\mc{Z}^2_\infty(\Gamma,\hat T \to \hat T_\tx{ad}) \to \mc{\bar Z}^2_\infty(\Gamma,\hat T \to \hat T_\tx{ad})$, which are (essentially) surjective and full, but not faithful, and induce bijections on the sets of isomorphism classes. We also have the evident functor $\mc{\bar Z}^ 2_\infty(\Gamma,\hat T \to \hat T_\tx{ad}) \to \mc{\bar Z}^2(\Gamma,\hat T \to \hat T_\tx{ad})$, which is an equivalence of categories according to Lemma \ref{lem:torsionstuff}(2).

Analogously, we have the category $\mc{\bar Z}^2(\Gamma,Z(\hat G))$ where the set of objects is $Z^2(\Gamma,Z(\hat G))$ and the set of morphisms $z_1 \to z_2$ is given by the quotient $C^1(\Gamma,Z(\hat G))_{z_1/z_2}/B^1(\Gamma,Z(\hat G))$. Its torsion analog is $\mc{\bar Z}^2_\infty(\Gamma,Z(\hat G))$, where the set of objects is $Z^2(\Gamma,Z(\hat G)[\infty])$ and the set of morphisms $z_1 \to z_2$ is given by the quotient $C^1(\Gamma,Z(\hat G)[\infty])_{z_1/z_2}/B^1(\Gamma,Z(\hat G)[\infty])$. Again, the obvious functor $\mc{\bar Z}^2_\infty(\Gamma,Z(\hat G)) \to \mc{\bar Z}^2(\Gamma,Z(\hat G))$ is an equivalence of categories.

Consider now the category $\mc{C}(G)$ of covers of $G(F)$ by $\mb{S}^1$ obtained from the pseudo-universal extension \eqref{eq:baruni}, and equipped with a splitting over $G_\tx{sc}(F)$. Consider also the category $\mc{C}_\infty(G)$ of covers of $G(F)$ by $\mu_\infty(\C)$ obtained in the same way, but from characters of $\bar\pi_1(G)$ of finite order, again equipped with a splitting over $G_\tx{sc}(F)$. Pushing out along $\mu_\infty(\C) \to \mb{S}^1$ gives a functor $\mc{C}_\infty(G) \to \mc{C}(G)$.

The objects of $\mc{Z}^2(\Gamma,\hat T \to \hat T_\tx{ad})$ are the continuous characters $\bar\pi_1(G) \to \mb{S}^1$. Constructions \ref{cns:h-univ} and \ref{cns:iso-h} provide a functor 
\[ \mc{Z}^2(\Gamma,\hat T \to \hat T_\tx{ad}) \to \mc{C}(G),\qquad t \mapsto G(F)_t, \]
and this functor factors through $\mc{\bar Z}^2(\Gamma,\hat T \to \hat T_\tx{ad})$.

The objects of $\mc{Z}^2_\infty(\Gamma,\hat T \to \hat T_\tx{ad})$ are the continuous characters of $\bar\pi_1(G)$ of finite order, or equivalently (Lemma \ref{lem:surjfin}) the continuous characters $t: \bar\pi_1(G) \to \mu_\infty(\C)$. We obtain a functor
\[ \mc{Z}^2_\infty(\Gamma,\hat T \to \hat T_\tx{ad}) \to \mc{C}_\infty(G),\qquad t \mapsto G(F)_{t,\infty}, \]
and this functor factors through $\mc{\bar Z}^2_\infty(\Gamma,\hat T \to \hat T_\tx{ad})$. This results in a functor $\mc{\bar Z}^2_\infty(\Gamma,\hat T \to \hat T_\tx{ad}) \to \mc{C}_\infty(G)$ which is an equivalence of categories according Lemma \ref{lem:h-iso}, Corollary \ref{cor:h-iso}, and Lemma \ref{lem:fullwide2}. 

Fix an element of $H^2(\Gamma,Z(\hat G))$ and let $m$ be its order. The $m$-th power map is surjective on the terms of the complex $\hat T \to \hat T_\tx{ad}$ and the resulting long exact sequence implies the surjectivity of $H^2(\Gamma,\hat T[m] \to \hat T_\tx{ad}[m]) \to H^2(\Gamma,\hat T \to \hat T_\tx{ad})[m]=H^2(\Gamma,Z(\hat G))[m]$. This shows that there exists an object $G(F)_{t,\infty}$ in the corresponding isomorphism class of $\mc{C}_\infty(G)$ with $t : \bar\pi_1(G) \to \mu_m(\C)$. At the same time, there also exists an object of this form with $t' : \tilde\pi_1(G) \to \mu_{m'}(\C)$. However, in the latter case we may not have $m'=m$. That is, the condition that $t' : \bar\pi_1(G) \to \mu_{m'}(\C)$ factor through $\tilde\pi_1(G)$ and the condition $m'=m$ may not be simultaneously satisfiable.

\subsection{The local correspondence for covers of connected reductive groups} \label{sub:llc}

The local correspondence for the cover $S(F)_\infty$ of an $F$-torus $S$ was obtained very quickly in Proposition \ref{pro:llc-s}, and the local correspondences for the various covers $S(F)_{t,*}$ follows directly. The case of a connected reductive group is less direct, and will be discussed now. We will only present a classification of the genuine representations of a cover $G(F)_{t,*}$ for a character $t : \bar\pi_1(G) \to \mb{S}^1$, where $*$ can be void, $n \in \N$, or $\infty$.

We begin with the definition of the $L$-group of $G(F)_{t,*}$. As in the case \eqref{eq:s-l} of tori, this $L$-group will only depend on $t$, but not on $*$. The case when $t$ factors through $\tilde\pi_1(G)$ is more immediate, and exactly analogous to that of tori. Indeed, we have $t \in Z^2(\Gamma,Z(\hat G))$, which allows us to form the twisted product
\begin{equation} \label{eq:h-l1}
^LG_t = \hat G \boxtimes_t \Gamma.
\end{equation}
It is equal to the pushout of $Z(\hat G) \boxtimes_t \Gamma$ along the inclusion $Z(\hat G) \to \hat G$.

To extend the definition of $^LG_t$ to $t \in Z^2(\Gamma,\hat T \to \hat T_\tx{ad})$ we follow the formal procedure of \S\ref{sub:elementary}. We interpret \eqref{eq:h-l1} as a functor from $\mc{Z}^2(\Gamma,Z(\hat G))$ to the category of extensions of $\Gamma$ by $Z(\hat G)$ equipped with a splitting. It sends the object $z \in Z^2(\Gamma,Z(\hat G))$ to $\hat G \boxtimes_z \Gamma$ and the morphism $c \in C^1(\Gamma,Z(\hat G))_{z_1/z_2}$ to the morphism 
\[ ^L\xi_c : \hat G \boxtimes_{z_2} \Gamma \to \hat G \boxtimes_{z_1} \Gamma,\qquad g \boxtimes \sigma \mapsto gc(\sigma)\boxtimes\sigma.\]
The unique extension of this functor to $\mc{Z}^2(\Gamma,\hat T \to \hat T_\tx{ad})$ provides an $L$-group $^LG_t$ for any $t \in Z^2(\Gamma,\hat T \to \hat T_\tx{ad})$. Explicitly, writing $t=(z,c)$, $^LG_t$ it is given by
\begin{equation} \label{eq:h-l2}
^LG_t = \varprojlim_x \hat G \boxtimes_{z \cdot \partial x^{-1}} \Gamma.
\end{equation}
where the limit is taken over the set $\{x \in C^1(\Gamma,\hat T)\,|\, \bar x = c\}$ and the morphisms in the compatible system are 
\[ \hat G \boxtimes_{z \cdot \partial x^{-1}} \Gamma \to \hat G \boxtimes_{z \cdot \partial x'^{-1}} \Gamma,\qquad g \boxtimes \sigma \mapsto g \cdot x^{-1}(\sigma)x'(\sigma) \boxtimes \sigma. \]
Given $t_1=(z_1,c_1),t_2=(z_2,c_2) \in Z^2(\Gamma,\hat T \to \hat T_\tx{ad})$ and $c \in C^1(\Gamma,\hat T)_{t_1/t_2}$ we have the isomorphism
\begin{equation} \label{eq:h-dual-iso}
^L\xi_c : {^LG_{t_2}} \to {^LG_{t_1}}
\end{equation}
obtained by splicing the isomorphisms $^L\xi_{cx_2^{-1}x_1} : {^LG}_{z_2 \cdot \partial x_2^{-1}} \to {^LG}_{z_1 \cdot \partial x_1^{-1}}$ for each $x_1,x_2 \in C^1(\Gamma,\hat T)$ with $\bar x_i=c_i$.

\begin{fct}
The $L$-group of the Baer sum $G(F)_{t_1,n} \oplus G(F)_{t_2,n}$ is naturally identified with the Baer sum of the $L$-groups $^LG_{t_1} \oplus {^LG}_{t_2}$.
\end{fct}

A Langlands parameter valued in $^LG_t$ is an $L$-homomorphism $L_F \to {^LG_t}$ that is continuous on $W_F$, algebraic on $\tx{SL}_2(\C)$, and respects Jordan decompositions. Two such parameters are considered equivalent if they are conjugate under $\hat G$. The set of equivalence classes of such parameters will be denoted by $\Phi(G_t)$.

For the next theorem we assume the refined local Langlands correspondence for all connected reductive groups that are central extensions of $G$ by induced tori, i.e. of the form $1 \to Z_1 \to G_1 \to G \to 1$ with $Z_1$ a product of restrictions of scalars of the multiplicative group. In fact, we will only need to consider so called $z$-extensions, i.e. those for which the derived subgroup of $G_1$ is simply connected. Recall that a $z$-extension always exists \cite[\S12.4]{BTBOOK}.

\begin{thm} \label{thm:llc-h}
Assume that the refined local Langlands correspondence is valid for $z$-extensions of $G$. Then it is also valid for the covers $G(F)_{t,*}$. More precisely, the following holds.
\begin{enumerate}
	\item There is a natural surjective map $\Pi(G_{t,*}) \to \Phi(G_t)$.
	\item A Whittaker datum for $G$ determines an injective (bijective when $F$ is non-archimedean) map from the fiber $\Pi_\varphi(G_{t,*})$ of this map over $\varphi$ to the set $\tx{Irr}(\pi_0(S_\varphi^+))$. 
	\item For any $c \in C^1(\Gamma,\hat T)_{t/t'}$, the above two points translate the isomorphism $\xi_{[c]}$ of Construction \ref{cns:iso-h} to the isomorphism $^L\xi_c$ of \eqref{eq:h-dual-iso}. If the Weil forms of the $L$-groups are used, the same holds more generally for $c \in \tilde Z^1_u(W_F,\hat T)_{t/t'}$.
\end{enumerate}
\end{thm}
This correspondence also satisfies an appropriate formulation of endoscopic transfer, cf. Theorem \ref{thm:charid}.
\begin{proof}
In this proof we are going to use the Weil forms of the $L$-groups rather than their Galois forms, without changing notation. Since the Langlands parameters are insensitive to whether they are valued in the Galois form or Weil form of the $L$-group, this change in notation will not affect the statement. We will also assume that $t \in Z^2(\Gamma,Z(\hat G))$, deferring the case $t \in Z^2(\Gamma,\hat  T \to \hat T_\tx{ad})/B^1(\Gamma,\hat T_\tx{ad})$ to the very end of the proof, when it will follow formally. Thus, from now on until explicitly stated otherwise, we have $^LG_t = \hat G \boxtimes_t W_F$. 

Choose a $z$-extension $1 \to Z_1 \to G_1 \to G \to 1$. We have the exact sequence of dual groups $1 \to \hat G \to \hat G_1 \to \hat Z_1 \to 1$. We will denote again by $t$ the pull-back of $t$ under $\tilde\pi_1(G_1) \to \tilde\pi_1(G)$, i.e. the image of $t$ under $Z^2(\Gamma,Z(\hat G)) \to Z^2(\Gamma,Z(\hat G_1))$. We then have the cover $G_1(F)_{t,n}$ and the $L$-group $^LG_{1,t} = \hat G_1 \boxtimes_t \Gamma$.
Note that $g \boxtimes \sigma \mapsto i(g) \boxtimes \sigma$ is an $L$-embedding $^Li : {^LG}_t \to {^LG}_{1,t}$, where we have written $i$ for the inclusion $\hat G \to \hat G_1$. From now on we will often suppress $i$ from the notation.

The center $Z(\hat G_1)$ is connected. According to Proposition \ref{pro:cc_rajan} there exists $r_1 \in C^1(W_F,Z(\hat G_1))$ whose differential equals (the inflation to $W_F$ of) the $2$-cocycle $t$. Therefore the map 
\[ {^LG}_{1,t} \to {^LG}_1 = \hat G_1 \rtimes W_F,\qquad g \boxtimes \sigma \mapsto gr_1(\sigma) \rtimes\sigma \]
is an isomorphism. Composing with $^Li$ provides an $L$-embedding 
\begin{equation} \label{eq:l-emb-z1}
^LG_t \to {^LG_1}.
\end{equation}
This $L$-embedding induces a $\hat G$-equivariant injection between the sets of $L$-parameters valued in $^LG_t$ and $^LG_1$, respectively. We claim that two parameters in the image of this injection are $\hat G_1$-conjugate if and only if they are $\hat G$-conjugate. For this consider $\varphi : L_F \to {^LG_t}$ and $g_1 \in \hat G_1$ so that  $g_1\varphi(\sigma)g_1^{-1}$ is again valued in $^LG_t$. Writing $g_1=z_1g$ with $z_1 \in Z(\hat G_1)$ and $g \in \hat G$ we reduce to the case $g=1$. Then $z_1\varphi(\sigma)z_1^{-1}=z_1 \cdot \sigma(z_1)^{-1} \cdot \varphi(\sigma)$ and we see that $z_1 \cdot \sigma(z_1)^{-1}$ represents an element of 
\begin{eqnarray*}
&&\tx{ker}(H^1(W_F,Z(\hat G)) \to H^1(W_F,Z(\hat G_1)))\\
&=&\tx{cok}(H^0(W_F,Z(\hat G_1)) \to H^0(W_F,\hat Z_1)))\\
&=&\tx{cok}(H^0(\Gamma,Z(\hat G_1)) \to H^0(\Gamma,\hat Z_1)))\\
&=&\tx{ker}(H^1(\Gamma,Z(\hat G)) \to H^1(\Gamma,Z(\hat G_1))),
\end{eqnarray*}
but the latter is trivial due to \cite[Corollary 2.3]{Kot84} and the fact that $Z_1$ is induced. Therefore $z_1 \in Z(\hat G) \cdot Z(\hat G_1)^\Gamma$, and the claim has been proved.

We next describe which $L$-parameters valued in $^LG_1$ factor through \eqref{eq:l-emb-z1}. The composition of \eqref{eq:l-emb-z1} with $\hat G_1 \rtimes W_F \to (\hat G_1/\hat G) \rtimes W_F$ kills $\hat G \subset {^LG_t}$, therefore factors through the projection $^LG_t \to W_F$ and induces an $L$\-homomorphism $\varphi_{\lambda_1} : W_F \to (\hat G_1/\hat G) \rtimes W_F$. An $L$\-parameter valued in $^LG_1$ factors through \eqref{eq:l-emb-z1} if and only if its composition with the projection $\hat G_1 \rtimes W_F \to (\hat G_1/\hat G) \rtimes W_F$ equals $\varphi_{\lambda_1}$. It follows that any irreducible representation of %
$G_1(F)$ whose $L$-parameter factors through \eqref{eq:l-emb-z1} will transform under the central subgroup $Z_1(F)$ by the character ${\lambda_1}$ of $Z_1(F)$ whose parameter is $\varphi_{\lambda_1}$.

Let $T$ and $T_1$ denote the universal maximal tori of $G$ and $G_1$, respectively. Then $T_1(F)_{t,*}$ is the pull back of $T(F)_{t,*}$ along $T_1(F) \to T(F)$ according to Corollary \ref{cor:llc-s}(5). Thus we have the commutative diagram with exact rows
\[\xymatrix{
1\ar[r]&Z_1(F)\ar@{=}[d]\ar[r]&T_1(F)_{t,*}\ar[d]\ar[r]&T(F)_{t,*}\ar[r]\ar[d]&1\\	
1\ar[r]&Z_1(F)\ar[r]&T_1(F)\ar[r]&T(F)\ar[r]&1
}\]
The 1-cochain $\sigma \mapsto r_1(\sigma)^{-1}$ is the parameter of a genuine character ${\mu_1}$ of $T_1(F)_{t,*}$ via Corollary \ref{cor:llc-s}(3). 

We claim that the restriction of ${\mu_1}$ to $Z_1(F)$ equals $\lambda_1^{-1}$, while the restriction of ${\mu_1}$ via the splitting $T_\tx{sc}(F) \to T_1(F)_{t,*}^{G_1}$ of Construction \ref{cns:h-split} is trivial. Both of these claims follow from Corollary \ref{cor:dc_pull}. For the restriction of ${\mu_1}$ to $T_\tx{sc}(F)$ we use the exact sequence $1 \to T_\tx{sc} \to T_1 \to T_1/T_\tx{sc} \to 1$ and see that this restriction has parameter given by the image of $r_1^{-1}$ in $C^1(W_F,\hat T_1/Z(\hat G_1))$, which is trivial. For the restriction of ${\mu_1}$ to $Z_1(F)$ we use the exact sequence $1 \to Z_1 \to T_1 \to T \to 1$ and see that the restriction of ${\mu_1}$ to $Z_1(F)$ has parameter equal to the image of $r_1(\sigma)^{-1}$ in $\hat Z_1=\hat G_1/\hat G$. On the other hand, the parameter of ${\lambda_1}$ is obtained by composing \eqref{eq:l-emb-z1} with the projection ${^LG_1} \to (\hat G_1/\hat G) \rtimes W_F$ and then factoring through the projection ${^LG_t} \to W_F$. The resulting homomorphism $W_F \to (Z(\hat G_1)/Z(\hat G))\rtimes W_F$ is given by $\sigma \mapsto r_1(\sigma) \rtimes \sigma$. The claim about the restrictions of ${\mu_1}$ is now proved.

Since the restriction of $\mu_1$ to $T_\tx{sc}(F)$ is trivial, $\mu_1$ descends to a genuine character of $\tx{cok}(T_\tx{sc}(F) \to T_1(F)_{t,*})$. According to Fact \ref{fct:gtn} we have the natural isomorphism $\tx{cok}(T_\tx{sc}(F) \to T_1(F)_{t,*})=\tx{cok}(G_\tx{sc}(F) \to G_1(F)_{t,*})$. This allows us to inflate $\mu_1$ to a genuine character of $G_1(F)_{t,*}$, which we again denote by $\mu_1$. As shown above, its restriction to $Z_1(F)$ equals $\lambda_1^{-1}$.

Consider now an $L$-parameter $\varphi : L_F \to {^LG_t}$ and let $\varphi_1 : L_F \to {^LG_1}$ be its composition with \eqref{eq:l-emb-z1}. Let $\Pi_{\varphi_1}$ denote the $L$-packet of representations of 
$G_1(F)$ associated to $\varphi_1$. 

Given a representation 
$\pi_1 \in \Pi_{\varphi_1}$ we inflate $\pi_1$ to a non-genuine representation of $G_1(F)_{t,*}$ and consider $\pi_1\otimes\mu_1$. From the above commutative diagram we have the exact sequence
\[ 1 \to Z_1(F) \to G_1(F)_{t,*} \to G(F)_{t,*} \to 1. \]
The restriction of $\mu_1$ to $Z_1(F)$ being $\lambda_1^{-1}$, we see that $\pi_1\otimes\mu_1$ factors through a genuine representation of $G(F)_{t,*}$. We 
define the $L$-packet $\Pi_\varphi$ associated to $\varphi$ as 
\[ \Pi_\varphi = \{ \pi_1\otimes\mu_1\,|\, \pi_1 \in \Pi_{\varphi_1}\}. \]
If a Whittaker datum for $G$ has been fixed, it induces a Whittaker datum for $G_1$, hence an injection of $\Pi_{\varphi_1}$ into $\tx{Irr}(\pi_0(S_{\varphi_1}/Z(\hat G_1)^\Gamma))$. Combining this with the identity $S_{\varphi_1}/Z(\hat G_1)^\Gamma = S_\varphi/Z(\hat G)^\Gamma$ 
we obtain the desired injection of $\Pi_\varphi$ into $\tx{Irr}(\pi_0(S_{\varphi}/Z(\hat G)^\Gamma))$.

The construction of $\Pi_\varphi$ and its injection into $\tx{Irr}(\pi_0(S_\varphi/Z(\hat G)^\Gamma))$ involved choosing the $z$-extension $G_1$ and the 1-cochain $r_1$. We want to argue that the outcome of the construction is independent of these choices. Keeping $G_1$
fixed, we can replace $r_1$ only by $r_2 := z_1r_1$ for some $z_1 \in Z^1(W_F,Z(\hat G_1))$. Let $\chi_1 : G_1(F) \to \C^\times$ be the character associated to $z_1$. From $r_1$ and $r_2$ we obtain two versions $s_1$ and $s_2$ of \eqref{eq:l-emb-z1}, hence two different $L$-embeddings $^LG_t \to {^LG_1}$. Let $\varphi_1,\varphi_2$ be the corresponding parameters for $G_1$. Then $\varphi_2 = z_1 \cdot \varphi_1$, hence $\Pi_{\varphi_2} = \chi_1 \otimes \Pi_{\varphi_1}$. At the same time the two versions $\mu_1,\mu_2$ of the genuine character of $G_1(F)_{t,*}$ are related by $\mu_2 = \chi_1^{-1}\mu_1$. Therefore $\Pi_\varphi$ remains unchanged. Since the bijection $\pi_2 \mapsto \pi_1 = \chi_1 \otimes \pi_2$ is compatible with the identification $S_{\varphi_1}=S_{\varphi_2}$, the bijection between $\Pi_\varphi$ and $\tx{Irr}(\pi_0(S_\varphi/Z(\hat G)^\Gamma))$ is also unaffected by passing from $r_1$ to $r_2$.

To see that the choice of $G_1$ is also irrelevant, we use the fact that given a second $z$-extension $G_2$ there exists a third $z$-extension $G_3$ that is simultaneously a $z$-extension of $G_1$ and $G_2$. It is enough to check that using $G_3$ gives the same result as using $G_1$. Let $Z_3 = \tx{ker}(G_3 \to G)$. We have the exact sequence $1 \to Z_3' \to Z_3 \to Z_1 \to 1$ where $Z_3'=\tx{ker}(G_3 \to G_1)$ is also an induced torus. After choosing $r_1 \in C^1(W_F,Z(\hat G_1))$ we can take $r_3 \in C^1(W_F,Z(\hat G_3))$ to be equal to $r_1$. This makes the character $\lambda_3$ of $Z_3(F)$ equal the inflation of the character $\lambda_1$ of $Z_1(F)$ under the surjective homomorphism $Z_3(F) \to Z_1(F)$, and the genuine character $\mu_3$ of $G_3(F)_{t,*}$ is the inflation of the genuine character $\mu_1$ of $G_1(F)_{t,*}$. The $L$-embedding $^LG_t \to {^LG_3}$ is the composition of the $L$-embedding $^LG_t \to {^LG_1}$ with the $L$-embedding $^LG_1 \to {^LG_3}$ that is given tautologically by $g_1 \rtimes \sigma \mapsto g_1 \rtimes \sigma$, i.e. by the inclusion $\hat G_1 \to \hat G_3$ dual to the projection $G_3 \to G_1$. Suppressing this tautological $L$-embedding from the notation we have $\varphi_1=\varphi_3$ and hence $\Pi_{\varphi_1}=\Pi_{\varphi_3}$, where we identify representations of $G_3(F)$ trivial on $Z_3'(F)$ with representations of $G_1(F)$. This identification is compatible with the identification $S_{\varphi_1}/Z(\hat G_1)^\Gamma = S_{\varphi_3}/Z(\hat G_3)^\Gamma$. 

The proof of (1) and (2) under the assumption $t \in Z^2(\Gamma,Z(\hat G))$ is now complete. We proceed to prove (3) under the same assumption. For this, we trace how the construction given so far changes when we pass from $t$ to $t'$ via $\xi_{c}$. We made the choice of $r_1 \in C^1(W_F,Z(\hat G_1))$ with $\partial r_1=t$. Setting $r_1'=c \cdot r_1$ we obtain $\partial r_1'=t'$. This leads to the commutative diagram
\[ \xymatrix{
\hat G \boxtimes_t\Gamma\ar[r]^{\cdot r_1}&\hat G_1 \rtimes \Gamma\ar@{=}[d]\\
\hat G \boxtimes_{t'}\Gamma\ar[r]^-{\cdot cr_1}\ar[u]^-{\cdot c}&\hat G_1 \rtimes \Gamma\\
}\]
Therefore, if $\varphi : L_F \to {^LG_t}$ and $\varphi' : L_F \to {^LG_{t'}}$ correspond under the left vertical map, they lead to the same $\varphi_1 : L_F \to {^LG_1}$, hence to the same packet $\Pi_{\varphi_1}$. Then we have $\Pi_\varphi = \Pi_{\varphi_1}\otimes\mu_1$ and $\Pi_{\varphi'} = \Pi_{\varphi_1} \otimes\mu_1'$, where $\mu_1$ and $\mu_1'$ are the genuine characters of $G_1(F)_{t,\infty}$ and $G_1(F)_{t',\infty}$ with parameters $r_1^{-1}$ and $c^{-1}r_1^{-1}$, respectively. The isomorphism $\xi_c : G(F)_{t,\infty} \to G(F)_{t',\infty}$ is compatible with the analogous isomorphism $\xi_c : G_1(F)_{t,\infty} \to G_1(F)_{t',\infty}$ constructed the same way, and the latter identifies the genuine characters $\mu_1$ and $\mu_1'$. The proof of (3) is thus complete.

Finally, we drop the assumption $t \in Z^2(\Gamma,Z(\hat G))$. As discussed in Remark \ref{rem:n-m}, the set $\Pi(G_{t,*})$ is independent of $*$, so we can write $\Pi(G_t)$ for it. We have the two functors $\mc{Z}^2(\Gamma,Z(\hat G)) \to \tx{Sets}$, one given by $t \mapsto \Pi(G_t)$ and sending $c \in C^1(\Gamma,\hat T)_{t_1/t_2}$ to the map $\Pi(G_{t_2}) \to \Pi(G_{t_1})$ that is pull-back by $\xi_{[c]}$, and the other given by $t \mapsto \{ (\varphi,\rho)\,|\, \varphi \in \Phi(G_t),\rho \in \tx{Irr}(S_\varphi/Z(\hat G)^\Gamma)\}$ and sending $c$ to the map obtained by composition with the isomorphism $^L\xi_c$ of \eqref{eq:h-dual-iso}. The content of (1), (2), and (3) that have been proved under the assumption $t \in Z^2(\Gamma,Z(\hat G))$ is the construction of a natural transformation between these two functors. The argument of \S\ref{sub:elementary} extends this natural transformation to $\mc{Z}^2(\Gamma,\hat T \to \hat T_\tx{ad})$. This completes the proof of (1), (2), and (3), in the case of $t \in Z^2(\Gamma,\hat T \to \hat T_\tx{ad})$.
\end{proof}

\subsection{Maximal tori}

In this subsection we will show how the extension $G(F)_\infty$, the various covers $G(F)_{t,*}$, and their $L$-groups $^LG_t$, interact with maximal tori of $G$.

\begin{cns} \label{cns:maxtori}
Let $S \subset G$ be a maximal torus. We denote by $S(F)_\infty^G$ the pushout of $S(F)_\infty$ along $\tilde\pi_1(S) \to \tilde\pi_1(G)$. We will now construct a lift $S(F)_\infty^G \to G(F)_\infty$ of the inclusion $S(F) \to G(F)$, where $G(F)_\infty$ denotes the extension \eqref{eq:uni}.

We first assume that $G$ has a simply connected derived subgroup and write $D=G/G_\tx{der}$. According to Remark \ref{rem:gsc_abelian}, it is enough to construct a pair of homomorphisms $S(F)_\infty^G \to D(F)_\infty$ and $S(F)_\infty^G \to G(F)$ that equalize towards $D(F)$. The homomorphism $S(F)_\infty^G \to G(F)$ is the composition of $S(F)_\infty^G \to S(F)$ and the inclusion $S(F) \to G(F)$. The functoriality of pseudo-universal covers of tori (Proposition \ref{pro:func-s}) provides a homomorphism $S(F)_\infty \to D(F)_\infty$ that factors through $S(F)_\infty^G$. This completes the construction.

Now drop the assumption that $G$ has simply connected derived subgroup. Let $1 \to Z_1 \to G_1 \to G \to 1$ be a $z$-extension. Let $S_1$ be the pull-back of $S \to G \from G_1$, a maximal torus of $G_1$. The previous paragraph provides a homomorphism $S_1(F)_\infty^{G_1} \to G_1(F)_\infty$. We will now argue that this homomorphism descends to a homomorphism $S(F)_\infty^G \to G(F)_\infty$.

Proposition \ref{pro:func-s} provides a homomorphism $S_1(F)_\infty \to S(F)_\infty$ which identifies the push-out $S_1(F)_\infty^G$ of $S_1(F)_\infty$ along $\tilde\pi_1(S_1) \to \tilde\pi_1(G)$ with the pull-back of $S_1(F) \to S(F) \from S(F)_\infty^G$. This identification provides a homomorphism $Z_1(F) \to S_1(F)_\infty^G$ whose cokernel is the homomorphism $S_1(F)_\infty^G \to S(F)_\infty^G$.

The same argument, applied to the universal maximal torus $T_1$ of $G_1$, provides a homomorphism $Z_1(F) \to T_1(F)_\infty^G$ whose cokernel is the homomorphism $T_1(F)_\infty^G \to T(F)_\infty^G$. Glancing at Construction \ref{cns:hdc} we see that the push-out $G_1(F)_\infty^G$ of $\tilde\pi_1(G) \from \tilde\pi_1(G_1) \to G_1(F)_\infty$ is identified with $\tx{cok}(T_\tx{sc}(F) \to G_\tx{sc}(F) \rtimes T_1(F)_\infty^G)$, and therefore we obtain a homomorphism $Z_1(F) \to G_1(F)_\infty^G$ whose cokernel is the homomorphism $G_1(F)_\infty^G \to G(F)_\infty$.

The homomorphism $S_1(F)_\infty^{G_1} \to G_1(F)_\infty$ provides, via push-out along the surjection $\tilde\pi_1(G_1) \to \tilde\pi_1(G)$, a homomorphism $S_1(F)_\infty^G \to G_1(F)_\infty^G$ that respects the homomorphisms of $Z_1(F)$ into both, and hence descends to the desired homomorphism $S(F)_\infty^G \to G(F)_\infty$. The independence of that homomorphism from the choice of $z$-extension can be verified in a routine manner using \cite[Lemma 2.4.4]{Kot84}, and is left to the reader.
\end{cns}

\begin{rem}%
For each $t : \tilde\pi_1(G) \to \mb{S}^1$, the homomorphism of Construction \ref{cns:maxtori} induces a homomorphism $S(F)_{t,*} \to G(F)_{t,*}$. For computations it may be useful to note that we can run this construction from the beginning with $G(F)_{t,*}$ in place of $G(F)_\infty$. This has the benefit of allowing us to apply the easy construction, that avoids $z$-extensions, in situations where the derived subgroup of $G$ may not be simply connected. Indeed, according to Proposition \ref{pro:gsc_abelian1} it is enough to assume that the class of $t$ in $H^2(W_F,Z(\hat G))$ is trivial.
\end{rem}

\begin{cns} \label{cns:s-t}
We will now upgrade our construction $S(F)_t \to G(F)_t$ from the case of character $t$ of $\tilde\pi_1(G)$ to the case of characters $t$ of $\bar\pi_1(G)$. This will require making sense of $S(F)_t$ in the first place.

This is a simple application of the arguments of \S\ref{sub:elementary}. The assignment $t \mapsto S(F)_t$ is a functor from $\mc{Z}^2(\Gamma,Z(\hat G))$ to the category of topological groups. Therefore, the argument of \S\ref{sub:elementary} gives an extension of this functor to $\mc{Z}^2(\Gamma,\hat T \to \hat T_\tx{ad})$. In addition, the map $S(F)_t \to G(F)_t$ is a natural transformation between functors on $\mc{Z}^2(\Gamma,Z(\hat G))$, and the same arguments extend it to a natural transformation between functors on $\mc{Z}^2(\Gamma,\hat T \to \hat T_\tx{ad})$.

The same arguments also make sense of $S(F)_{t,\infty}$ and construct $S(F)_{t,\infty} \to G(F)_{t,\infty}$ when $t$ is a finite-order character of $\bar\pi_1(G)$; one just replaced $\mc{Z}^2(\Gamma,\hat T \to \hat T_\tx{ad})$ with $\mc{Z}^2_\infty(\Gamma,\hat T \to \hat T_\tx{ad})$.
\end{cns} 

\begin{fct} \label{fct:maxtori}
Let $S \subset G$ be a maximal torus. 
\begin{enumerate}
	\item The homomorphism $S(F)_\infty^G \to G(F)_\infty$ of Construction \ref{cns:maxtori} identifies $S(F)_\infty^G$ with the pull-back of $S(F) \to G(F) \from G(F)_\infty$.
	\item For any $t \in Z^2(\Gamma,Z(\hat G))$, the homomorphism $S(F)_{t,*} \to G(F)_{t,*}$ identifies $S(F)_{t,*}$ with the pull-back of $S(F) \to G(F) \from G(F)_{t,*}$.
	\item For any $t \in Z^2(\Gamma,\hat T \to \hat T_\tx{ad})$, the homomorphism $S(F)_{t,\infty} \to G(F)_{t,\infty}$ identifies $S(F)_{t,\infty}$ with the pull-back of $S(F) \to G(F) \from G(F)_{t,\infty}$.
\end{enumerate}
\end{fct}

We now turn to the dual side. Let $S$ be an $F$-torus equipped with a $\Gamma$-stable $G(\bar F)$-conjugacy class of embeddings $S \to G$ whose images are maximal tori. As discussed in \cite[\S5.1]{KalRSP} this provides further structures, such as a $\Gamma$-invariant subset $R(S,G) \subset X^*(S)$ and a $\Gamma$-stable $\hat G$-conjugacy class of embeddings $\hat S \to \hat G$. In particular, one has a natural $\Gamma$-equivariant embedding $Z(\hat G) \to \hat S$.

\begin{cns} \label{cns:l-s-t}
Let $t \in Z^2(\Gamma,\hat T \to \hat T_\tx{ad})$ and $z' \in Z^2(\Gamma,\hat S)$. We construct an $L$-group $^LS_{z't}$ using the arguments of \S\ref{sub:elementary} as follows. When $t \in Z^2(\Gamma,Z(\hat G))$, we use the natural inclusion $Z(\hat G) \to \hat S$ to form the product $z' \cdot t$ and define $^LS_{z't}=\hat S\boxtimes_{z't}\Gamma$. This is a functor on $\mc{Z}^2(\Gamma,Z(\hat G))$, which acts on morphisms via \eqref{eq:iso-l-s}. The arguments of \S\ref{sub:elementary} extend this functor to $\mc{Z}^2(\Gamma,\hat T \to \hat T_\tx{ad})$. 

Explicitly, write $t=(z,c)$. For every $x \in C^1(\Gamma,\hat T)$ with $\bar x=c$ we have $z\partial x^{-1} \in Z^2(\Gamma,Z(\hat G))$ and can form $\hat S \boxtimes_{zz'\partial x^{-1}}\Gamma$. Given another $x'$ we have the isomorphism
\[ \hat S \boxtimes_{z'z\partial x^{-1}}\Gamma \to \hat S \boxtimes_{z'z\partial x'^{-1}}\Gamma,\qquad s \boxtimes\sigma \mapsto sx'(\sigma)x^{-1}(\sigma)\boxtimes\sigma.\]
Define $^LS_{tz'}$ as the limit of this system.
\end{cns}

\begin{pro}
There is a natural correspondence between genuine characters of $S(F)_{z't}$ and $\hat S$-conjugacy classes of $L$-parameters valued in $^LS_{z't}$.
\end{pro}
\begin{proof}
According to the arguments of \S\ref{sub:elementary} it is enough to treat the case $t \in Z^2(\Gamma,Z(\hat G))$, which is however just Corollary \ref{cor:llc-s}(3).
\end{proof}

\begin{fct} \label{fct:l-s-t}
Assume given $z' \in Z^2(\Gamma,\hat S)$ and $t' \in Z^2(\Gamma,\hat T \to \hat T_\tx{ad})$, and an $L$-embedding $^LS_{z't'} \to {^LG}_{t'}$. Then, for any $t \in Z^2(\Gamma,\hat T \to \hat T_\tx{ad})$ the same map induces an $L$-embedding $^LS_{z't't} \to {^LG}_{t't}$.
\end{fct}
\begin{proof}
When $t,t' \in Z^2(\Gamma,Z(\hat G))$ the claim is evident. The general case follows from the arguments in \S\ref{sub:elementary}.
\end{proof}

\begin{pro} \label{pro:lemb-comp}
Fix $\hat\jmath : \hat S \to \hat G$ and let $\mc{S}=\{\gamma \in {^LG_t}\,|\,\gamma \hat\jmath(s) \gamma^{-1} = \hat\jmath(\sigma_\gamma(s)) \}$, where $\sigma_\gamma \in \Gamma$ denotes the image of $\gamma$ under the projection $^LG_t \to \Gamma$.
\begin{enumerate}
	\item For any $z' \in Z^2(\Gamma,\hat S)$ and any $L$-embedding $^Lj : {^LS}_{z't} \to {^LG}_t$ extending $\hat\jmath$, the image of $^Lj$ equals $\mc{S}$, and $^Lj$ is an isomorphism onto $\mc{S}$.
	\item Let $z'_1,z'_2 \in Z^2(\Gamma,\hat S)$ and consider $L$-embeddings $^Lj_i : {^LS}_{z'_it} \to {^LG}_t$ extending $\hat\jmath$. Then $^Lj_2 = {^Lj_1} \circ {^L\eta}$ for an $L$-isomorphism $^L\eta : {^LS_2} \to {^LS_1}$ extending the identity on $\hat S$. The resulting bijection between genuine characters is given by
	\[ \Pi(S_{z'_2t,n}) \to \Pi(S_{z'_1t,n}),\qquad \chi \mapsto \chi \cdot \eta, \]
	where $\eta$ is a genuine character of $S(F)_{z_1'/z_2'}$ depending only on the $\hat G$-conjugacy classes of $^Lj_i$.
\end{enumerate}
\end{pro}
\begin{proof}
As in Construction \ref{cns:l-s-t} we write $t=(z,c)$ and pick $x \in C^1(\Gamma,\hat T)$ with $\bar x=c$, which allows us to represent $^LS_{z't}=\hat S \boxtimes_{zz'\partial x^{-1}}\Gamma$ and $^LG_t=\hat G \boxtimes_{z\partial x^{-1}}\Gamma$.

(1) Let $\sigma \in \Gamma$ and let $\gamma = {^Lj}(1 \boxtimes \sigma)$. The multiplicativity of $^Lj$ implies $\gamma \in \mc{S}$. It follows that the image of $^Lj$ is a full subgroup of $\mc{S}$. But this image also contains the kernel of the projection $\mc{S} \to \Gamma$, so it equals all of $\mc{S}$. The second statement follows from the open mapping theorem \cite[Theorem 5.29]{HRv1e2}.

(2) According to (1) we can form $^Lj_1^{-1} \circ {^Lj_2}$, which is an $L$-isomorphism $^L\eta : {^LS_{z_2't}} \to {^LS_{z_1't}}$ and restricts to the identity on $\hat S$. This implies the existence of $c \in C^1(\Gamma,\hat S)$ such that $^L\eta(s \boxtimes \sigma) = sc(\sigma)\boxtimes\sigma$. This implies $\partial c = z_2'/z_1'$ and thus $c$ gives a genuine character $\eta$ of $S(F)_{z_1'/z_2'}$.

If we conjugate $\hat\jmath$, $^Lj_1$ and $^Lj_2$, by the same element of $\hat G$, then $c$ remains unchanged. If we keep $\hat\jmath$ fixed, the requirement that $^Lj_i$ both restrict to $\hat\jmath$ then implies that each of $^Lj_i$ can be conjugated only by an element of the image of $\hat\jmath$. Such conjugation changes $c$ by an element of $B^1(\Gamma,\hat S)$. Therefore $\eta$ depends only on the $\hat G$-conjugacy classes of $^Lj_i$.
\end{proof}

\subsection{Double covers via finite admissible sets} \label{sub:admset}

Let $S$ be an $F$-torus equipped with a $\Sigma$-invariant map $R \to X^*(S)$, where $\Sigma=\Gamma \times \{\pm1\}$ and $R$ is a finite $\Sigma$-set on which $\{\pm1\} \subset \Sigma$ has no fixed points. In \cite{KalDC} we associated to this data a double cover $S(F)_R$ of $S(F)$. In this subsection we will show that this construction fits in the general framework developed in this paper. We will also generalize this construction to quasi-split connected reductive groups.

For any gauge $p : R \to \{\pm1\}$ we have the Tits cocycles $t_p \in Z^2(\Gamma,\hat S[2])$, cf. \cite[Definition 3.8]{KalDC}. Pushing out the extension \eqref{eq:uni} along $t_p : \tilde\pi_1(S) \to \{\pm1\}$ we obtain the double cover $S(F)_{t_p,2}$. For two gauges $p,q$ we have the 1-cochain $s_{q/p} \in C^1(\Gamma,\hat S[2])$ with $\partial s_{q/p} = t_q/t_p$, and hence the isomorphism $\xi_{s_{q/p}} : S(F)_{t_p,2} \to S(F)_{t_q,2}$. These isomorphisms form a compatible system. 

\begin{pro} \label{pro:dc_comp}
The limit of that system is naturally identified with the cover $S(F)_R$ constructed in \cite{KalDC}. 
\end{pro}
\begin{proof}
Fix a gauge $p$. It is enough to identify $\tilde H^1_u(W_F,\hat S)_{t_p}$ with the set of unitary genuine characters of $S(F)_R$ as topological spaces that are torsors under $H^1_u(W_F,\hat S)$. Pontryagin duality implies that unitary $\chi$-data $(\chi_\alpha)_{\alpha \in R}$ exist. The resulting genuine character $\chi$ of $S(F)_R$ is then unitary. On the other hand, its parameter $r_p$ is by construction an element of $\tilde H^1_u(W_F,\hat S)_{t_p}$. Any other unitary genuine character of $S(F)_R$ is the product of $\chi$ with a unitary character of $S(F)$. The correspondence for $S(F)_R$ established in \cite[Theorem 3.15]{KalDC} and its multiplicativity now establish a bijection between the set of unitary genuine characters of $S(F)_R$ and the set $\tilde H^1_u(W_F,\hat S)_{t_p}$. Both sets are topologized as torsors under the topological groups $\tx{Hom}_\tx{cts}(S(F),\mb{S}^1)$ and $H^1_u(W_F,\hat S)$, respectively. Since these topological groups are canonically isomorphic and the bijection we have established respects this isomorphism, the claim is proved.
\end{proof}

\begin{rem}
The above proposition provides an explicit description of the double cover $S(F)_{t,2}$ associated to a certain $t \in Z^2(\Gamma,\hat S)$, namely the Tits cocycle associated to a finite admissible $\Sigma$-set $R \to X^*(S)$ and a gauge $p$. It would be desirable to obtain explicit description of other covers $S(F)_{t,n}$.
\end{rem}

\begin{cns} \label{cns:lemb-s-t}
Let $t \in Z^2(\Gamma,\hat T \to \hat T_\tx{ad})$. We define $^LS_{t\pm}$ by presenting $^LS_\pm={^LS}_{t_p}$ and using Construction \ref{cns:l-s-t} to form $^LS_{t\cdot t_p}$. The $L$-embedding $^LS_\pm \to {^LG}$ provides via Fact \ref{fct:l-s-t} an $L$-embedding $^LS_{t\pm} \to {^LG}_t$.
\end{cns}

We now generalize this construction to a quasi-split connected reductive group $G$. Let $T$ be the universal maximal torus. Consider a finite admissible $\Sigma$-set $R \to X^*(T)$. We can compose it with the map $X^*(T) \to X^*(T_\tx{sc})$. In this way we obtain the double covers $T(F)_R$ and $T_\tx{sc}(F)_R$. Assume given a genuine sign character $\alpha : T_\tx{sc}(F)_R \to \{\pm1\}$.

\begin{cns} \label{cns:hdc}
Given $R$ and $\alpha$ we obtain a double cover $G(F)_{R,\alpha}$ as follows. Let $p : R \to \{\pm1\}$ be a gauge and let $z_p \in Z^2(\Gamma,\hat T[2])$ be the Tits cocycle associated to $R$ and $p$. Let $\bar z_p \in Z^2(\Gamma,\hat T_\tx{ad}[2])$ be its image. Let $c_p \in C^1(W_F,\hat T_\tx{ad})$ be the parameter of the genuine character $\alpha$. Thus $\partial c_p = \bar z_p$ and $c_p^2 \in B^1(W_F,\hat T_\tx{ad})$. According to Lemma \ref{lem:[n]}(2)(3), the class of $[c_p]$ gives a well-defined element, which we also denote by $[c_p]$, of $\tilde H^1(\Gamma,\hat T_\tx{ad}[2])$. Then $t_p := (z_p,[c_p]) \in Z^2(\Gamma,\hat T[2] \to \hat T_\tx{ad}[2])/B^1(\Gamma,\hat T_\tx{ad}[2])$ depends only on $R$, $\alpha$, and $p$. We have the associated double cover $G(F)_{t_p}$.

Given another gauge $q$ we have $z_q/z_p=\partial s_{q/p}$ and $[c_q] = [s_{q/p}] \cdot [c_p]$. Thus $[s_{q/p}] \in \tilde H^1(\Gamma,\hat T)_{t_q/t_p}[2]$. From Construction \ref{cns:iso-h} we obtain the isomorphism $\xi_{q,p} : G(F)_{t_p} \to G(F)_{t_q}$. In this way we obtain a compatible system indexed by the set of gauges on $R$. We define $G(F)_{R,\alpha}$ to be the limit of that system.

Equivalently (cf. Construction \ref{cns:h-univ}, Fact \ref{fct:gtn}(2)), we define $G(F)_{R,\alpha}$ as 
\[ \tx{cok}(T_\tx{sc}(F) \to G_\tx{sc}(F) \rtimes T(F)_R), \]
where $T(F)_R$ acts on $G_\tx{sc}(F)$ through the projection $T(F)_R \to T(F)$ and the conjugation action of $T(F)$ on $G_\tx{sc}(F)$ via an arbitrary admissible embedding $j : T \to G$, and $T_\tx{sc}(F) \to T(F)_R$ is the splitting obtained by composing natural map $T_\tx{sc}(F)_R \to T(F)_R$ coming from Corollary \ref{cor:llc-s}(5) and the splitting $T_\tx{sc}(F) \to T_\tx{sc}(F)_R$ coming $\alpha$ via Lemma \ref{lem:dc_split}.
\end{cns}

\begin{cns} \label{cns:hdc_auto}
Let $\theta$ be an $F$-automorphism of $G$. It induces an $F$\-automorphism of the universal maximal torus $T$, which we also denote by $\theta$. Let $\theta_R$ be an automorphism of the $\Sigma$-set $R$ so that the map $R \to X^*(T)$ is $(\theta_R,\theta)$-equivariant. This induces an automorphism of the covers $T(F)_\R$ and $T_\tx{sc}(F)_R$, which we also denote by $\theta$, cf. \cite[Example 5.3]{KalDC}. Assume that $\alpha$ is fixed by $\theta$. Then the splitting $T_\tx{sc}(F) \to T(F)_R$ is $\theta$-equivariant, and the identification $G(F)_{R,\alpha} = \tx{cok}(T_\tx{sc}(F) \to G_\tx{sc}(F) \rtimes T(F)_R)$ endows $G(F)_{R,\alpha}$ with an automorphism lifting $\theta$.
\end{cns}

\begin{rem}
Consider a maximal torus $S \subset G$ and let $S_{R,\alpha}$ be the pull-back of the cover $G(F)_{R,\alpha}$ under $S(F) \to G(F)$. The character $t_p : \bar\pi_1(G) \to \{\pm1\}$ that appears in Construction \ref{cns:hdc} leads, via Construction \ref{cns:s-t}, to a cover $S(F)_{t_p}$. We know from Fact \ref{fct:maxtori} that there is a natural isomorphism $S(F)_{R,\alpha} \to S(F)_{t_p}$. It would be desirable to have an explicit description of this double cover, in the spirit of the construction of \cite{KalDC}. 
\end{rem}

\subsection{Comparison with the constructions of Adams-Vogan for $F=\R$} \label{sub:avcomp}
 
Consider the case $F=\R$. In \cite[(7.11)(a)]{AV92}, Adams-Vogan define a finitely generated discrete abelian group $\pi_1(G)(\R)$ as $\pi_1(G(\C))/(1+\theta)\pi_1(G(\C))$. Here $\pi_1(G(\C))$ is the topological fundamental group for the analytic topology on $G(\C)$ based at the identity element. It is known that $\pi_1(G(\C))$ coincides with Borovoi's algebraic fundamental group $\pi_1(G)$. The automorphism $\theta$ is the Cartan involution, which coincides on $\pi_1(G(\C))$ with the action of complex conjugation. However, due to the conventions of \cite{AV92}, the action of complex conjugation is not the same as the one on Borovoi's fundamental group, but instead differs from it by multiplication by $-1$. This is owed to the use of the normalization $2\pi i$, cf. \cite[(5.1)(b)]{AV92}. Therefore, in terms of Borovoi's fundamental group, one has the identity $\pi_1(G)(\R)=\pi_1(G)_\Gamma$. In turn, this leads to the identity
\begin{equation} \label{eq:avfund}
\pi_1(G)(\R)=X^*(Z(\hat G)^\Gamma).
\end{equation}

\begin{pro} \label{pro:avcomp_pi1}
There exists a natural injective group homomorphism 
\[ \pi_1(G)(\R) \to \tilde\pi_1(G) \] 
whose composition with the projection to $\pi_0(\tilde\pi_1(G))$ remains injective and identifies $\pi_0(\tilde\pi_1(G))$ with the profinite completion of the discrete group $\pi_1(G)(\R)$.
\end{pro}
\begin{proof}
We will use Lemma \ref{lem:diag} with the diagonalizable group $D=Z(\hat G)^\Gamma$. On the one hand, we have the identity \eqref{eq:avfund}. On the other hand, for any $\Gamma$-module $M$ evaluation at $(\sigma,\sigma)$ induces a group isomorphism $Z^2(\Gamma,M) \to M^\Gamma$. Applying this to $M=Z(\hat G)$ leads to $D=Z(\hat G)^\Gamma=Z^2(\Gamma,Z(\hat G))$, and hence $D_\tx{disc}=\tilde\pi_1(G)^*$. Then Lemma \ref{lem:diag} provides the homomorphism $\pi_1(G)(\R) \to \tilde\pi_1(G)$ with the desired properties. Note that injectivity follows from the fact that $\pi_1(G)(\R)$ is residually finite, being finitely generated abelian.
\end{proof}

It is worth recording the following statement, which follows formally from the above proposition, and was also used in its proof.

\begin{cor} \label{cor:avcomp_pi1}
The map $\pi_1(G)(\R) \to \tilde\pi_1(G)$ induces a bijection between the groups of characters of \emph{finite order} on both sides.
\end{cor}

In addition to $\pi_1(G)(\R)$, in \cite[(5.3)(c),(7.11)(c)]{AV92} Adams and Vogan define an extension 
\[ 1 \to \pi_1(G)(\R) \to G(\R)_\tx{AV} \to G(\R) \to 1, \]
where $G(\R)_\tx{AV}$ is denoted by $G(\R)^\sim$ in loc. cit. We will now compare $G(\R)_\tx{AV}$ with $G(\R)_\infty$.

Consider first the case of a torus $G=T$. Choose a finite order character $t : \tilde\pi_1(T) \to \mu_n(\C)$ and pull it back to a character of $\pi_1(T)(\R)$. Pushing out $T(\R)_\infty$ we obtain the $n$-fold cover $T(\R)_{t,n}$, while pushing out $T(\R)_\tx{AV}$ we obtain another $n$-fold cover, which we denote by $T(\R)_{\tx{AV},t,n}$.

\begin{pro} \label{pro:avcomp_torus1}
We have a natural isomorphism $T(\R)_{t,n} \to T(\R)_{\tx{AV},t,n}$ as extensions of $T(\R)$ by $\mu_n(\C)$, as well as between the group $^LT_t$ of \eqref{eq:s-l} and the $E$-group determined by $t$ as in \cite[Definition 5.9]{AV92}, and these two isomorphisms are compatible with the local correspondences.
\end{pro}
\begin{proof}
The identification between $^LT_t$ and the $E$-group determined by $t$ follows at once from the definitions. We have used here the identification $Z^2(\Gamma,\hat T)=Z^0(\Gamma,\hat T)$ and hence identified $t$ with $t(\sigma,\sigma)$. 

To identify $T(\R)_{t,n}$ and $T(\R)_{\tx{AV},t,n}$ it is enough to identify their groups of unitary characters in a way compatible with the subgroups of unitary characters of $T(\R)$, and the quotients $\Z/n\Z$. For this it is enough to identify the two cosets of genuine characters together with the action of the group of unitary characters of $T(\R)$. According to Corollary \ref{cor:llc-s}(4), the coset of unitary characters is identified with the set of bounded parameters $W_\R \to {^LS_t}$, and the action of the group of unitary characters of $T(\R)$ is translated to the action of $H^1_u(W_\R,\hat T)$ by multiplication. On the other hand, the same identification for $T(\R)_{\tx{AV},t,n}$ follows from \cite[Theorem 5.11]{AV92}.
\end{proof}

\begin{cor} \label{cor:avcomp_torus2}
There exists a natural embedding $T(\R)_\tx{AV} \to T(\R)_\infty$ which covers the identity of $T(\R)$ and induces the embedding $\pi_1(T)(\R) \to \tilde\pi_1(T)$.
\end{cor}
\begin{proof}
Let $T(\R)_\tx{AV}^1$ denote the fiber product over $T(\R)$ of the covers $T(\R)_{\tx{AV},t,n}$ for all $n \in \N$ and all $t \in \pi_1(T)(\R) \to \mu_n(\C)$. Thus $T(\R)_\tx{AV}^1$ is an extension of $T(\R)$ by $\prod_{t,n}\mu_n(\C)$ and equals the pushout of $T(\R)_\tx{AV}$ under the map $\pi_1(T)(\R) \to \prod_{t,n}\mu_n(\C)$ that sends $x \in \pi_1(T)(\R)$ to $(t(x))_{t,n}$.

We perform the same construction with $T(\R)_{t,n}$ for all $n \in \N$ and all $t : \tilde\pi_1(T) \to \mu_n(\C)$ and call the result $T(\R)_\infty^1$. According to Corollary \ref{cor:avcomp_pi1} both $T(\R)_\tx{AV}^1$ and $T(\R)_\infty^1$ are extensions of $T(\R)$ by the same compact group $K=\prod_{t,n}\mu_n(\C)$, and according to Proposition \ref{pro:avcomp_torus1} they are naturally isomorphic. Thus, the pushout of $T(\R)_\tx{AV}$ under $\pi_1(T)(\R) \to K$ is naturally isomorphic to the pushout of $T(\R)_\infty$ under $\tilde\pi_1(T) \to K$. This implies that $T(\R)_\infty$ is naturally isomorphic to the pushout of $T(\R)_\tx{AV}$ under $\pi_1(T)(\R) \to \tilde\pi_1(T)$.
\end{proof}

Consider now a quasi-split reductive $\R$-group $G$.

\begin{cor} \label{cor:avcomp}
There is a natural embedding $G(\R)_\tx{AV} \to G(\R)_\infty$ which covers the identity of $G(\R)$ and induces the embedding $\pi_1(G)(\R) \to \tilde\pi_1(G)$.
\end{cor}
\begin{proof}
Let $\tilde G$ and $\tilde T$ be universal covers of the complex Lie groups $G$ and $T$. We have the following commutative diagram with exact rows and columns
\[ \xymatrix{
&1\ar[d]&1\ar[d]&1\ar[d]\\
1\ar[r]&\pi_1(T_\tx{sc})\ar[d]\ar[r]&\pi_1(T)\ar[r]\ar[d]&\pi_1(G)\ar[d]\ar[r]&1\\
1\ar[r]&\tilde T_\tx{sc}\ar[d]\ar[r]&G_\tx{sc} \rtimes \tilde T\ar[d]\ar[r]&\tilde G\ar[d]\ar[r]&1\\
1\ar[r]&T_\tx{sc}\ar[d]\ar[r]&G_\tx{sc} \rtimes T\ar[d]\ar[r]&G\ar[d]\ar[r]&1\\
&1&1&1
}
\]
For the exactness at the middle term, note that there a no non-trivial continuous homomorphisms $T_\tx{sc} \to \pi_1(G)$, since $T_\tx{sc}$ is connected and $\pi_1(G)$ is discrete.

Write $\tilde T_\tx{sc}(\R)$, $\tilde T(\R)$, and $\tilde G(\R)$, for the inverse images of $T_\tx{sc}(\R),$ $T(\R)$, and $G(\R)$, respectively. Since $G$ is quasi-split, $G_\tx{sc}(\R) \rtimes T(\R) \to G(\R)$ is surjective, and this implies the exactness of 
\[ 1 \to \tilde T_\tx{sc}(\R) \to G_\tx{sc}(\R) \rtimes \tilde T(\R) \to \tilde G(\R), \] 
which in turn implies the exactness of
\[ 1 \to T_\tx{sc}(\R) \to G_\tx{sc}(\R) \rtimes \frac{\tilde T(\R)}{\pi_1(T_\tx{sc})+(1-\sigma)\pi_1(T)} \to \frac{\tilde G(\R)}{(1-\sigma)\pi_1(G)} \to 1. \] 
The right term of the above sequence is by definition $G(\R)_\tx{AV}$. The middle term is the pushout of $T(\R)_\tx{AV}$ along $\pi_1(T)(\R) \to \pi_1(G)(\R)$, which we shall denote by $T(\R)_\tx{AV}^G$. Corollary \ref{cor:avcomp_torus2} gives a natural embedding $T(\R)_\tx{AV}^G \to T(\R)_\infty^G$, where we recall that $T(\R)_\infty^G$ is the pushout of $T(\R)_\infty$ along $\tilde\pi_1(T) \to \tilde\pi_1(G)$. The embedding $T(\R)_\tx{AV}^G \to T(\R)_\infty^G$ is compatible with the embeddings of $T_\tx{sc}(\R)$ into both sides.
This provides the desired embedding $G(\R)_\tx{AV} \to G(\R)_\infty$.
\end{proof}

\begin{cor}
There is a natural isomorphism from $G(\R)_\infty$ to the pushout of $G(\R)_\tx{AV}$ along the inclusion $\pi_1(G)(\R) \to \tilde\pi_1(G)$.
\end{cor}

\section{Covers arising from certain subgroups of the $L$-group} \label{sec:l-covers}

Let $G$ be a connected reductive $F$-group, $\hat G$ its dual group (taken over an arbitrary algebraically closed field). In this section we will show that a subgroup $\mc{H} \subset {^LG}$ with certain properties leads naturally to a double cover $H(F)_\pm$ of $H(F)$ together with an isomorphism $^LH_\pm \to \mc{H}$. In order to obtain a closed theory, we will also consider more generally a subgroup $\mc{H} \subset {^LG_t}$, for any $t : \bar\pi_1(G) \to \mu_n$ and will show that this leads to a cover $H(F)_{t\pm}$ of $H$ and an isomorphism $^LH_{t\pm} \to \mc{H}$.

\subsection{The double cover associated to $\mc{H} \subset {^LG}$} \label{sub:l1-covers}

Let $^LG = \hat G \rtimes \Gamma$ the Galois form of the $L$-group.  We consider a subgroup $\mc{H} \subset {^LG}$, given up to conjugation by $\hat G$, satisfying the following.
\begin{prp}\ \\[-15pt] \label{prp:hdc}
\begin{enumerate}
	\item $\mc{H}$ is \emph{full}, i.e. its image under $^LG \to \Gamma$ is all of $\Gamma$;
	\item $\hat H = \mc{H} \cap \hat G$ is a connected reductive subgroup of the same rank as $\hat G$.
\end{enumerate}	
\end{prp}

We choose a $\Gamma$-pinning $(\hat T,\hat B,\{X_{\hat\alpha}\})$ of $\hat G$ and conjugate $\mc{H}$ by an element of $\hat G$ to ensure $\hat T \subset \hat H$. Then $\hat B^H := (\hat H \cap \hat B)^\circ$ is a Borel subgroup of $\hat H$ containing $\hat T$. For $\sigma \in \Gamma$, choose an arbitrary lift $h_\sigma \in \mc{H}$ that normalizes the pair $(\hat T,\hat B^H)$. Then $h_\sigma$ is well-defied up to left multiplication by $\hat T$ and therefore the automorphism $\sigma_H := \tx{Ad}(h_\sigma)$ of $\hat T$ does not depend on the choice of $h_\sigma$. We write $\hat T^H$ for the complex torus $\hat T$ equipped with $\Gamma$-action given by $\sigma \mapsto \sigma_H$.

We have $\sigma_H = \sigma_{H,G} \circ \sigma_G$, where $\sigma_G$ is the automorphism of $\hat T$ induced by $\tx{Ad}(1 \rtimes \sigma) \in \hat G \rtimes \Gamma$, and $\sigma_{H,G} \in \Omega(\hat T,\hat G)$. Let $n(\sigma_{H,G}) \in N(\hat T,\hat G)$ be the Tits lift of $\sigma_{H,G}$ with respect to the fixed pinning. Then $n(\sigma_{H,G}) \rtimes \sigma_G \in \hat G \rtimes \Gamma$ normalizes $\hat T$ and acts on it as $\sigma_H$. It follows that $n(\sigma_{H,G}) \rtimes \sigma_G$ and $h_\sigma$ are elements of $^LG$ that differ by left multiplication by an element of $\hat T$. Therefore
\begin{equation} \label{eq:tits-splitting}
\sigma \mapsto n(\sigma_{H,G}) \rtimes \sigma_G
\end{equation}
is a (usually non-homomorphic) section of the projection $\mc{H} \to \Gamma$. 

Recall \cite[Proposition 2.7.11]{BTBOOK} that the pinning of $\hat G$ determines a weak Chevalley system, i.e. in each root space $\tx{Lie}(\hat G)_\alpha$ a pair of non-zero elements that differ by a sign. We will say that a pinning of $\hat H$ is \emph{adapted} to the pinning of $\hat G$ if the associated Borel pair is $(\hat T,\hat B^H)$ and the associated simple root vectors belong to the weak Chevalley system. Thus, every two adapted pinnings of $\hat H$ differ from each other by changing signs in the various simple root vectors.

Let $\hat T^H_\tx{ad}=\hat T^H/Z(\hat H)$.

\begin{lem} \label{lem:modified-splitting}
Fix an adapted pinning of $\hat H$. There exists a unique element $c \in C^1(\Gamma,\hat T^H_\tx{ad}[2])$ such that the action of $\Gamma$ on $\hat H$ obtained from the section
\begin{equation} \label{eq:modified-splitting}
c^{-1}\cdot s : \Gamma \to \mc{H}/Z(\hat H),\qquad \sigma \mapsto c(\sigma)^{-1}n(\sigma_{H,G}) \rtimes \sigma_G
\end{equation}
preserves that pinning. The class of $c$ modulo $B^1(\Gamma,\hat T^H_\tx{ad}[2])$ is independent of the choice of adapted pinning.
\end{lem}
\begin{proof}
It is clear that \eqref{eq:tits-splitting} already preserves the Borel pair $(\hat T^H,\hat B^H)$. Therefore there exists a necessarily unique $c \in C^1(\Gamma,\hat T^H_\tx{ad})$ so that \eqref{eq:modified-splitting} preserves the chosen pinning of $\hat H$. Write $X_\alpha^H$ for the element of $\tx{Lie}(\hat H)_\alpha$ that belongs to the fixed adapted pinning. Since the Tits group preserves the weak Chevalley system (\cite[Definition 2.7.8]{BTBOOK}) of $\hat G$, the element $c$ is uniquely determined by the rule
\[ c(\sigma) = \prod_{\alpha} \check\omega_\alpha(\epsilon_\alpha), \]
where the product runs over the simple roots $\alpha \in \Delta(\hat T,\hat B^H)$, $\check\omega_\alpha$ is the corresponding fundamental coweight, and $\epsilon_\alpha$ is the scalar by which the two elements $X^H_\alpha$ and $n(\sigma_{H,G}) \rtimes \sigma_G \cdot X^H_{\sigma_H^{-1}\alpha}$ of $\tx{Lie}(\hat H)$ differ. Since both of these elements belong to the weak Chevalley system, we have $\epsilon_\alpha \in \{\pm 1\}$, and therefore $c \in C^1(\Gamma,\hat T^H_\tx{ad}[2])$. 

Another choice of an adapted pinning differs from the fixed one by $\tx{Ad}(t)$, where $t \in T^H_\tx{ad}[2]$ is of the form
\[ t = \prod_{\alpha} \check\omega_\alpha(\epsilon_\alpha) \]
for some $\epsilon_\alpha \in \{\pm 1\}$. For this choice, the corresponding splitting \eqref{eq:modified-splitting} would be given by $t \cdot c \cdot s \cdot t^{-1}$, which equals $t \cdot \sigma_H(t)^{-1} \cdot c \cdot s$, and we see that $c$ and $t \cdot \sigma_H(t)^{-1}\cdot c$ give the same class modulo $B^1(\Gamma,\hat T^H_\tx{ad}[2])$. 
\end{proof}

\begin{rem}
The inverse in \eqref{eq:modified-splitting} is just cosmetic, because $c^{-1}=c$, but we have kept it in order to have the formulas consistent.
\end{rem}

\begin{rem}
The value of $c$ can be computed explicitly using work of Kottwitz, cf. \cite[\S4]{Cas20}.
\end{rem}

The 2-cocycle of the Tits section \eqref{eq:tits-splitting} is an element $z \in Z^2(\Gamma,\hat T^H[2])$. Since the section \eqref{eq:modified-splitting} preserves a pinning, it is a group homomorphism. Therefore, class $[c] \in C^1(\Gamma,\hat T^H_\tx{ad}[2])/B^1(\Gamma,\hat T^H_\tx{ad}[2])$ provided by Lemma \ref{lem:modified-splitting} satisfies $\partial c=\bar z$. We thus obtain an element 
\begin{equation} \label{eq:h-tits}
t=x_{H,G}=(z,[c]) \in Z^2(\Gamma,\hat T^H[2] \to \hat T^H_\tx{ad}[2])/B^1(\Gamma,\hat T^H_\tx{ad}[2]).
\end{equation}
Let $H$ be the unique quasi-split $F$-group that is dual to $\hat H$ and whose rational structure is given by the homomorphism $\Gamma \to \tx{Out}(\hat H)=\tx{Out}(H)$ determined by the extension $\mc{H}$. Then $\hat T^H$ is identified with the dual of the universal torus $T^H$ of $H$ in such a way that the positive chamber in $X^*(T^H)$ equals the positive chamber in $X_*(\hat T^H)$ specified by $\hat B^H$. Via this identification, Construction \ref{cns:h-univ} provides a double cover $H(F)_\pm$ of $H(F)$.

The dual group of $H$ is naturally identified with $\hat H$, with $\Gamma$ acting on $\hat H$ via the splitting \eqref{eq:modified-splitting}. The $L$-group of $H(F)_\pm$ is thus identified via \eqref{eq:h-l2} with $\hat H \boxtimes_{z \cdot \partial x^{-1}} \Gamma$ for any $x \in C^1(\Gamma,\hat T^H)$ satisfying $[\bar x]=[c]$. Then

\begin{equation} \label{eq:liso}
^LH_\pm = \hat H \boxtimes_{z \cdot \partial x^{-1}} \Gamma \to \mc{H},\qquad h \boxtimes \sigma \mapsto h x^{-1}(\sigma) n(\sigma_{H,G})\rtimes \sigma_G	
\end{equation}
is an $L$-isomorphism, independent of the choice of $x$. Therefore, it induces an $L$-embedding $^LH_\pm \to {^LG}$.

\begin{rem}
There is an equivalent way to describe the double cover of $H(F)$ using \S\ref{sub:admset}. We have the admissible set $R^\vee(\hat T,\hat G) \subset X_*(\hat T^H) = X^*(T^H)$. It gives rise to a double cover $T^H(F)_\pm$. Note that it is the same double cover as for the admissible set $R^\vee(\hat T,\hat G) \sm R^\vee(\hat T,\hat H)$, because all elements of $R^\vee(\hat T,\hat H)$ are asymmetric.

The class $[c] \in C^1(\Gamma,\hat T^H_\tx{ad}[2])/B^1(\Gamma,\hat T^H_\tx{ad}[2])$ of Lemma \ref{lem:modified-splitting} is the parameter of a genuine sign character $\alpha : T^H_\tx{sc}(F)_\pm \to \{\pm 1\}$. Construction \ref{cns:hdc} provides a double cover of $H(F)$. It is easily seen that this is the double cover $H(F)_\pm$. 
\end{rem}

The following lemma shows that the reductive group $H$, the double cover $H(F)_\pm$, and the $L$-embedding $^LH_\pm \to {^LG}$ \eqref{eq:liso}, depend only on the $\hat G$-conjugacy class of $\mc{H}$, but not on the individual representative of that conjugacy class, or on the chosen pinning of $\hat G$. Write temporarily $^L\xi$ for \eqref{eq:liso}.

\begin{lem}
Let $(H,{^L\xi})$ and $(H',{^L\xi'})$ be two data constructed by choosing different representatives in the $\hat G$-conjugacy class of $\mc{H}$ and different $\Gamma$-pinnings of $\hat G$. There exists an $F$-isomorphism $f : H \to H'$ such that ${^L\xi'}={^L\xi}\circ {^Lf_\pm}$ up to $\hat G$-conjugation, where $f_\pm : H(F)_\pm \to H'(F)_\pm$ is the lift of $f$ as in Construction \ref{cns:hdc_auto}. Moreover, $f$ is unique in $\tx{Out}(H)(F)=\tx{Out}(\hat H)(F)$ up to conjugation by an element of $N_{\hat G}(\hat H)$.
\end{lem}
\begin{proof}
It is clear that if we conjugate both $\mc{H}$ and the pinning by the same element of $\hat G$ nothing will change. Since all $\Gamma$-pinnings of $\hat G$ are conjugate under $\hat G$, even under $\hat G^\Gamma$ (\cite[Corollary 1.7]{Kot84}), it is enough to fix the pining of $\hat G$ and consider two subgroups $\mc{H}$ and $\mc{H}'$ of $^LG$ that are conjugate under $\hat G$, and that both contain $\hat T$. 

Write $\hat H=\mc{H} \cap \hat G$ and $\hat H'=\mc{H}' \cap \hat G$. In the following paragraphs we are going to specify $n \in \hat G$ and $y \in C^1(\Gamma,\hat T[2])$ with the following properties.
\begin{enumerate}
	\item $\tx{Ad}(n)$ conjugates $\mc{H}$ to $\mc{H}'$ in such a way that the isomorphism $\hat H \to \hat H'$ it provides identifies a pinning of $\hat H$ with a pinning of $\hat H'$, both adapted to the chosen pinning of $\hat G$.
\end{enumerate}
In particular, $\tx{Ad}(n)\circ\sigma_H=\sigma_{H'}$. As before we write $\hat T^H$ resp. $\hat T^{H'}$ for $\hat T$ equipped with the $\Gamma$-action given by $\sigma_H$ resp. $\sigma_{H'}$, as well as $z \in Z^2(\Gamma,\hat T^H[2])$ resp. $z' \in Z^2(\Gamma,\hat T^{H'}[2])$ for 2-cocycle of the section \eqref{eq:tits-splitting} with respect to $\mc{H}$ resp. $\mc{H'}$.
\begin{enumerate}[resume]
	\item The identity 
	\begin{equation} \label{eq:id1}
	\tx{Ad}(n)(\partial y \cdot z)=z'
	\end{equation}
	holds in $Z^2(\Gamma,\hat T^{H'}[2])$.
	\item We can choose $x \in C^1(\Gamma,\hat T^H)$ and $x' \in C^1(\Gamma,\hat T^{H'})$ such that $[\bar x]=[c]$, $[\bar x']=[c']$, and the identities 
	\begin{equation} \label{eq:id2}
	n \cdot x(\sigma)^{-1}n(\sigma_{H,G})\rtimes\sigma_G \cdot n^{-1} = x'(\sigma)^{-1}n(\sigma_{H',G}) \rtimes \sigma_G	
	\end{equation}
	in $^LG$ and
	\begin{equation} \label{eq:id3}
	\tx{Ad}(n)([y] \cdot [x])=[x']
	\end{equation}
	in $C^1(\Gamma,\hat T^{H'})/B^1(\Gamma,\hat T^{H'})$ hold.
\end{enumerate}
For a moment, let us assume that this has been done and see how it completes the proof of the lemma. First, the identification via $\tx{Ad}(n)$ of the pinned groups $\hat H$ and $\hat H'$ leads to an identification of the quasi-split groups $H$ and $H'$. Second, identity \eqref{eq:id3} implies $\tx{Ad}(n)([y] \cdot [c])=[c']$ in $C^1(\Gamma,\hat T_\tx{ad}^{H'}[2])/B^1(\Gamma,\hat T_\tx{ad}^{H'}[2])$, which together with identity \eqref{eq:id1} means that the isomorphism
\[ Z^2(\Gamma,\hat T^H[2] \to \hat T^H_\tx{ad}[2])/B^1(\Gamma,\hat T^H_\tx{ad}[2]) \to Z^2(\Gamma,\hat T^{H'}[2] \to \hat T^{H'}_\tx{ad}[2])/B^1(\Gamma,\hat T^{H'}_\tx{ad}[2]) \]
effected by $\tx{Ad}(n)$ identifies $\partial y \cdot t$ with $t'$, where $t=(z,[c])$ is computed relative to $\mc{H}$ and $t'=(z',[c'])$ relative to $\mc{H'}$. The identification of $\partial y \cdot t$ and $t'$ leads to an identification of the double covers $H(F)_\pm$ and $H'(F)_\pm$, namely via the isomorphism $\tx{Ad}(n)$ and the isomorphism $\xi_{[y]}$. Finally, we argue that these identifications are compatible with the $L$-embeddings \eqref{eq:liso}, in the sense that for suitable lifts $x \in C^1(\Gamma,\hat T^H)$ and $x' \in C^1(\Gamma,\hat T^{H'})$ of $c$ and $c'$, respectively, the diagram
\[ \xymatrix{
	\hat H \boxtimes_{z \cdot \partial x^{-1}}\ar[r]\ar[d]_{\tx{Ad}(n)\boxtimes\tx{id}}\Gamma&\mc{H}\ar[d]^{\tx{Ad}(n)}\\
	\hat H' \boxtimes_{z' \cdot \partial x'^{-1}}\ar[r]\Gamma&\mc{H}'
}
\]
commutes. Indeed, identities \eqref{eq:id1} and \eqref{eq:id3} show that the left vertical map is an isomorphism, while the commutativity of the diagram follows from the formulas for the top and bottom map given by \eqref{eq:liso} and identity \eqref{eq:id2}.

We now begin with the specification of $n$ and $y$ and the proof of the three points above. To specify $n$, let $g \in \hat G$ be such that $\mc{H}'=g\mc{H}g^{-1}$. Then $(\hat T,\hat B \cap \hat H)$ and $(g^{-1}\hat Tg,g^{-1}(\hat B \cap \hat H')g)$ are two Borel pairs in $\hat H$, so there exists $h \in \hat H$ such that $(\hat T,\hat B \cap \hat H)=((gh)^{-1}\hat T(gh),(gh)^{-1}(\hat B \cap \hat H')(gh))$. The element $n=gh$ lies in $N(\hat T,\hat G)$ and satisfies $\mc{H}'=n\mc{H}n^{-1}$. The isomorphism $\tx{Ad}(n) : \hat H \to \hat H'$ identifies the Borel pair $(\hat T,\hat B \cap \hat H)$ with the Borel pair $(\hat T,\hat B \cap \hat H')$, but possibly does not yet identify adapted pinnings. To achieve this, we are going to adjust $n$ modulo $\hat T$ as follows. Let $w$ denote the image of $n$ in the Weyl group $N(\hat T,\hat G)/\hat T$. We demand that $n=n(w)$ is the Tits lift of $w$ with respect to the pinning of $\hat G$. This completes the specification of $n$. Since $n$ belongs to the Tits group relative to the chosen pinning of $\hat G$, $\tx{Ad}(n)$ preserves the associated weak Chevalley system (\cite[Definition 2.7.8]{BTBOOK}). Therefore, the image under $\tx{Ad}(n)$ of a pinning of $\hat H$ that is adapted to that of $\hat G$ is a pinning of $\hat H'$ that is again adapted to that of $\hat G$. We have thus proved point (1) above.

To specify $y$, we let $p : R^\vee(\hat T,\hat G) \to \{\pm 1\}$ be the gauge corresponding to $\hat B$, and note that $\tx{Ad}(n) : \hat T \to \hat T$ does \emph{not} preserve $p$. Let $q : R^\vee(\hat T,\hat G) \to \{\pm 1\}$ be defined by $q(\alpha)=p(n\alpha)$. Then $\tx{Ad}(n)$ identifies the gauge $q$ on its source with the gauge $p$ on its target. We define $y=s_{q/p}$, cf. \cite[\S3.3]{KalDC}.

Before we prove (2) and (3) we discuss the uniqueness of $n$ and $y$, and hence of the isomorphism $f$ in the statement of the lemma. Both $n$ and $y$ depend only on one single choice, namely that of $w$. In turn, $w$ can be changed only by multiplication on the left by an element $u$ of $N(\hat T,\hat G)/\hat T$ that preserves the root system of $\hat H$ and the chamber in it given by $\hat B \cap \hat H$. This implies that $u$ is fixed by $\sigma_H$

We now prove point (2) by establishing identity \eqref{eq:id1}. The cocycle $z$ is a Tits cocycle relative to the gauge $p$. We emphasize this by writing $z_p$. We have
\[ z_p(\sigma,\tau) = \prod_{\substack{\alpha>0 \\ \sigma_H^{-1}\alpha<0 \\ (\sigma\tau)_H^{-1}\alpha>0}} \alpha(-1), \]
where the product runs over all $\alpha \in R^\vee(\hat T,\hat G)$ and $\alpha>0$ means $p(\alpha)=+1$. Recall that $\partial s_{q/p}=z_q/z_p$. Therefore we want to check that $\tx{Ad}(n)z_q=z'_p$, where $z'_p$ is defined by the same formula, but with $\sigma_H$ replaced by $\sigma_{H'}$. The verification of this identity is immediate.

We now come to the proof of (3). Recall that a representative $c$ of $[c]$ is fixed by choosing a pinning of $\hat H$ that is adapted to the chosen pinning of $\hat G$. To specify a representative $c'$ of $[c']$ we need to do the same for $\hat H'$. Since $\tx{Ad}(n)$ maps a pinning of $\hat H$ adapted to that of $\hat G$ to a pinning of $\hat H'$ adapted to that of $\hat G$, we arrange the pinnings of $\hat H$ and $\hat H'$ to be compatible with $\tx{Ad}(n)$.

By construction $c \in C^1(\Gamma,\hat T^H_\tx{ad}[2])$ is such that $c(\sigma)^{-1}n(\sigma_{H,G})\rtimes \sigma_G$ preserves the chosen pinning of $\hat H$, and $c' \in C^1(\Gamma,\hat T^{H'}_\tx{ad}[2])$ is such that $c'(\sigma)^{-1}n(\sigma_{H',G})\rtimes \sigma_G$ preserves the chosen pinning of $\hat H'$. Therefore both $c(\sigma)^{-1}n(\sigma_{H,G})\rtimes \sigma_G$ and $n^{-1} \cdot c'(\sigma)^{-1}n(\sigma_{H',G})\rtimes\sigma_G \cdot n$ preserve the pinning of $\hat H$, and hence are equal elements of $\mc{H}/Z(\hat H)$. This allows us to choose $x$ and $x'$ in such a way that identity \eqref{eq:id2} holds.

It remains to prove identity \eqref{eq:id3}. This will require more work. We begin with the equality
\begin{equation} \label{eq:hh'1}
x(\sigma)^{-1}n(\sigma_{H,G})=n^{-1}x'(\sigma)^{-1}n(\sigma_{H',G})\sigma_G(n)
\end{equation}
in $\hat H$ that follows immediately from identity \eqref{eq:id2}. Recalling that $w \in \Omega(\hat T,\hat G)$ is the image of $n$, we have
\[ \sigma_{H',G}\sigma_G = \sigma_{H'} = w\sigma_H w^{-1} = w\sigma_{H,G}\sigma_G w^{-1} = w\sigma_{H,G}\sigma_G(w)^{-1}\sigma_G, \]
thus
\[ n(\sigma_{H',G})=z_p(w,\sigma_{H,G}\sigma_G(w)^{-1})n(w)z_p(\sigma_{H,G},\sigma_G(w)^{-1})n(\sigma_{H,G})n(\sigma_G(w)^{-1}). \]
Recall that we have $n=n(w)$. Since the Tits lifting is $\sigma_G$-equivariant, we have
\[ n(\sigma_G(w)^{-1})=\sigma_G(n(w^{-1}))=\sigma_G(z_p(w^{-1},w)n(w)^{-1}).\]
Substituting the last two displayed equations into \eqref{eq:hh'1} and using $\sigma_{H,G}\sigma_G=\sigma_H$ we obtain
\[ 
x(\sigma)^{-1}=w^{-1}x'(\sigma)^{-1} z_p(w,\sigma_{H,G}\sigma_G(w)^{-1})wz_p(\sigma_{H,G},\sigma_G(w)^{-1})\sigma_H(z_p(w^{-1},w)),
\]
which implies that $n^{-1}x'(\sigma)n \cdot x(\sigma)^{-1}$ equals
\begin{equation} \label{eq:hh'2}
n^{-1}z_p(w,\sigma_{H,G}\sigma_G(w)^{-1})n \cdot z_p(\sigma_{H,G},\sigma_G(w)^{-1})\cdot\sigma_H(z_p(w^{-1},w)).	
\end{equation}
Our goal is to show that this is cohomologous to $s_{q/p}$. The latter being cohomologous to $s_{p/q}=s_{p/q}^{-1}$, we are free to use either one, and $s_{p/q}$ turns out more convenient. We will show that the product of \eqref{eq:hh'2} with $s_{p/q}$ is cohomologically trivial.

The first of the three factors in \eqref{eq:hh'2} equals $(-1)$ raised to the sum over the set
\[ w^{-1}\{\alpha\,|\,p(\alpha)=+1,p(w^{-1}\alpha)=-1,p((w\sigma_{H,G}\sigma_G(w)^{-1})^{-1}\alpha)=+1\}. \]
This set equals
\[ \{\alpha\,|\,p(w\alpha)=+1,p(\alpha)=-1,p((\sigma_{H,G}\sigma_G(w)^{-1})^{-1}\alpha)=+1\}.\]
Note that $p(w\alpha)=q(\alpha)$. Note further that, since $\sigma_G$ preserves $p$, we have $p((\sigma_{H,G}\sigma_G(w)^{-1})^{-1}\alpha)=p(\sigma_G(w)\sigma_{H,G}^{-1}\alpha)=p(\sigma_G(w\sigma_H^{-1}\alpha))=p(w\sigma_H^{-1}\alpha)=q(\sigma_H^{-1}\alpha)$. With this, the above set becomes
\[ \{\alpha\,|\, q(\alpha)=+1,p(\alpha)=-1,q(\sigma_H^{-1}\alpha)=+1\}. \]
Consider now the second factor in \eqref{eq:hh'2}. It equals $(-1)$ raised to the sum over the set
\[ \{ \alpha\,|\, p(\alpha)=+1,p(\sigma_{H,G}^{-1}\alpha)=-1,p((\sigma_{H,G}\sigma_G(w)^{-1})^{-1}\alpha)=+1\}, \]
which by an analogous argument equals
\[ \{ \alpha\,|\, p(\alpha)=+1,p(\sigma_H^{-1}\alpha)=-1,q(\sigma_H^{-1}\alpha)=+1\}. \]
Note that the corresponding sets in both the first and second factors of \eqref{eq:hh'2} are sets of roots determined by conditions on $p(\alpha)$, $p(\sigma_H^{-1}\alpha)$, $q(\alpha)$, and $q(\sigma_H^{-1}\alpha)$. This is also true for the sets of roots that enter the definition of $s_{q/p}$. Moreover, in all these sets we have the condition $q(\sigma_H^{-1}\alpha)=+1$. We list the contributions to the first and second factor of \eqref{eq:hh'2}, and to $s_{q/p}$, in the following table.

\begin{center}
\begin{tabular}{|c|c|c|c|c|c|}
\hline
$p(\alpha)$&$q(\alpha)$&$p(\sigma_H^{-1}\alpha)$&$(1)$&$(2)$&$s_{p/q}$\\
\hline
$+$&$+$&$+$&&&\\
$+$&$+$&$-$&&\checkmark&\checkmark\\
$+$&$-$&$+$&&&\checkmark\\
$+$&$-$&$-$&&\checkmark&\\
$-$&$+$&$+$&\checkmark&&\\
$-$&$+$&$-$&\checkmark&&\\
$-$&$-$&$+$&&&\\
$-$&$-$&$-$&&&\\
\hline
\end{tabular}	
\end{center}

To compute the  product of the first two factors of \eqref{eq:hh'2} and $s_{p/q}$ we must compute the product of $(-1)$ raised to the sum of roots in the various sets determined by the above table. Note that any set can be replaced by its negative, because that will not change the value of the image of $(-1)$ under the sum of its elements.

The set labelled with $(+,+,-)$ appears both in the second factor and in $s_{p/q}$, so its contribution vanishes. The conditions $(+,-,+)$ and $(-,+,-)$ are negatives of each other. But remember that in both cases the condition $q(\sigma_H^{-1}\alpha)=+1$ is also present. Therefore, replacing the set labelled by $(-,+,-)$ by its negative and combining it with the set labelled by $(+,-,+)$ we obtained the set
\[ \{\alpha\,|\, p(\alpha)=+1,q(\alpha)=-1,p(\sigma_H^{-1}\alpha)=+1\} \]
and there is now no condition on $q(\sigma_H^{-1}\alpha)$. We do the same procedure to the sets labelled with $(+,-,-)$ and $(-,+,+)$ and obtain the set
\[ \{(\alpha\,|\, p(\alpha)=+1,q(\alpha)=-1,p(\sigma_H^{-1}\alpha)=-1\}. \]
Finally, the last two displayed sets combine to the set
\[ \{(\alpha\,|\, p(\alpha)=+1,q(\alpha)=-1\}.\]
We conclude that the product of the first two factors of \eqref{eq:hh'2} with $s_{p/q}$ equals $z_p(w^{-1},w)=z_p(w^{-1},w)^{-1}$. Thus the product of \eqref{eq:hh'2} with $s_{p/q}$ equals
\[ z_p(w^{-1},w)^{-1} \cdot \sigma_H(z_p(w^{-1},w)) \]
and is thus a coboundary.
\end{proof}

\subsection{The cover associated to $\mc{H} \subset {^LG_{x_G}}$} \label{sub:l2-covers}

Here we consider a slight generalization of the preceding subsection. Instead of the $L$-group $\hat G \rtimes \Gamma$ of $G$ we consider the $L$-group $^LG_{x_G}$ of the cover $G(F)_{x_G}$ associated to some $x_G : \bar\pi_1(G) \to \mu_n(\C)$. 

Assume given a subgroup $\mc{H} \subset {^LG_{x_G}}$, up to $\hat G$-conjugation, satisfying Properties \ref{prp:hdc}. We want to associate a quasi-split connected reductive $F$-group $H$, a cover $H(F)_{x_H}$ of $H(F)$, and an isomorphism $^LH_{x_H} \to \mc{H}$.

The first basic observation is that giving $\mc{H}$ is equivalent to giving a subgroup $\mc{\bar H}$ of $^L(G_\tx{sc})={^LG}/Z(\hat G)$ satisfying Properties \ref{prp:hdc} with respect to $G_\tx{sc}$. Indeed, $\mc{\bar H}=\mc{H}/Z(\hat G)$ and $\mc{H}$ is the preimage of $\mc{\bar H}$ in $^LG$. The equivalent statement holds with $^LG_{x_G}$ in place of $^LG$. However, we have the identification $^LG_{x_G}/Z(\hat G) = {^LG}/Z(\hat G)$. Therefore, $\mc{H} \subset {^LG}_{x_G}$ determines $\mc{H'} \subset {^LG}$ satisfying Properties \ref{prp:hdc}.

The construction of the previous subsection applies to $\mc{H'}$ and produces a quasi-split reductive group $H$, and a character $x_{H,G} : \bar\pi_1(H) \to \{\pm1\}$, hence a double cover $H(F)_{x_{H,G}}$ and an $L$-embedding $^LH_{x_{H,G}} \to \hat G \rtimes \Gamma$ whose image is $\mc{H}'$. 

For a moment we assume that $x_G$ factors through $\tilde\pi_1(G)$. The inclusions $Z(\hat G) \to Z(\hat H) \to \hat T^H$ dualize to surjections $\bar\pi_1(H) \to \tilde\pi_1(H) \to \tilde\pi_1(G)$ via which we pull back $x_G : \tilde\pi_1(G) \to \mu_n(\C)$ and define $x_H : \bar\pi_1(H) \to \mu_m(\C)$ to be the product of $x_G$ and $x_{H,G}$, where $m$ is the least common multiple of $n$ and $2$. Then the $L$-embedding $^LH_{x_{H,G}} \to {^LG}$, composed with the identifications of sets $^LG = \hat G \boxtimes_{x_G} \Gamma$ and $^LH_{x_H} = \hat H \boxtimes_{x_{H,G}}\Gamma$, becomes an isomorphism of groups $^LH_{x_H} \to \mc{H}$.

The arguments of \S\ref{sub:elementary} now extend this construction to the case of a general $x_G : \bar\pi_1(G) \to \mu_n(\C)$. More explicitly, these arguments, combined with \eqref{eq:liso}, give the following formula for the $L$-isomorphism $^LH_{x_H} \to \mc{H}$
\begin{equation} \label{eq:liso1}
\hat H \boxtimes_{z_{H,G} \partial x^{-1} \cdot z_G \partial x_G^{-1}} \Gamma \to \hat G \boxtimes_{z_G \partial x_G^{-1}} \Gamma,\quad h \boxtimes\sigma \mapsto h x(\sigma)^{-1}n(\sigma_{H,G}) \boxtimes\sigma_G
\end{equation}
where we have written $x_{H,G}=(z_{H,G},c)$, $x \in C^1(\Gamma,\hat T^H)$ is any lift of $c \in C^1(\Gamma,\hat T^H_\tx{ad})$, $x_G=(z_G,c_G)$, and $x_G \in C^1(\Gamma,\hat T)$ is any lift of $c_G \in C^1(\Gamma,\hat T_\tx{ad})$. We are taking the limit over all $x_G$ and all $x$. On the right, this limit is $^LG_{x_G}$ according to \eqref{eq:h-l2}. On the left, we define the limit to be $^LH_{x_H}$.

\subsection{Relative position of embeddings $H \from S \to G$} \label{sub:rel2}

As in the previous subsection, we consider a (possibly trivial) character $x_G : \bar\pi_1(G) \to \mu_n$ and a $\hat G$-conjugacy class of subgroups $\mc{H} \subset {^LG}_{x_G}$ satisfying Properties \ref{prp:hdc}. Let $H$ be the associated quasi-split group, and $^L\xi_{H,G} : {^LH}_{x_H} \to {^LG}_{x_G}$ the associated $\hat G$-conjugacy class of $L$-embeddings, as in \S\ref{sub:l2-covers}. We write $\hat\xi_{H,G} : \hat H \to \hat G$ for the restriction of $^L\xi_{H,G}$ to $\hat H$, a $\hat G$-conjugacy class of embeddings $\hat H \to \hat G$.

Consider in addition an $F$-torus $S$ equipped with a $\Gamma$-invariant $G(\bar F)$\-conjugacy class of embeddings $\xi_{S,G} : S \to G$, and a $\Gamma$-invariant $H(\bar F)$\-conjugacy class of embeddings $\xi_{S,H} : S \to H$. As discussed in \cite[\S5.1]{KalRSP} these determine a $\Gamma$-invariant $\hat G$-conjugacy class of embeddings $\hat\xi_{S,G} : \hat S \to \hat G$, and a $\Gamma$-invariant $\hat H$-conjugacy class of embeddings $\hat\xi_{S,H} : \hat S \to \hat H$. We assume that $\hat\xi_{S,G} = \hat\xi_{H,G} \circ \hat\xi_{S,H}$, up to $\hat G$-conjugacy.

\begin{cns} \label{cns:rel2}
We will associate to this data a cover 
\[ S(F)_{G} \oplus S(F)_{-H} \oplus S(F)_{x_G} \oplus S(F)_{-x_H} \] 
of $S(F)$ and a genuine character $\<\tx{inv}(\xi_{S,H},\xi_{S,G}),-\>$ of that cover.

For any $\xi_{S,G}$ within its $G(\bar F)$-conjugacy class let $R(S,G) \subset X^*(S)$ be the preimage of the absolute root system $R(\xi_{S,G}(S),G)$. Let $R(S,H) \subset X^*(S)$ be the analogous set for $H$ in place of $G$. The covers $S(F)_G$ and $S(F)_H$ of $S(F)$ are the double covers corresponding to these admissible sets as constructed in \cite{KalDC} and reviewed in \S\ref{sub:admset}. We are denoting by $S(F)_{-H}$ the Baer inverse of $S(F)_H$, which however is canonically isomorphic to $S(F)_H$ since the degree of this cover is $2$.

For any $\xi_{S,G} : S \to G$ within the given $G(\bar F)$-conjugacy class that is defined over $F$ Construction \ref{cns:s-t} provides a cover $S(F)_{x_G}$ and an embedding $S(F)_{x_G}\to G(F)_{x_G}$, which we will again denote by $\xi_{S,G}$. We note that the cover $S(F)_{x_G}$ does not depend on the chosen embedding within the given conjugacy class, because all of these embeddings induce the same map $Z(\hat G) \to \hat S$. In the same way, we obtain the cover $S(F)_{x_H}$ and, for each $\xi_{S,H} : S \to H$ defined over $F$ withing the given $H(\bar H)$-conjugacy class, a lift $S(F)_{x_H} \to H(F)_{x_H}$, which we will again denote by $\xi_{S,H}$. We denote by $S(F)_{-x_H}$ the Baer inverse of $S(F)_{x_H}$.

Note that $S(F)_G \oplus S(F)_{-H}$ is the double cover associated to the admissible set $R(S,G) \sm R(S,H)$ and $S(F)_{x_G} \oplus S(F)_{-x_H}$ is the Baer inverse of the double cover associated to the character $x_{H,G}=x_H/x_G$ of \eqref{eq:h-tits}.

We now construct the genuine character $\<\tx{inv}(\xi_{S,H},\xi_{S,G}),-\>$. There is a canonical $\hat G$\-conjugacy class of $L$\-embeddings $^L\xi_{S,G} : {^LS}_G \to {^LG}$, and a canonical $\hat H$\-conjugacy class of $L$\-embeddings $^L\xi_{S,H} : {^LS}_H \to {^LH}$, cf. \cite[\S4.1]{KalDC}. The genuine character  will measure the difference between ${^L\xi}_{S,G}$ and ${^L\xi}_{H,G} \circ {^L\xi}_{S,H}$. But we first have to make sense of this composition and of the comparison.

Let us recall from Proposition \ref{pro:dc_comp} that $S(F)_G$ is given as a limit of a system $S(F)_{t^{S,G}_p,2}$ of double covers associated to Tits cocycles $t^{S,G}_p$, the limit being taken over the set of gauges $p : R(S,G) \to \{\pm1\}$. The $L$-group $^LS_G$ is also obtained as the limit of $^LS_{t^{S,G}_p}$. The same is true for $S(F)_H$, where we'll use the notation $t^{S,H}_p$ for the corresponding Tits cocycle. Since $R(S,H) \subset R(S,G)$, a gauge for $R(S,G)$ restricts to a gauge for $R(S,H)$, so we may take the limits for both $G$ and $H$ over the set of gauges for $R(S,G)$.

The homomorphism $^L\xi_{S,G} : {^LS}_{t^{S,G}_p} \to {^LG}$ becomes, via Construction \ref{cns:lemb-s-t}, a homomorphism $^LS_{x_G \cdot t^{S,G}_p} \to {^LG}_{x_G}$, which we still denote by $^L\xi_{S,G}$. In the same way, we obtain from $^L\xi_{S,H}$ a homomorphism $^LS_{x_H \cdot t^{S,H}_p} \to {^LH}_{x_H}$, which we again denote by $^L\xi_{S,H}$. The composition ${^L\xi_{H,G}} \circ {^L\xi_{S,H}}$ is now an $L$\-embedding $^LS_{x_H \cdot t_p^{S,H}} \to {^LG}_{x_G}$. 

To compare $^L\xi_{S,G} : {^LS}_{x_G \cdot t^{S,G}_p} \to {^LG}_{x_G}$ with ${^L\xi_{H,G}} \circ {^L\xi_{S,H}} : {^LS}_{x_H \cdot t_p^{S,H}} \to {^LG}_{x_G}$, both of which depend on $p$ and are given up to conjugation by $\hat G$, we first arrange that we are using the same $p$ for both, and arrange by $\hat G$-conjugacy that both restrict to the same embedding $\hat S \to \hat G$. Then Proposition \ref{pro:lemb-comp} provides an $L$-isomorphism $^L\eta : ^LS_{x_H \cdot t_p^{S,H}} \to {^LS}_{x_G \cdot t_p^{S,G}}$  extending the identity on $\hat S$ such that ${^L\xi_{H,G}} \circ {^L\xi_{S,H}} = {^L\xi_{S,G}} \circ {^L\eta}$, and dually a genuine character $\eta$ of $S(F)_G \oplus S(F)_{-H} \oplus S(F)_{x_G} \oplus S(F)_{-x_H}$.
\end{cns}

\section{Endoscopic data and transfer factors} \label{sec:endo}

\subsection{Relative position of cover element and pinning} \label{sub:rel1}

A classical object of endoscopy is the cohomological invariant that measures the relative position of two stably conjugate (strongly regular) semi-simple elements of a given connected reductive group $G$; it is denoted by $\tx{inv}(\gamma,\gamma')$ in \cite[\S5.6]{Kot86}. For transfer between different groups, a finer invariant becomes necessary. Such a finer invariant appears implicitly in the construction of transfer factors, \cite[\S2.3]{LS87}. Using the double covers of tori introduced in \cite{KalDC}, we can make this invariant explicit. It takes the form of an invariant between an element of a cover of a maximal torus, and a pinning of the quasi-split form. 

Let $G$ be a connected reductive $F$-group and let $S \subset G$ be a maximal torus. Recall that the double cover $S(F)_\pm$ was defined in \cite[Definition 3.1]{KalDC} as the quotient of the ``big cover'' $S(F)_{\pm\pm}$, by pushing out under the multiplication map $\prod_O \{\pm1\} \to \{\pm1\}$. Let us call elements $\tilde\delta \in S(F)_\pm$ ``regular'', if their image in $S(F)$ is regular. Given a regular element $\tilde\delta \in S(F)_{\pm\pm}$ and a pinning of the quasi-split form of $G$, called pin, we will define in this section an element
\begin{equation} \label{eq:splinv}
\tx{inv}(\tilde\delta,\tx{pin})
\end{equation}
that lies in $H^1(F,S)$ when $G$ is quasi-split, or a pure inner twist of its quasi-split form, and more generally in $H^1(P \to \mc{E},Z(G) \to S)$. This construction will be a reinterpretation of work of Langlands--Shelstad from \cite{LS87}.

Before we explain the construction of this invariant we will review its Lie algebra version, which is more intuitive. Assume first that $G$ is quasi-split and fix an $F$-pinning $\tx{pin}=(T,B,\{X_\alpha\})$. Write $\mf{s}$ for the Lie algebra of $S$. For a regular semi-simple $Y$ we will define an element of $H^1(F,S)$ (when $G$ is quasi-split) and more generally $H^1(P \to \mc{E},Z(G) \to S)$ (when $G$ is general), which we will call
\begin{equation} \label{eq:splinv-la}
\tx{inv}(Y,\tx{pin}).
\end{equation}
This definition depends on the Kostant section, which we briefly review. For a simple root $\alpha$ define $X_{-\alpha}$ by $[X_\alpha,X_{-\alpha}]=H_\alpha$ and $X_-=\sum_\alpha X_{-\alpha}$. Then $X_-$ is a regular nilpotent element in $\tx{Lie}(G)(F)$. Write $\mf{b}$ and $\mf{t}$ for the Lie algebras of $B$ and $T$, respectively, $U$ for the unipotent radical of $B$, and $\Omega$ for the Weyl group. A result of Kostant \cite{Kos63}, summarized in \cite[\S2.4]{Kot99}, shows that $\mf{b} + X_- \to \mf{t}/\Omega$ is a principal $U$-bundle that meets every regular $G$-conjugacy class in $\mf{g}$ and admits sections. Passing to rational points and using $H^1(F,U)=\{0\}$, we obtain the principal $U(F)$-bundle $\mf{b}(F) + X_- \to (\mf{t}/\Omega)(F)$ that meets every stable conjugacy class in $\mf{g}(F)$ and admits sections. Therefore, every stable conjugacy class of regular semi-simple elements of $\mf{g}(F)$ meets $\mf{b}(F) + X_-$ in a set of elements that are all conjugate to each other under $U(F)$. Let $\mf{s}$ be the Lie algebra of $S$. Given a regular semi-simple $Y \in \mf{s}(F)$ there is $Y' \in \mf{b}(F) + X_-$, unique up to $U(F)$-conjugacy, that is stably conjugate to $Y$. Let $g \in G(F^s)$ be such that $\tx{Ad}(g)Y=Y'$. Then $\sigma \mapsto g^{-1}\sigma(g)$ is an element of $Z^1(F,S)$ whose class in $H^1(F,S)$ is independent of the choices of $Y'$ and $g$, and is the desired invariant \eqref{eq:splinv-la}.

More generally, assume that $G$ is given as a rigid inner twist $(G,\xi,z)$ of a quasi-split group $G^*$. Fix a pinning $\tx{pin}=(T,B,\{X_\alpha\})$ of $G^*$ as before. Given a regular element $Y \in \mf{s}(F)$, the $G(F^s)$-conjugacy class of $\xi^{-1}(Y) \in \mf{g}(F^s)$ is defined over $F$, so some element $Y'$ of it lies in $\mf{b}(F)+X_-$. Let $g \in G(F^s)$ be such that $\tx{Ad}(g)\xi^{-1}(Y)=Y'$. Then $e \mapsto g^{-1}z_e\sigma_e(g)$ is an element of $Z^1(P \to \mc{E},Z(G) \to S)$ whose class is independent of the choices of $Y'$ and $g$, and is the desired invariant \eqref{eq:splinv-la}.

This concludes the review of the invariant for Lie algebras. We now come to the construction of the invariant \eqref{eq:splinv} for groups. For this, we recall that elements of $S(F)_{\pm\pm}$ were described explicitly in \cite[Remark 3.3]{KalDC}, according to which some elements are given by tuples $\tilde\delta=(\delta,(\delta_\alpha)_{\alpha \in R(S,G)})$ such that $\delta \in S(F)$, $\delta_\alpha \in F_\alpha^\times$, $\sigma(\delta_\alpha)=\delta_{\sigma\alpha}$ for all $\sigma \in \Gamma$, and $\delta_\alpha/\delta_{-\alpha}=\alpha(\delta)$. While such tuples don't describe all elements of $S(F)_{\pm\pm}$, they describe sufficiently many, in the following sense. For each asymmetric $\Sigma$-orbit $O \subset R(S,G)$ the non-trivial element $\epsilon_O$ of the kernel of $J_O(F)_\pm \to J_O(F)$ naturally lies in $S(F)_{\pm\pm}$, and any element of $S(F)_{\pm\pm}$ can be obtained from a tuple as above by multiplying it by a product of $\epsilon_O$ for suitable asymmetric $O$. We will define our invariant to be constant under multiplication by $\epsilon_O$ for asymmetric $O$, and will therefore focus on elements of $S(F)_{\pm\pm}$ describable by tuples $\tilde\delta=(\delta,(\delta_\alpha)_{\alpha \in R(S,G)})$.

Assume first that $G$ is quasi-split and fix an $F$-pinning $\tx{pin}=(T,B,\{X_\alpha\})$. Choose $g \in G(F^s)$ such that $gTg^{-1}=S$. Then $\sigma \mapsto g^{-1}\sigma(g)$ is an element of $Z^1(\Gamma,N(T,G)(F^s))$. Let $\omega_\sigma \in Z^1(\Gamma,\Omega(T,G)(F^s))$ denote its image, where $\Omega(T,G)=N(T,G)/T$ is the Weyl group. Write $\sigma_T$ for the action of $\sigma \in \Gamma$ on $T(F^s)$ via the rational structure of $T$, and write $\sigma_S = \tx{Ad}(\omega_\sigma)\circ \sigma_T$. Then $\tx{Ad}(g) : T \to S$ translates the action of $\sigma_S$ on $T(F^s)$ to action of $\sigma$ on $S(F^s)$ coming from the rational structure of $S$. For $\sigma$, let $n(\omega_\sigma) \in N(T,G)(F^s)$ denote the Tits lift of $\omega_\sigma$  with respect to the pinning. Then
\begin{equation}
x_\sigma(\tilde\delta) := \prod_{\substack{\alpha>0\\ \sigma_S^{-1}\alpha<0}} \alpha^\vee(\delta_{g\alpha}-\delta_{-g\alpha}) n(\omega_\sigma) 	
\end{equation}
also lies in $Z^1(\Gamma,N(T,G)(F^s))$ according to \cite[Lemma 2.2.A]{LS87}. It follows that
\begin{equation} 
\sigma \mapsto g\cdot (x_\sigma(\tilde\delta) \cdot (g^{-1}\sigma(g))^{-1}) \cdot g^{-1}
\end{equation}
is an element of $Z^1(\Gamma,S(F^s))$. According to \cite[(2.3.3)]{LS87} the class of this element is independent of the choice of $g$, and is the desired invariant \eqref{eq:splinv}.

We now drop the assumption that $G$ is quasi-split. Instead, we consider a rigid inner twist $(\xi,z) : G^* \to G$ in the sense of \cite[\S5.1]{KalRI} or \cite[\S7.1]{Dillery20}, where $G^*$ is quasi-split. We fix a pinning $\tx{pin}=(T,B,\{X_\alpha\})$ of $G^*$. Let $S \subset G$ be a maximal torus. Then we obtain
\begin{equation} 
\tx{inv}((\xi,z,\tilde\delta),\tx{pin}) \in H^1_\tx{bas}(\mc{E},S)	
\end{equation}
represented by
\begin{equation}
e \mapsto \xi(g\cdot (x_{\sigma_e}(\tilde\delta) \cdot (g^{-1}z_e\sigma_e(g))^{-1}) \cdot g^{-1}),
\end{equation}
where $g \in G^*(F^s)$ has the property that $\xi(gTg^{-1})=S$. Note that $g^{-1}z_e \sigma_e(g)$ is an element of $Z^1(\mc{E},N(T,G)(F^s))$ whose image in $Z^1(\mc{E},\Omega(T,G)(F^s))$ lies in $Z^1(\Gamma,\Omega(T,G)(F^s))$ and we denote by $\omega_\sigma$. The construction of $x_\sigma$ remains the same. This completes the definition of the invariant \eqref{eq:splinv}.

Note that here we are using the notation $g^{-1}z_e\sigma_e(g)$ for the characteristic-zero case. In the case of local function fields, the notation should be changed to $p_1^{-1}(g)zp_2(g)$, cf. \cite[\S7.1]{Dillery20}. Since the arguments we will be making in this section remain the same, owing to the material developed in \cite{Dillery20}, we will avoid introducing the positive-characteristic notation, and give the arguments in the characteristic-zero notation.

We now examine how the invariant depends on the choice of $\tilde\delta$ with fixed $\delta$. Another choice is of the form $\eta\tilde\delta = (\delta,(\delta_\alpha\eta_\alpha)_{\alpha \in R(S,G)})$, where $\eta=(\eta_\alpha)_{\alpha \in R(S,G)}$ is a collection of elements $\eta_\alpha \in F_{\pm\alpha}^\times$ such that $\sigma(\eta_\alpha)=\eta_{\sigma\alpha}$ for all $\sigma \in \Gamma$ and $\eta_{-\alpha}=\eta_\alpha$. Recall that for each $\Sigma$-orbit $O \subset R(S,G)$ the $F$-torus $J_O$ is equipped with homomorphisms $J_O \to S_\tx{sc}$ and $S_\tx{ad} \to J_O$. The subcollection $(\eta_\alpha)_{\alpha \in O}$ provides an element of $H^1(F,J_O)=F_{\pm\alpha}^\times/N_{F_\alpha/F_{\pm\alpha}}(F_\alpha^\times)$. Note that this quotient is isomorphic to $\Z/2\Z$ when $O$ is symmetric, and trivial when $O$ is asymmetric. We'll write $\eta_O$ for this element of $H^1(F,J_O)$ as well as his images in $H^1(F,S_\tx{sc})$ and $H^1(F,S)$.

\begin{lem} \label{lem:rel1_gen1}
\[ \tx{inv}((\xi,z,\eta\tilde\delta),\tx{pin}) = \tx{inv}((\xi,z,\tilde\delta),\tx{pin}) \cdot \prod_{O \in R(S,G)_\tx{sym}/\Gamma} \eta_O. \]
\end{lem}
\begin{proof}
According to \cite[(2.3.2)]{LS87}, $\tx{inv}((\xi,z,\eta\tilde\delta),\tx{pin})=\tx{inv}((\xi,z,\tilde\delta),\tx{pin}) \cdot [b_q]$, where $[b_q]=\prod_O [b_{q,O}] \in H^1(F,S)$ and $[b_{q,O}]$ is the class of the following 1-cocycle
\[ \prod_{\substack{\alpha>0\\ \sigma^{-1}\alpha<0\\\alpha\in O}}\alpha^\vee(\eta_{\alpha}). \]
Here $O$ runs over the set of $\Sigma$-orbits in $R(S,G)$ and $\alpha>0$ is taken with respect to an arbitrary gauge $q$; the choice of $q$ does not influence the cohomology class of $[b_{q,O}]$. But $b_{q,O}$ is simply an explicit description of the class $\eta_O$ by means of the formula for the Shapiro isomorphism on 1-cochains.
\end{proof}

\begin{cor}
The element $\tx{inv}((\xi,z,\tilde\delta),\tx{pin})$ does not change if $\tilde\delta$ is replaced by a different tuple that represents the same element of $S(F)_{\pm\pm}$.
\end{cor}
\begin{proof}
When $\alpha$ is asymmetric then $J_O$ is an induced torus, hence $H^1(F,J_O)$ vanishes. When $\alpha$ is symmetric, then $H^1(F,J_O) = \{\pm 1\}$ via the sign character $\kappa_\alpha : F_{\pm\alpha}^\times/N_{F_\alpha/F_{\pm\alpha}}(F_\alpha^\times) \to \{\pm 1\}$. The tuples $\tilde\delta$ and $\eta\tilde\delta$ represent the same element of $S(F)_{\pm\pm}$ if and only if $\kappa_\alpha(\eta_\alpha)=1$ for all symmetric $\alpha \in R(S,G)$. This is exactly the case when the class $\eta_O$ of $H^1(F,J_O)$ is trivial for all $O$.
\end{proof}

This completes the definition of the invariant between a strongly regular element of $S(F)_{\pm\pm}$ and a fixed pinning. It refines the usual invariant between two stably conjugate strongly regular semi-simple elements in the following sense.

Let $(\xi_i,z_i) : G^* \to G_i$ be two rigid inner twists. Let $\delta_i \in G_i(F)$ be two related strongly regular semi-simple elements and denote their centralizers by $S_i$. We have the admissible isomorphism $\varphi_{\delta_1,\delta_2} : S_1 \to S_2$ mapping $\delta_1$ to $\delta_2$. It lifts canonically to an isomorphism $S_1(F)_{\pm\pm} \to S_2(F)_{\pm\pm}$ which we again denote by $\varphi_{\delta_1,\delta_2}$. For any lift $\tilde\delta_1 \in S_1(F)_{\pm\pm}$ let $\tilde\delta_2 = \varphi_{\delta_1,\delta_2}(\tilde\delta_1)$.

\begin{lem} \label{lem:delta1_kappa}
For any $F$-pinning of $G^*$ we have
\[ \tx{inv}((\xi_1,z_1,\delta_1),(\xi_2,z_2,\delta_2)) = \tx{inv}((\xi_1,z_1,\tilde\delta_1),\tx{pin}) \cdot \varphi_{\delta_1,\delta_2}^{-1}(\tx{inv}((\xi_2,z_2,\tilde\delta_2),\tx{pin}))^{-1}. \]
\end{lem}
\begin{proof}
Since $\delta_1$ and $\delta_2$ are related elements there exist $g_1,g_2 \in G^*(F^s)$ such that $g_1^{-1}\xi_1^{-1}(\delta_1)g_1 = g_2^{-1}\xi_2^{-1}(\delta_2)g_2 \in T$. Then $\xi_i(g_iTg_i^{-1})=S_i$ and $\tx{inv}((\xi,z,\tilde\delta_i),\tx{pin})$ is represented by $\xi_i(g_i(x_{\sigma_e}(\tilde\delta_i)\cdot(g_i^{-1}z_i(e)\sigma_e(g_i))^{-1})g_i^{-1}))$. We have $\varphi_{\delta_1,\delta_2} = \xi_2 \circ \tx{Ad}(g_2g_1^{-1})\circ \xi_1^{-1}$ and the right hand side of the above equation becomes
\[ \xi_1(g_1(x_{\sigma_e}(\tilde\delta_1)(g_1^{-1}z_1(e)\sigma_e(g_1))^{-1}(g_2^{-1}z_2(e)\sigma_e(g_2))x_{\sigma_e}(\tilde\delta_2)^{-1})g_1^{-1}). \]
On the other hand, the left hand side of the equation is
\[ \tx{inv} := \xi_1(g_1g_2^{-1}z_2(e)\sigma_e(g_2g_1^{-1})z_1(e)^{-1}). \]
Combining this with $x_\sigma(\tilde\delta_1)=x_\sigma(\tilde\delta_2)$ we see that the right hand side equals
\[ \xi_1(g_1 \cdot y \cdot g_1^{-1}) \cdot \tx{inv} \cdot \xi_1(g_1 \cdot y \cdot g_1^{-1})^{-1},\qquad y = x_{\sigma_e}(\tilde\delta_1)(g_1^{-1}z_1(e)\sigma_e(g_1))^{-1}. \]
Now both $x_{\sigma_e}(\tilde\delta_1)$ and $g_1^{-1}z_1(e)\sigma_e(g_1)$ are elements of $N(T,G)(F^s)$ with the same projection to $\Omega(T,G)(F^s)$, hence $y \in T(F^s)$. Therefore $\xi_1(g_1\cdot y \cdot g_1^{-1})$ commutes with $\tx{inv}$.
\end{proof}

We will now establish a relationship between the invariant \eqref{eq:splinv} and its Lie algebra analog \eqref{eq:splinv-la}, under the assumption that $F$ is of characteristic zero. Recall from \cite[\S3.7]{KalDC} that the exponential map $\tx{Lie}(S)(U) \to S(F)$ factors through a map $\tx{Lie}(S)(U) \to S(F)_{\pm\pm}$. Here $\tx{Lie}(S)(U)$ is an open subgroup of $\tx{Lie}(S)(F)$ on which the exponential map converges.

\begin{lem} \label{lem:expinv}
Assume that $F$ has characteristic zero. There exists an open neighborhood $V \subset \tx{Lie}(S)(U)$ of $0$ such that for all regular $Y \in V$
\[ \tx{inv}(Y,\tx{pin}) = \tx{inv}(\tilde\delta,\tx{pin}), \]
with $\tilde\delta=\exp(Y)$.
\end{lem}
\begin{proof}
Take a regular element $Y \in V \subset \tx{Lie}(S)(U)$, for some open neighborhood $V$ of $0$ to be determined in the course of the proof. Let $Y' \in \tx{Lie}(G^*)(F)$ be stably conjugate to $Y$ and lie in the space $\mf{b}(F) + X_-$ used in the definition of \eqref{eq:splinv-la}. Let $S' \subset G^*$ be the centralizer of $Y'$ and let $\tilde\delta'=\exp(Y') \in S'(F)_{\pm\pm}$. Lemma \ref{lem:delta1_kappa} applied with $(\xi,z,\tilde\delta)$ and $(1,1,\tilde\delta')$ reduces the proof to showing that $\tx{inv}(\tilde\delta',\tx{pin}) \in H^1(F,S')$ is trivial.

Let $a_\alpha=\delta_\alpha-\delta_{-\alpha}$. This is a set of $a$-data in the sense of \cite[\S2.2]{LS87}. Another set of $a$-data is $a'_\alpha = d\alpha(Y')$. To each set of $a$-data a splitting invariant $\lambda(S',-) \in H^1(F,S')$ is defined in \cite[\S2.3]{LS87}. The invariant $\tx{inv}(\tilde\delta',\tx{spl})$ is equal to $\lambda(S',(a_\alpha))$. On the other hand, the main result of \cite{Kot99}, Theorem 5.1 (or, technically, it's proof) shows that the splitting invariant $\lambda(S',(a_\alpha'))$ vanishes. It is therefore enough to show that, for every symmetric $\alpha \in R(S',G)$, the element $b_\alpha := a_\alpha/a'_\alpha$, a-priori belonging to $F_{\pm\alpha}^\times$, actually lies in the smaller subgroup $N_{F_\alpha/F_{\pm\alpha}}(F_\alpha^\times)$.

Recall from \cite[\S3.7]{KalDC} that $\delta_\alpha=\exp(d\alpha(Y')/2)$. The compatibility of the exponential map with homomorphisms implies 
\[ a_\alpha=\exp(d\alpha(Y'))/2)-\exp(-d\alpha(Y')/2) = \exp(a_\alpha'/2) - \exp(-a_\alpha'/2). \]
We thus find ourselves in the following abstract situation: Given a quadratic extension $E/F$ there exists an open neighborhood $V$ in the set of trace-zero elements of $E$, such that for $0 \neq z \in V$ the element $(\exp(z/2)-\exp(-z/2))/z \in F^\times$ is a norm from $E^\times$. When $E/F=\C/\R$, then $z = ir$ for some $r \in \R$ and the element under consideration equals $\sin(r/2)/(r/2)$, which is positive (hence a norm from $\C^\times$) when $-2\pi<r<2\pi$. When $E/F$ is non-archimedean with residual characteristic $p$, we have
\[ (\exp(z/2)-\exp(-z/2))/z = \sum_{n=0}^\infty \frac{(z/2)^{2n}}{(2n+1)!} = 1+\sum_{n=1}^\infty \frac{(z/2)^{2n}}{(2n+1)!}. \]
It is clear that by placing a condition on the valuation of $z$ we can ensure that the valuation of each summand is larger than any fixed number. This in turn implies that the whole expression lies in the open subgroup $N_{E/F}(E^\times)$ of $F^\times$.

\end{proof}

\begin{rem}
The proof of Lemma \ref{lem:expinv} quantifies the neighborhood $V$ as the intersection of the open subsets $d\alpha^{-1}(V_\alpha)$, where $V_\alpha \subset F_\alpha^0$ is the open subset of those $z \in F_\alpha^0$ for which $\exp(z/2)$ converges and moreover $(\exp(z/2)-\exp(-z/2))/z \in F_{\pm\alpha}^\times$ is a norm from $F_{\alpha}^\times$.

When $F=\R$ then $F_{\pm\alpha}=\R$, $F_\alpha=\C$, and $V_\alpha = (-2i\pi,2i\pi) \subset i\R$. When $F$ then $V_\alpha=\{z \in F_\alpha^0\,|\,\tx{ord}(z/2) > r \}$ for some positive real number $r$. Using the normalization $\omega(p)=1$, note that $r = 1/(p-1)$ is already necessary for the convergence of the exponential map. When $p \neq 2$, then $r=3/(2p-2)$ would be sufficient. When $p=2$, then $r=1$ is sufficient.
\end{rem}

The invariant $\tx{inv}(\tilde\delta,\tx{pin})$ is by itself not an endoscopic quantity, but rather an intrinsic object of the rigid inner twist $(G,\xi,z)$. The relation to endoscopy emerges when this invariant is paired with a character of the group $H^1(F,S)$ in the quasi-split case, or more generally of the group $H^1_\tx{bas}(\mc{E},S)$.

Consider the 0-th homology group $H_0(\Gamma,X_*(S_\tx{sc}))$. We can replace $\Gamma$ by a finite quotient $\Gamma_{E/F}$ through which it acts on $X_*(S_\tx{sc})$ without changing this homology group. Let $\kappa : H_0(\Gamma_{E/F},X_*(S_\tx{sc})) \to \C^\times$ be a character. Inflating it to $X_*(S_\tx{sc})$ and intersecting its kernel with $R^\vee(S,G)$ we obtain a closed subrootsystem $R_\kappa^\vee \subset R^\vee(S,G)$. The set $R_\kappa := \{\alpha \in R(S,G)\,|\,\alpha^\vee \in R_\kappa^\vee\}$ is a (not necessarily closed) subrootsystem of $R(S,G)$. On the other hand, restricting $\kappa$ to $H^{-1}(\Gamma_{E/F},X_*(S_\tx{sc})) \subset H_0(\Gamma_{E/F},X_*(S_\tx{sc}))$ and using the Tate-Nakayama isomorphism $H^{-1}(\Gamma_{E/F},X_*(S_\tx{sc})) \to H^1(F,S_\tx{sc})$ we obtain from $\kappa$ a character of $H^1(F,S_\tx{sc})$, which we also denote by $\kappa$.

\begin{lem} \label{lem:rel1_desc}
Let $\dot\kappa : H^1_\tx{bas}(\mc{E},S) \to \C^\times$ be a character whose pull back to $H^1(F,S_\tx{sc})$ coincides with $\kappa$. Then the function $\<\tx{inv}((\xi,z,-),\tx{pin}),\dot\kappa\>$ descends to a genuine function of the double cover $S(F)_\pm$ associated to $R(S,G) \sm R_\kappa$.
\end{lem}
\begin{proof}
The kernel of $S(F)_{\pm\pm} \to S(F)_\pm$ equals the kernel of the product map
\[ \prod_{\alpha \in R(S,G)/\Sigma} \{\pm 1\} \to \{\pm 1\},\qquad (\epsilon_\alpha)_\alpha \mapsto \prod_{\alpha \in (R(S,G) - R_\kappa)/\Sigma} \epsilon_\alpha. \] 
We thus need to show that the function $\<\tx{inv}((\xi,z,-),\tx{pin}),\dot\kappa\>$ transforms under the group $\prod_{\alpha \in R(S,G)/\Sigma} \{\pm 1\}$ by the above sign character. 

Let $\beta \in R(S,G)$ and let $\epsilon_\beta$ denote the corresponding element $-1$ of the subgroup $\prod_\alpha \{\pm 1\} \subset S(F)_{\pm\pm}$. Consider the action of $\epsilon_\beta$ on $S(F)_{\pm\pm}$ by multiplication. If $\beta$ is asymmetric then the function $\tx{inv}((\xi,z,-),\tx{pin})$ is invariant under this action by construction. If $\beta$ is symmetric, then the action of $\epsilon_\beta$ sends $\tilde\delta$ to $\eta\tilde\delta$, where $\eta=\prod_O\eta_O$ has exactly one non-trivial component, namely for the $\Sigma$-orbit $O$ of $\beta$. Lemma \ref{lem:rel1_gen1} shows that $\<\tx{inv}(\tx{pin},\xi,z,\epsilon_\beta\tilde\delta),\dot\kappa\> = \<\eta_O,\kappa\> \cdot \<\tx{inv}(\tx{pin},\xi,z,\tilde\delta),\dot\kappa\>$. Since $\eta_O$ comes from $H^1(F,J_O)$ we may compute the pairing by pulling back $\kappa$ to $H^1(F,J_O)$. The functoriality of the Tate-Nakayama isomorphism implies that this is the same as pulling back $\kappa$ under $\Z/2\Z = H_0(\Gamma_{F_\beta/F_{\pm\beta}},\Z_{(-1)})=H_0(\Gamma,X_*(J_O)) \to H_0(\Gamma,X_*(S_\tx{sc}))$. The map $\Z_{(-1)} \to X_*(S_\tx{sc})$ sends $1$ to $\beta^\vee$, so the pull back of $\kappa$ equals the image of $\<\beta^\vee,\kappa\> \in \Z$ in $\Z/2\Z$. This image is trivial precisely when $\beta \in R_\kappa$. Therefore 
\[ \<\eta_O,\kappa\>=\begin{cases}
1,&\beta \in R_\kappa\\
-1,&\beta \notin R_\kappa.
\end{cases}
\qedhere
\]
\end{proof}

\subsection{Review of endoscopic data and transfer factors} \label{sub:endoreview}

Let $G^*$ be a quasi-split connected reductive $F$-group and $\xi : G^* \to G$ an inner twist. There are various equivalent ways to package the information of an endoscopic datum. For example, \cite[\S1.2]{LS87} uses tuples $(H,s,\mc{H},\xi)$, \cite[\S7.1]{Kot84} uses pairs $(s,\rho)$, and \cite[\S7.4]{Kot84} uses triples $(H,s,\eta)$. Moreover, the proper handling of inner twists requires a refinement, which is discussed in \cite[\S5.3]{KalRI}.

In this paper we find it most convenient to work with pairs $(s,\mc{H})$. Here $s \in \hat G$ is semi-simple and $\mc{H}$ is a full subgroup of $^LG$ that centralizes $s$ and such that $\hat H=\mc{H} \cap \hat G$ is precisely the identity component of the centralizer of $s$ in $\hat G$. An isomorphism between two such tuples $(s_i,\mc{H}_i)$ is an element $g \in \hat G$ such that $g\mc{H}_1g^{-1}=\mc{H}_2$ and $gs_1g^{-1}=s_2$ modulo $Z(\hat G)$. The reason we find this formulation most convenient is that the component $\mc{H}$ of such a tuple is clearly a special case of the group $\mc{H}$ considered in \S\ref{sub:l1-covers}.

Let us briefly indicate how this notion is equivalent to the other notions listed above. Given $(s,\mc{H})$ we obtain $\rho : \Gamma \to \tx{Out}(\hat H)$ by $\rho(\sigma)=\tx{Ad}(h_\sigma)$, where $h_\sigma \in \mc{H}$ any lift of $\sigma$. Conversely, given $(s,\rho)$ we can define $\mc{H}=\{x \in N(\hat H,{^LG})\,|\, \tx{Ad}(x)|_{\hat H} = \rho(\sigma_x)\}$, where $\sigma_x \in \Gamma$ is the image of $x$ under $^LG \to \Gamma$, and the identity is taken in $\tx{Out}(\hat H)$. The passage between pairs $(s,\rho)$ and triples $(H,s,\eta)$ is described in \cite[\S7.6]{Kot84}. To obtain a tuple $(H,s,\mc{H},\xi)$, we let $H$ be the unique quasi-split group with dual group $\hat H$ and such that the homomorphism $\Gamma \to \tx{Out}(H_{F^s})$ given by its $F$-structure is translated to the homomorphism $\rho : \Gamma \to \tx{Out}(\hat H)$ under the identification $\tx{Out}(H_{F^s})=\tx{Out}(\hat H)$, and we let $\xi$ be the tautological embedding $\mc{H} \subset {^LG}$.

Again, the proper handling of inner twists requires a refinement of this notion. It requires to record a preimage $\dot s$ of $s$ in the universal cover $\hat{\bar G}$ of the complex Lie group $\hat G$. It also requires an isomorphism between two resulting tuples $(\dot s_i,\mc{H}_i)$ to satisfy $\tx{Ad}(g)\dot s_1=\dot s_2$ modulo the identity component of $Z(\hat{\bar H_2})^+$, where $Z(\hat{\bar H_2})^+$ is the preimage in $\hat{\bar G}$ of $Z(\hat H_2)^\Gamma$, and the latter is the subgroup of $\hat H_2$ that commutes with all of $\mc{H}_2$.

Next we want to discuss various normalizations of transfer factors. There are three meanings of the word ``normalization''. One pertains to the fact that for general connected reductive groups the transfer factor is defined only up to a scalar constant multiple, and one needs a Whittaker datum and a rigid inner twist datum to fix that constant. Another meaning is that there are two different choices for the isomorphism of local class field theory (the Artin choice and the Deligne choice), and these result in two different versions of the transfer factor. Yet another meaning of the word normalization is that one can invert individual pieces of the transfer factor (or equivalently, invert the endoscopic element $s$).

Let us belabor the second meaning. The basic question of normalization concerns the isomorphism $W_F^\tx{ab} \to F^\times$ of local class field theory. An arithmetic Frobenius element of $W_F$ is one that acts as $x \mapsto x^q$ on the residue field $k_F$, while a geometric Frobenius element is the inverse of an arithmetic Frobenius element. Under the classical (i.e. Artin) normalization of the isomorphism $W_F^\tx{ab} \to F^\times$, arithmetic Frobenius elements map to uniformizing elements. Under the Deligne normalization, geometric Frobenius elements map to uniformizing elements. 

The normalization of the isomorphism $W_F^\tx{ab} \to F^\times$ implies normalization of the Tate-Nakayama duality pairing $H^1(F,T) \times H^1(F,X^*(T)) \to \C^\times$ as well as of the Langlands duality pairing $H^0(F,T) \times H^0_c(W_F,\hat T) \to \C^\times$ for an arbitrary $F$-torus $T$.

In \cite{LS87} is defined an absolute transfer factor (up to a scalar constant multiple)
\[ \Delta : H_1(F)^\tx{sr} \times G(F)^\tx{sr} \to \C,\qquad  \Delta = \Delta_I \cdot \Delta_{II} \cdot \Delta_{III_1} \cdot \Delta_{III_2} \cdot \Delta_{IV}. \]
In order to define it, one has to choose a $z$-extension $H_1 \to H$ as well as an $L$-embedding $\mc{H} \to {^LH_1}$. The transfer factor depends on these choices.

The terms $\Delta_I$ and $\Delta_{III_1}$ involve Tate-Nakayama duality, while the term $\Delta_{III_2}$ involves Langlands duality for tori. In \cite{LS87} both are normalized with respect to the classical normalization of $W_F^\tx{ab} \to F^\times$. According to \cite[\S3.4]{LS87} this factor satisfies
\[ \Delta(\gamma_1,\delta') = \Delta(\gamma_1,\delta) \cdot \<\tx{inv}(\delta,\delta)',\hat\varphi_{\gamma,\delta}^{-1}(s)\>^{-1}, \]
where $\gamma_1 \in H_1(F)^\tx{sr}$ maps to $\gamma \in H(F)^\tx{sr}$, $\delta,\delta' \in G^*(F)^\tx{sr}$, $\varphi : T_\gamma \to T_\delta$ is the unique admissible isomorphism between the centralizers of $\gamma$ and $\delta$ that maps $\gamma$ to $\delta$, and $\tx{inv}(\delta,\delta') \in H^1(F,T_\delta)$ is the class of the element $g^{-1}\sigma(g)$ for any $g \in G(\bar F)$ such that $g\delta g^{-1}=\delta'$.

Following \cite[\S5.1]{KS12} one can also define 
\begin{equation} \label{eq:tf}
\Delta' : H_1(F)^\tx{sr} \times G^*(F)^\tx{sr} \to \C,\qquad  \Delta = \Delta_I^{-1} \cdot \Delta_{II} \cdot \Delta_{III_1}^{-1} \cdot \Delta_{III_2} \cdot \Delta_{IV},
\end{equation}
which then has the property
\[ \Delta'(\gamma_1,\delta') = \Delta'(\gamma_1,\delta) \cdot \<\tx{inv}(\delta,\delta'),\hat\varphi_{\gamma,\delta}^{-1}(s)\>. \]
One reason to introduce $\Delta'$ is that this factor generalizes to the twisted setting, while $\Delta$ does not. The problem is that, in the twisted setting, $\Delta_{III_1}$ and $\Delta_{III_2}$ must be glued to a single factor $\Delta_{III}$ involving hypercohomology, and this gluing necessitates the sign change. Another place where $\Delta'$ appears is the character identities for $p$-adic groups proved in \cite{KalRSP} and \cite{FKS}. The switch from $\Delta$ to $\Delta'$ can be effected by inverting the endoscopic element $s$, since both $\Delta_I$ and $\Delta_{III_1}$ are formed by pairing cohomological invariants with the element $s$.

We mention another clash in notation that is visible in the unnumbered identity 
\[ \Delta'(\gamma_1,\delta') = \<\tx{inv}(\delta,\delta'),\kappa_\delta\>^{-1} \cdot \Delta'(\gamma_1,\delta) \]
which appears below \cite[(5.4.1)]{KS12}, and seems to contradict \eqref{eq:tf}. The reason for the discrepancy is that the notation $\tx{inv}(\delta,\delta')$ used in \cite{KS12} is defined in \cite[\S5.3]{KS99} and is the inverse of the same notation used in \eqref{eq:tf}.

Let us now briefly mention the first meaning of the word ``normalization'', which is, once a choice has been made regarding which version of the local class field theory isomorphism will be used, and whether the endoscopic element $s$ should be inverted or not, how does one pin down the scalar constant multiple. When $G$ is quasi-split, one can either fix a pinning or a Whittaker datum, and these choices lead to the normalizations of $\Delta'$ denoted by $\Delta'_0$ and $\Delta'_\lambda$ in \cite[\S5.5]{KS12}. When $G$ is not quasi-split, one fixes a pinning or Whittaker datum for its quasi-split form, and in addition a rigid inner twist datum, thereby obtaining the normalization of $\Delta'$ defined in \cite[(5.10)]{KalRI}.

Finally, we mention that in this paper we will follow Waldspurger's preference of normalizing orbital integrals and characters by the Weyl discriminant, and hence omit the factor $\Delta_{IV}$ from the transfer factor.

\subsection{Transfer factors as genuine functions on $H(F)_x^\tx{sr} \times G(F)^\tx{sr}$} \label{sub:tf1}

Let $G^*$ be a quasi-split connected reductive $F$-group and let $(\xi,z) : G^* \to G$ be a rigid inner twist. Let $(\dot s,\mc{H})$ be a refined endoscopic pair as in \S\ref{sub:endoreview}. Recall that the refinement is only necessary when working with general rigid inner twists. If on the other hand $G$ is quasi split, $\xi=\tx{id}$, and $z=1$, we can take a usual endoscopic pair $(s,\mc{H})$.

Let $H(F)_x$ be the double cover of $H(F)$ determined by $\mc{H}$ as in \S\ref{sub:l1-covers}. We write $G(F)^\tx{sr}$ and $H(F)^\tx{sr}$ for the subsets of strongly regular semi-simple elements, and $H(F)_x^\tx{sr}$ for the preimage of $H(F)^\tx{sr}$ in $H(F)_x$. We will now define a transfer factor
\[ \Delta'_x : H(F)_x^\tx{sr} \times G(F)^\tx{sr} \to \C. \]
Unlike the definitions of \cite{LS87} or \cite{KS12}, it will not depend on a choice of a $z$-extension $H_1 \to H$ or $L$-embedding $\mc{H} \to {^LG}$. Instead, it uses the canonical double cover $H(F)_x$ and canonical $L$-embedding $^LH_x \to {^LG}$.

As with \cite{LS87} and \cite{KS12}, there will be two ways to normalize it -- by fixing either a pinning or a Whittaker datum of $G^*$. We will furthermore follow Waldspurger's preference of normalizing orbital integrals and characters by the Weyl discriminant, and hence omit the factor $\Delta_{IV}$ from the transfer factor.

Let $\gamma_x \in H(F)_x$ and let $\gamma \in H(F)$ be its image. Let $\delta \in G(F)$. We assume that $\gamma$ and $\delta$ are strongly regular. If they are not related we set $\Delta_x(\gamma_x,\delta)=0$. Otherwise there exists a unique admissible isomorphism $\varphi_{\gamma,\delta}$ between the centralizer of $\gamma$ in $H$ and the centralizer of $\delta$ in $G$. We identify these centralizers via $\varphi_{\gamma,\delta}$ and denote them both by $S$. Then we have $R(S,H) \subset R(S,G) \subset X^*(S)$ and we have the natural embeddings $\xi_{S,H} : S \to H$ and $\xi_{S,G} : S \to G$. 

Let $S(F)_G$ and $S(F)_H$ be the double covers of $S(F)$ associated to the admissible sets $R(S,G)$ and $R(S,H)$. Then $S(F)_H$ is canonically isomorphic to its Baer inverse $S(F)_{-H}$ and $S(F)_{G/H}=S(F)_G \oplus S(F)_{-H}$ is the double cover of $S(F)$ associated to the admissible set $R(S,G) \sm R(S,H)$. Fix a lift $\delta_\pm \in S(F)_{G/H}$ of $\delta$.

Let $S(F)_{x_H}=S(F)_{x_{H,G}}$ be the double cover associated to the character $x_{H,G}$ of \eqref{eq:h-tits}. Again this cover is canonically isomorphic to its Baer inverse $S(F)_{-x_H}$, since $x_H$ is of order $2$. The map $S(F) \to H(F)$ extends canonically to a map $S(F)_{x_H} \to H(F)_x$ by Construction \ref{cns:s-t}. Thus $\gamma_x$ is naturally an element of $S(F)_{x_H}$.

The embeddings $\xi_H : S \to H$ and $\xi_G : S \to G$ form a relatively admissible pair and Construction \ref{cns:rel2} provides a genuine character $\<\tx{inv}(\xi_{S,H},\xi_{S,G}),-\>$ of $S(F)_{G/H} \oplus S(F)_{x_H}$. The absence of the factor $S(F)_{x_G}$ is due to $x_G=1$. Since the embeddings $\xi_H$ and $\xi_G$ depend only on the elements $\gamma_x$ and $\delta_\pm$, we abbreviate this character as 
\begin{equation} \label{eq:d3}
\tx{inv}_\mc{H}(\gamma_x,\delta_\pm) = \<\tx{inv}(\xi_{S,H},\xi_{S,G}),(\delta_\pm,\gamma_x)\>, 
\end{equation}
and note that its value depends only on the stable classes of $\gamma_x$ and $\delta_\pm$.

Consider given a pinning $\tx{pin}=(T,B,\{X_\alpha\})$ of $G^*$. According to Lemma \ref{lem:rel1_desc}, the term
\begin{equation} \label{eq:d1}
\<\tx{inv}((\xi,z,\delta_\pm),\tx{pin}),\varphi_{\gamma,\delta}(\dot s)\>
\end{equation}
is well-defined. Since both \eqref{eq:d3} and \eqref{eq:d1} are genuine functions in the variable $\delta_\pm$, their ratio descends to a function of $\delta$, which allows us to define
\begin{equation} \label{eq:tf_pin}
\Delta'_x[\tx{pin}](\gamma_x,(\xi,z,\delta)) = \<\tx{inv}((\xi,z,\delta_\pm),\tx{pin}),\varphi_{\gamma,\delta}(\dot s)\>^{-1} \cdot \tx{inv}_\mc{H}(\gamma_x,\delta_\pm),
\end{equation}
by choosing an arbitrary lift $\delta_\pm \in S(F)_{G/H}$ of $\delta$.

If instead of a pinning we fix a Whittaker datum $\mf{w}$, then we can choose a pinning $\tx{pin}$ as above and a character $\Lambda : F \to \C^\times$ which together give rise to $\mf{w}$. Following \cite[\S5.3]{KS99} we set
\begin{equation} \label{eq:tf_whit}
\Delta'_x[\mf{w}](\gamma_x,(\xi,z,\delta)) = \epsilon(1/2,X^*(T)_\C-X^*(T^H)_\C,\Lambda) \cdot \Delta_x[\tx{pin}](\gamma_x,(\xi,z,\delta)).	
\end{equation}

\begin{lem} \label{lem:tf1_behavior}
Let $(\xi_i,z_i,\delta_i)$ for $i=1,2$ be stably conjugate. Then 
\begin{eqnarray*}
&&\Delta'_x[\tx{pin}](\gamma_x,(\xi_2,z_2,\delta_2))\\
&=&\Delta'_x[\tx{pin}](\gamma_x,(\xi_1,z_1,\delta_1)) \cdot \<\tx{inv}(\xi_1,z_1,\delta_1),(\xi_2,z_2,\delta_2)),\varphi_{\gamma,\delta_1}(\dot s)\>.
\end{eqnarray*}
The same holds for $\Delta'_x[\mf{w}]$.
\end{lem}
\begin{proof}
By definition \eqref{eq:tf_pin} and the fact that \eqref{eq:d3} depends only on the stable class of $\delta_\pm$ we see that it is enough to consider \eqref{eq:d1}. Using the functoriality of the Tate-Nakayama pairing and Lemma \ref{lem:delta1_kappa}, we obtain
\begin{eqnarray*}
&&\<\tx{inv}((\xi_2,z_2,\delta_{2,\pm})^{-1},\tx{pin}),\varphi_{\gamma,\delta_2}(\dot s)\>\\
&=&\<\varphi_{\gamma,\delta_1}^{-1}(\tx{inv}((\xi_2,z_2,\delta_{2,\pm}),\tx{pin}))^{-1},\varphi_{\gamma,\delta_1}(\dot s)\>\\
&=&\<\tx{inv}((\xi_1,z_1,\delta_{1,\pm}),\tx{pin})^{-1}\cdot \tx{inv}((\xi_1,z_1,\delta_1),(\xi_2,z_2,\delta_2)),\varphi_{\gamma,\delta_1}(\dot s)\>
\qedhere
\end{eqnarray*} 
\end{proof}

The next lemma uses the notion of stable conjugacy of elements of $H(F)_x^\tx{sr}$. Two such elements $\gamma_{i,\pm}$ are called stably conjugate if there exists $g \in H(F^s)$ such that $\tx{Ad}(g)$ is an $F$-isomorphism between their centralizers $S_1 \to S_2$ and such that the induced isomorphism $\tx{Ad}(g) : S_1(F)_H \to S_2(F)_H$ maps $\gamma_{1,\pm}$ to $\gamma_{2,\pm}$.

\begin{lem} \label{lem:tf1_antigen}
In the variable $\gamma_x$ the functions $\Delta'_x[\tx{pin}]$ and $\Delta'_x[\mf{w}]$ are genuine and stably invariant.
\end{lem}
\begin{proof}
The term \eqref{eq:d1} depends only on $\gamma$ and not on $\gamma_x$, and the term \eqref{eq:d3} is genuine in $\gamma_x$ by construction.

The term \eqref{eq:d3} depends only on the stable class of $\gamma_x$ by construction. Changing $\gamma$ by a stable conjugate also does not affect \eqref{eq:d1}, because the change in $\varphi_{\gamma,\delta}$ is by an $H$-admissible isomorphism of maximal tori of $H$, while $\dot s$ is central in $\hat{\bar H}$, and hence fixed by such isomorphisms. 
\end{proof}

\subsection{Transfer factors as genuine functions on $H(F)_{x_H}^\tx{sr} \times G(F)_{x_G}^\tx{sr}$} \label{sub:tf2}

Let $G$ be a quasi-split connected reductive group, $x_G \in Z^2(\Gamma,\hat T \to \hat T_\tx{ad})$, and $(s,\mc{H})$ an endoscopic pair as in \S\ref{sub:endoreview} (since $G$ is assumed quasi-split we do not need a refinement). Let $H(F)_{x_H}$ be the double cover determined by $\mc{H}$ as in \S\ref{sub:l2-covers}. A slight modification of the discussion of \S\ref{sub:tf1} produces a transfer factor
\[ \Delta'_x : H(F)_{x_H}^\tx{sr} \times G(F)_{x_G}^\tx{sr} \to \C. \]
But we need to be mindful when working with Baer inverses and distinguish genuine from anti-genuine behavior, since the characters $x_G$ and $x_H$ need not be of order $2$ any more.

Consider $\gamma_H \in H(F)_{x_H}$ and $\delta_G \in G(F)_{x_G}$ with strongly regular semi-simple images $\gamma \in H(F)$ and $\delta \in G(F)$. As before we identify via $\varphi_{\gamma,\delta}$ the centralizer of $\gamma$ in $H$ with the centralizer of $\delta$ in $G$, and denote both by $S$. Let $S(F)_{x_H}$ and $S(F)_{x_G}$ be the covers of $S(F)$ associated to $x_H$ and $x_G$ by Construction \ref{cns:s-t}. This construction provides natural embeddings $S(F)_{x_H} \to H(F)_{x_H}$ and $S(F)_{x_G} \to G(F)_{x_G}$ lifting the embeddings $S(F) \to H(F)$ and $S(F) \to G(F)$. We thus have the element $(\gamma_H,\delta_G) \in S(F)_{x_H} \oplus S(F)_{x_G}$. 

On the other hand, fix an arbitrary lift $\delta_{G/H} \in S(F)_{G/H}$ of $\delta$, where $S(F)_{G/H}$ is the double cover of $S(F)$ associated to the admissible set $R(S,G) \sm R(S,H)$. We remind ourselves that $S(F)_{G/H}=S(F)_G \oplus S(F)_H$, where $S(F)_G$ and $S(F)_H$ are the double covers of $S(F)$ associated to the admissible sets $R(S,G)$ and $R(S,H)$. We now have the element $(\delta_{G/H},\gamma_H,\delta_G) \in S(F)_{G/H} \oplus S(F)_{x_H} \oplus S(F)_{x_G}$. The genuine character $\<\tx{inv}(\xi_{S,H},\xi_{S,G}),-\>$ of $S(F)_{G/H} \oplus S(F)_{-x_H} \oplus S(F)_{x_G}$ can be interpreted via Remark \ref{rem:antigen} as a function that is genuine in the first and third variable, and anti-genuine in the second. We again write $\tx{inv}_\mc{H}$ for this character. For a fixed $F$-pinning of $G$ we define $\Delta'_x[\tx{pin}](\gamma_H,\delta_G)$ as
\begin{equation} \label{eq:tf1_pin}
\<\tx{inv}((\xi,z,\delta_{G/H}),\tx{pin}),\varphi_{\gamma,\delta}(s)\>^{-1} \cdot \tx{inv}_\mc{H}(\delta_{G/H},\gamma_H,\delta_G).	
\end{equation}
The genuine behavior of both factors in the variable $\delta_{G/H}$ cancels out and the result is independent of the choice of $\delta_{G/H}$.

For a Whittaker datum $\mf{w}$ of $G$ we define $\Delta'_x[\mf{w}](\gamma_H,\delta_G)$ by \eqref{eq:tf_whit} but using \eqref{eq:tf1_pin} in place of \eqref{eq:tf_pin}.

For the next two lemmas let $\Delta_x$ denote either of $\Delta_x[\tx{pin}]$ or $\Delta_x[\mf{w}]$.

\begin{lem} \label{lem:tf2_behavior}
In the variable $\delta_G$, the functions $\Delta_x$ is genuine. If $\delta_i \in G(F)_{x_G}$ for $i=1,2$ are stably conjugate, then 
\[ \Delta_x(\gamma_H,\delta_2)=\Delta_x(\gamma_H,\delta_1) \cdot \<\tx{inv}(\delta_1,\delta_2),\varphi_{\gamma,\delta_1}(s)\>. \]
\end{lem}
\begin{proof}
The proof is the same as for Lemma \ref{lem:tf1_behavior}.
\end{proof}

\begin{lem} \label{lem:tf2_antigen}
In the variable $\gamma_H$, the function $\Delta_x(\gamma_H,\delta_G)$ is anti-genuine and stably-invariant.
\end{lem}
\begin{proof}
The proof is the same as for Lemma \ref{lem:tf1_antigen}.
\end{proof}

\subsection{Comparison with the transfer factors of Langlands--Shelstad--Kottwitz} \label{sub:ks_comp}

Let $G^*$ be a quasi-split connected reductive group and let $(\xi,z) : G^* \to G$ be a rigid inner twist as in \cite[\S5.1]{KalRI} or \cite[\S7.1]{Dillery20}. Let pin be an $F$-pinning of $G^*$. We have defined in \eqref{eq:tf_pin} a transfer factor
\[ \Delta'_x[\tx{pin}] : H(F)_x^\tx{sr} \times G(F)^\tx{sr} \to \C. \]
On the other hand, after fixing a $z$-extension $1 \to Z_1 \to H_1 \to H \to 1$ and an $L$-embedding $\mc{H} \to {^LH_1}$, there one has the Langlands--Shelstad--Kottwitz transfer factor 
\[ \Delta'[\tx{pin}] : H_1(F)^\tx{sr} \times G(F)^\tx{sr} \to \C \]
that was reviewed in \S\ref{sub:endoreview}. 

In order to compare these two transfer factors we form the intermediary group $H_1(F)_x$ that is the pull back of $H_1(F) \to H(F) \from H(F)_x$. As in the proof of Theorem \ref{thm:llc-h}, we have the genuine character $\mu_1 : H_1(F)_x \to \C^\times$ whose restriction to $Z_1(F)$ equals $\lambda_1^{-1}$.

\begin{pro} \label{pro:ks_comp}
Let $\gamma_1 \in H_1(F)$ and $\gamma_x \in H(F)_x$ be lifts of $\gamma \in H(F)$. Let $\gamma_{1,\pm} \in H_1(F)_x$ be the element determined by $(\gamma_1,\gamma_x)$.
\[ \Delta'[\tx{pin}](\gamma_1,\delta) = \Delta'_x[\tx{pin}](\gamma_x,\delta)\cdot \mu_1(\gamma_{1,\pm}). \]
\end{pro}

\begin{rem}
We note first that this identity is consistent in the following sense. According to Lemma \ref{lem:tf1_antigen} the right hand side descends to a non-genuine function of $H_1(F)$. On the other hand, the left-hand side transforms under $Z_1(F)$ by the character $\lambda_1^{-1}$ -- this is \cite[Lemma 3.5.A]{LS90}, except that the character $\lambda_1$ we are using here matches the one constructed in \cite[\S2.2]{KS99} and is the inverse of the character $\lambda$ of \cite[Lemma 3.5.A]{LS90}. 
\end{rem}

\begin{proof}
As a first step, we reduce to the case that $G$ is quasi-split. To do this, choose $\delta^* \in G^*(F)$ that is stably conjugate to $\delta$. As discussed in \cite[\S5.1]{KalRI}, a theorem of Steinberg guarantees the existence of $\delta^*$. Then \cite[(5.1)]{KalRI} and Lemma \ref{lem:tf1_behavior} show that both $\Delta'[\tx{pin}]$ and $\Delta'_x[\tx{pin}]$ have the same behavior when $\delta$ is replaced by $\delta^*$. This allows us to assume that $G=G^*$, $\xi=\tx{id}$, $z=1$. We now proceed by examining how the different pieces of $\Delta'[\tx{pin}]$ and showing how they compare to the pieces of $\Delta_x[\tx{pin}]$. The pieces of $\Delta'[\tx{pin}]$ depend on further choices: $a$-data, $\chi$-data, and element admissible embedding $\tx{Cent}(\gamma,H) \to G$. We will choose the latter to be the admissible isomorphism $\varphi_{\gamma,\delta}$. With that choice, $\Delta_{III_1}=1$.

\ul{$\Delta_I$}: We choose a lift $\delta_\pm \in S(F)_{G/H}$ of $\delta$. Then $a_\alpha = \delta_\alpha-\delta_{-\alpha}$ is a set of $a$-data for $R(S,G/H)_\tx{sym}$, which we extend arbitrarily to a set of $a$-data for $R(S,G)$.  Then 
\[ \Delta_I(\gamma,\delta)=\<\tx{inv}((\xi,z,\delta_\pm),\tx{pin}),\varphi_{\gamma,\delta}(\dot s)\>. \] 

\ul{$\Delta_{II}$}: Choose $\chi$-data for $R(S,G)$. Let $\chi : S(F)_{G/H} \to \C^\times$ be the genuine character relative to the part of the $\chi$-data for $R(S,G/H)$. Then 
\begin{eqnarray*}
\Delta_{II}(\gamma,\delta)&=&\prod_{\alpha \in R(S,G/H)/\Sigma} \chi_\alpha\left(\frac{\alpha(\delta)-1}{a_\alpha}\right)\\
&=&\prod_{\alpha \in R(S,G/H)/\Sigma} \chi_\alpha\left(\frac{\delta_\alpha-\delta_{-\alpha}}{\delta_{-\alpha}a_\alpha}\right)\\
&=&\prod_{\alpha \in R(S,G/H)/\Sigma} \chi_\alpha(\delta_\alpha)\\
&=&\chi(\delta_\pm),
\end{eqnarray*}
where $\chi : S(F)_{G/H} \to \C^\times$ is the genuine character associated to this $\chi$-data as in \cite[\S3.2]{KalDC}, and we have used $a_\alpha=\delta_\alpha-\delta_{-\alpha}$.

\ul{$\Delta_{III_2}$}: We now compare $\Delta_{III_2}$ with $\tx{inv}_\mc{H}$. More previsely, we will show that 
\[ \Delta_{II}(\gamma,\delta) \cdot \Delta_{III_2}(\gamma_1,\delta) \cdot \mu_1(\gamma_{1,\pm})^{-1}  = \tx{inv}_\mc{H}(\gamma_x,\delta_\pm). \] 
We begin by recalling that the construction of $\Delta_{III_2}$ involves the $L$-embedding
\[ \hat S \rtimes W_F \to \hat G \rtimes W_F,\qquad s \rtimes \sigma \mapsto s\cdot r_p^G(\sigma)n_p^G(\sigma_{S,G})\rtimes\sigma \]
constructed in \cite[\S2.6]{LS87}. We have chosen a pinning $(\hat T,\hat B,\{X_{\hat\alpha}\})$ of $\hat G$ and $n_p^G(\sigma_{S,G}) \in N(\hat T,\hat G)$ denotes the corresponding Tits lift of the Weyl element $\sigma_{S,G} \in \Omega(\hat T,\hat G)$, and $r_p^G : W_F \to \hat T$ is the 1-cochain determined by this pinning and the $\chi$-data for $R(S,G)$. The $\hat G$-conjugacy class of this $L$-embedding depends only on the $\chi$-data, but not on the choice of pinning. Analogously we have
\[ \hat S \rtimes W_F \to \hat H \rtimes W_F,\qquad s\rtimes \sigma \mapsto s\cdot r_q^H(\sigma)n_q^H(\sigma_{S,H})\rtimes\sigma, \]
where again this particular presentation depends on a chosen pinning of $\hat H$. To relate these two embeddings we conjugate $\mc{H}$ within $\hat G$ so that $\hat T \subset \hat H$, and we arrange that the Borel subgroup $\hat B \cap \hat H$ is part of the chosen pinning for $\hat H$. Then $q=p$ and moreover $\hat T$ is naturally identified with the dual of the universal torus of $G$, as well as that of $H$. These identifications induce the two usually distinct Galois actions $\sigma_G$ and $\sigma_H$ on $\hat T$. We have $\sigma_H=\sigma_{H,G}\sigma_G$ with $\sigma_{H,G} \in \Omega(\hat T,\hat G)$. Recall the non-homomorphic section \eqref{eq:tits-splitting} of the natural projection $\mc{H} \to W_F$, as well as its modification \eqref{eq:modified-splitting} that involved a choice of $x \in C^1(\Gamma,\hat T)$. The latter was given by $\sigma \mapsto x(\sigma)^{-1}n_p^G(\sigma_{H,G})\rtimes\sigma_G$ and had the property that it preserves the chosen pinning of $\hat H$. There is $r_1 \in C^1(W_F,Z(\hat H_1))$ such that $\sigma \mapsto r_1(\sigma)^{-1}x(\sigma)^{-1}n_p^G(\sigma_{H,G})\rtimes\sigma_G$ is a homomorphic section of the projection $\mc{H}_1 \to W_F$, where $\mc{H}_1$ is the pushout of $Z(\hat H_1) \from Z(\hat H) \to \mc{H}$. There are of course many choices for $r_1$, but one of them is such that
\[ \mc{H} \to \hat H_1 \rtimes W_F,\qquad hx(\sigma)^{-1}n_p^G(\sigma_{H,G})\rtimes\sigma  \mapsto r_1(\sigma)h \rtimes \sigma\]
is the chosen $L$-embedding $\mc{H} \to {^LH_1}$. 

By definition, $\Delta_{III_2}(\gamma_1,\delta)$ is the value at $\gamma_1$ of a certain character of the centralizer $S_1$ of $\gamma_1$. The character in question has the $L$-parameter given by comparing the three $L$-embeddings above. Using the explicit formulas we have presented, the parameter equals
\[ r_p^H(\sigma)n_p^H(\sigma_{S,H})r_1(\sigma)^{-1}x(\sigma)^{-1}n_p^G(\sigma_{H,G})n_p^G(\sigma_{S,G})^{-1}r_p^G(\sigma)^{-1}. \]
We have $\sigma_{S,G}=\sigma_{S,H}\sigma_{H,G}$, and therefore the above can be rewritten as
\[ r_p^G(\sigma)^{-1}r_p^H(\sigma)r_1(\sigma)^{-1}\cdot n_p^H(\sigma_{S,H})x(\sigma)^{-1}n_p^G(\sigma_{H,G})n_p^G(\sigma_{S,G})^{-1}. \]
The 1-cochain $r_p^G(\sigma)r_p^H(\sigma)^{-1}$ is by construction the parameter (relative to the gauge $p$) of the genuine character $\chi$ of $S(F)_{G/H}$. If we denote by $S_1(F)_{G/H}$ the pull-back of $S_1(F) \to S(F) \from S(F)_{G/H}$, then the product $\Delta_{II}(\gamma,\delta)\Delta_{III_2}(\gamma_1,\delta)$ is the value at $(\gamma_1,\delta_\pm)$ of the genuine character of $S_1(F)_{G/H}$ with parameter 
\[ r_1(\sigma)^{-1}\cdot n_p^H(\sigma_{S,H})x(\sigma)^{-1}n_p^G(\sigma_{H,G})n_p^G(\sigma_{S,G})^{-1}. \]
The parameter of the genuine character $\mu_1$ of $H_1(F)_x$ is $r_1(\sigma)^{-1}\boxtimes\sigma$. Let $S_1(F)_{x_H}$ be the pull back of $S_1(F) \to H_1(F) \from H_1(F)_x$. It's $L$-group is $\hat S_1 \boxtimes_{\partial x^{-1}t_p^{G,H}} \Gamma$ according to Fact \ref{fct:maxtori} and Construction \ref{cns:l-s-t}. The restriction to $S_1(F)_{x_H}$ of $\mu_1$ has parameter equal to $r_1(\sigma)^{-1}\boxtimes\sigma$. Therefore the product $\Delta_{II}(\gamma,\delta)\Delta_{III_2}(\gamma_1,\delta)\mu_1(\gamma_{1,\pm})^{-1}$ is the genuine character of the Baer sum $S_1(F)_{-x_H} \oplus S_1(F)_{G/H}$ with parameter
\begin{equation} \label{eq:d3pmpar}
n_p^H(\sigma_{S,H})x(\sigma)^{-1}n_p^G(\sigma_{H,G})n_p^G(\sigma_{S,G})^{-1}.	
\end{equation}

We claim that this equals $\tx{inv}_\mc{H}(\gamma_x,\delta_\pm)$. The construction of the latter involves the $L$-embeddings
\[ ^LS_H \oplus {^LS_{x_H}} = \hat S \boxtimes_{t_p^{S,H}\partial x^{-1}t_p^{H,G}} \Gamma \to \hat H_{\partial x^{-1}t_p^{H,G}}\Gamma,\qquad s\boxtimes\sigma \mapsto sn_p^H(\sigma_{S,H})\boxtimes\sigma, \]
\[ ^LH_x = \hat H \boxtimes_{\partial x^{-1}t_p^{H,G}} \Gamma \to \hat G \rtimes \Gamma,\qquad h \boxtimes\sigma \mapsto hx(\sigma)^{-1}n_p^G(\sigma_{H,G})\rtimes \sigma, \]
and
\[ ^LS_G = \hat S \boxtimes_{t_p^{S,G}} \Gamma \to \hat G \rtimes \Gamma,\qquad s\boxtimes\sigma \mapsto sn_p^G(\sigma_{S,G})\rtimes\sigma. \]
Comparing the composition of the first two with the third we see that the parameter of $\tx{inv}_\mc{H}$ is given precisely by \eqref{eq:d3pmpar}.
\end{proof}

\subsection{Transfer of orbital integrals and characters between $G(F)_{x_G}$ and $H(F)_{x_H}$}
Let $G^*$ be a quasi-split connected reductive group and let $(\xi,z) : G^* \to G$ be a rigid inner twist. Let $x_G \in Z^2(\Gamma,\hat T \to \hat T_\tx{ad})$. If $x_G\neq 1$ we set assume $G=G^*$, $\xi=\tx{id}$, $z=1$. Fix a Whittaker datum $\mf{w}$ for $G^*$.

For any $g \in G(F)$ the automorphism $\tx{Ad}(g)$ of $G(F)$ lifts naturally to the cover $G(F)_{x_G}$, namely as conjugation by an arbitrary lift of $g$ in $G(F)_{x_G}$. In this way, we obtain an action of $G(F)$ on $G(F)_{x_G}$ by conjugation.

Let $\delta_G \in G(F)_{x_G}$ be an element whose image $\delta \in G(F)$ is strongly regular semi-simple, and let $S \subset G$ be its centralizer. For any $f \in \mc{C}_c^\infty(G(F)_{x_G})$ we can consider the orbital integral
\[ O_{\delta_G}(f) = \int_{G(F)/S(F)}f(g\delta_Gg^{-1})dg. \]
Assume now that $f$ is anti-genuine, i.e. $f(\epsilon \tilde g)=\epsilon^{-1}f(\tilde g)$. Then $O_{\epsilon\delta_G}(f)=\epsilon^{-1}O_{\delta_G}(f)$.

\begin{dfn}
We say that an anti-genuine $f^H \in \mc{C}_c^\infty(H(F)_{x_H})$ matches $f$ if for all strongly regular $\gamma_H \in H(F)_{x_H}$ we have
\[ SO_{\gamma_H}(f^H) = \sum_\delta \Delta(\gamma_H,\delta_G) O_{\delta_G}(f), \]
where the sum runs over the $G(F)$-conjugacy classes of strongly regular semi-simple elements $\delta \in G(F)$ and $\delta_G \in G(F)_{x_G}$ is an arbitrary lift of $\delta$.
\end{dfn}

\begin{thm} \label{thm:orbit}
For every anti-genuine function $f \in \mc{C}^\infty_c(G(F)_{x_G})$ there exists a matching anti-genuine function $f^H \in \mc{C}^\infty_c(H(F)_{x_H})$.
\end{thm}

We continue with an anti-genuine $f \in \mc{C}_c^\infty(G(F)_{x_G})$. Given an admissible genuine representation $\pi_G$ of $G(F)_{x_G}$ we can form the operator
\[ \pi(f) = \int_{G(F)} f(\tilde g)\pi(\tilde g)vdg, \]
where $\tilde g \in G(F)_{x_G}$ is an arbitrary lift of $g$. This operator is of trace class and we define
\[ \Theta_\pi(f) = \tx{tr}\,\pi(f). \]

Given a tempered $L$-parameter $\varphi^G : L_F \to {^LG_{x_G}}$ let 
\[ S\Theta_{\varphi_G}=\sum_{\pi \in \Pi_{\varphi^G}} \dim(\rho_\pi)\cdot \Theta_\pi, \] 
where $\rho_\pi \in \tx{Irr}(\pi_0(S_\varphi^+))$ correspond to $\pi$ via the refined local Langlands correspondence, i.e. \cite[\S5.4]{KalRI} when $x_G=1$, or Theorem \ref{thm:llc-h} when $x_G \neq 1$. More generally, for $\dot s \in S_\varphi^+$, let
\[ \Theta_{\varphi_G}^{\dot s,\mf{w}}=\sum_{\pi \in \Pi_{\varphi^G}} \tx{tr}\,(\rho_\pi(\dot s))\cdot \Theta_\pi. \]
Let $\mc{H}=\tx{Cent}(s,\hat G) \cdot \varphi(W_F)$. Then $(\dot s,\mc{H})$ is a refined endoscopic datum for $G$. Let $H(F)_{x_H}$ be the cover of $H(F)$ constructed in \S\ref{sub:l2-covers}  and let $^L\xi_{H,G} : {^LH_{x_H}} \to {^LG_{x_G}}$ be the $L$-embedding \eqref{eq:liso1}. Then $\varphi^G={^L\xi_{H,H}}\circ \varphi^H$ for a tempered $L$-parameter $\varphi_H : L_F \to {^LH_{x_H}}$.

\begin{cnj} \label{cnj:charid}
If $f \in \mc{C}_c^\infty(G(F)_{x_G})$ and $f^H \in \mc{C}_c^\infty(H(F)_{x_H})$ are matching anti-genuine functions, then
\[ \Theta^{s,\mf{w}}_{\varphi_G}(f) = S\Theta_{\varphi_H}(f^H). \]
Equivalently,
\[ \Theta^{s,\mf{w}}_{\varphi_G}(\delta_G) = \sum_\gamma \Delta'_x[\mf{w}](\gamma_H,\delta_G)S\Theta_{\varphi_H}(\gamma_H), \]
where the sum runs over the strongly regular semi-simple elements $\gamma$ of $H(F)$ up to stable conjugacy, and $\gamma_H \in H(F)_{x_H}$ is an arbitrary lift of $\gamma$.
\end{cnj}

\begin{thm} \label{thm:charid}
Assume that \cite[Conjecture G]{KalSimons} holds for all $z$-extensions of $G$. Then Conjecture \ref{cnj:charid} holds.
\end{thm}

\subsection{Proofs of Theorems \ref{thm:orbit} and \ref{thm:charid}}

The arguments follow the line of the proofs of Theorem \ref{thm:llc-h} and Proposition \ref{pro:ks_comp}.

Choose a $z$-extension $1 \to A_1 \to G_1 \to G \to 1$ and form the fiber product $G_1(F)_{x_G}$ of $G_1(F) \to G(F) \from G(F)_{x_G}$. According to Proposition \ref{pro:gsc_abelian} and Corollary \ref{cor:gsc_abelian} there exists a genuine character $\mu_G : G_1(F)_{x_G} \to \mb{S}^1$. If $f^G$ is a complex valued anti-genuine function on $G(F)_{x_G}$, then after pulling back to $G_1(F)_{x_G}$ and multiplying with $\mu_G$ we obtain the function $f^{G_1}=f^G \cdot \mu_G$ of $G_1(F)_{x_G}$ that descends to $G_1(F)$. It is $\lambda_G$-antigenuine, where $\lambda_G$ is the restriction of $\mu_G^{-1}$ to $A_1(F)$. 

This provides a bijection between $x_G$-antigenuine functions on $G(F)_{x_G}$ and $\lambda_G$-antigenuine functions on $G_1(F)$. In the same way, but using $\mu_G^{-1}$, one obtains a bijection between $x_G$-genuine representations of $G(F)_{x_G}$ and $\lambda_G$-genuine representations of $G_1(F)$.

Dually, let $r_1 : W_F \to Z(\hat G_1)$ be the parameter of the anti-genuine character $\mu_G^{-1}$. Then 
\[ ^LG_{x_G} \to {^LG_1},\quad g \boxtimes \sigma \mapsto gr_1(\sigma) \rtimes \sigma \]
is an $L$-embedding via which the $L$-parameters valued in $^LG_{x_G}$ correspond to those $L$-parameters valued in $^LG_1$ whose composition with the projection $\hat G_1 \rtimes W_F \to (\hat G_1/\hat G) \rtimes W_F$ equals $\varphi_{\lambda_G}$.

Let $(\dot s,\mc{H})$ be a refined endoscopic pair for $G$. Setting $\mc{H}_1 = \mc{H} \cdot Z(\hat G_1)$ we obtain a refined endoscopic pair $(\dot s,\mc{H}_1)$ for $G_1$. The group $H_1$ comes equipped with a surjective homomorphism $H_1 \to H$ whose kernel is $A_1$. Its derived subgroup need not be simply connected, so $H_1 \to H$ need not be a $z$-extension. Choose a $z$-extension $1 \to A_2 \to H_2 \to H_1 \to 1$. Then $H_2 \to H$ is a surjection whose kernel $A$ is a torus that is an extension of $A_1$ by $A_2$, and thus cohomologically trivial. In particular, $H_2(F) \to H(F)$ is surjective.

Let $H_2(F)_{x_H}$ be the fiber product of $H_2(F) \to H(F) \to H(F)_{x_H}$. According to Proposition \ref{pro:gsc_abelian} and Corollary \ref{cor:gsc_abelian} there exists a genuine character $\mu_H : H_2(F)_{x_H} \to \mb{S}^1$. As for $G$, the map $f^H \mapsto f^H \cdot \mu_H =: f^{H_2}$ induces a bijection between $x_H$-antigenuine functions on $H(F)_{x_H}$ and $\lambda_H$-antigenuine functions on $H_2(F)$, where $\lambda_H^{-1}$ is the restriction of $\mu_H$ to $A(F)$. The same holds for representations.

Dually, let $r_2 : W_F \to Z(\hat H_2)$ be the parameter of the anti-genuine character $\mu_H^{-1}$. Then
\[ ^LH_{x_H} \to {^LH_2},\quad h \boxtimes \sigma \mapsto hr_2(\sigma) \rtimes \sigma \]
is an $L$-embedding via which the $L$-parameters valued in $^LH_{x_H}$ correspond to those $L$-parameters valued in $^LH_2$ whose composition with the projection $\hat H_2 \rtimes W_F \to (\hat H_2/\hat H) \rtimes W_F$ equals $\varphi_{\lambda_H}$.

\begin{lem} 
Let $\mc{H} \to {^LH_2}$ be the composition of above $L$-embedding with the inverse of \eqref{eq:liso1}. The Langlands--Shelstad--Kottwitz transfer factor corresponding to that $L$-embedding and the transfer factor \eqref{eq:tf_pin} are related as follows.

Let $\gamma_2 \in H_2(F)$ and $\gamma_H \in H(F)_{x_H}$ be lifts of $\gamma \in H(F)$ and let $\gamma_{2,x_H} \in H_2(F)_{x_H}$ be the element determined by $(\gamma_2,\gamma_H)$. Let $\delta_1 \in G_1(F)$ and $\delta_G \in G(F)_{x_G}$ be lifts of $\delta \in G(F)$ and let $\delta_{1,x_G} \in G_1(F)_{x_G}$ be the element determined by $(\delta_1,\delta_G)$. Then
\[ \Delta'(\gamma_2,\delta_1) = \Delta'_x(\gamma_H,\delta_G)\cdot \mu_H(\gamma_{2,x_H})\mu_G(\delta_{1,x_G})^{-1}, \]
where each factor is normalized either by a pinning of by a Whittaker datum.
\end{lem}
\begin{proof}
Similar to Proposition \ref{pro:ks_comp} and left to the reader.
\end{proof}
Let $f^{H_2} \in \mc{C}_c^\infty(H_2(F))$ be the transfer of $f^{G_1}$ with respect to the transfer factor $\Delta'[\tx{pin}]$. Then $f^{H_2}$ is $\lambda_H$-antigenuine
and $f^H=f^{H_2} \cdot \mu_H^{-1}$ is a $x_H$-genuine function on $H(F)_{x_H}$. We have
\begin{eqnarray*}
SO_{\gamma_H}(f^H)&=&\mu_H^{-1}(\gamma_{2,x_H})SO_{\gamma_2}(f^{H_2})\\
&=&\mu_H^{-1}(\gamma_{2,x_H})\sum_{\delta_1}\Delta'(\gamma_2,\delta_1) O_{\delta_1}(f^{G_1})\\
&=&\sum_{\delta_1}\Delta'(\gamma_2,\delta_1) \mu_H^{-1}(\gamma_{2,x_H})\mu_G(\delta_{1,x_G})O_{\delta_G}(f^G)\\
&=&\sum_{\delta_G}\Delta'_x(\gamma_H,\delta_G) O_{\delta_G}(f^G),\\
\end{eqnarray*}
proving Theorem \ref{thm:orbit}.

Assuming the character identity for the group $G_1$ and its endoscopic group $H_2$ we have
\begin{eqnarray*}
\Theta_{\varphi_G}^{s,\mf{w}}(\delta_G)&=&\mu_G^{-1}(\delta_{1,x_G})\Theta_{\varphi_{G_1}}^{s,\mf{w}}(\delta_1)\\
&=&\mu_G(\delta_{1,x_G})\sum_{\gamma_2}\Delta'[\mf{w}](\gamma_2,\delta_1)S\Theta_{\varphi_{H_2}}(\gamma_2)\\
&=&\sum_{\gamma_2}\Delta'[\mf{w}](\gamma_2,\delta_1)\mu_G(\delta_{1,x_G})\mu_H^{-1}(\gamma_{2,x_H}) S\Theta_{\varphi_H}(\gamma_H)\\
&=&\sum_{\gamma_H}\Delta'_x[\mf{w}](\gamma_H,\delta_G)S\Theta_{\varphi_H}(\gamma_H),\\
\end{eqnarray*}
proving Theorem \ref{thm:charid}.

\begin{appendices}

\renewcommand{\thesubsection}{\Alph{subsection}}

\numberwithin{thm}{subsection}

\subsection{Review of some bits of Pontryagin duality} \label{app:cag}

Let $A$ be a compact abelian group (assumed Hausdorff) and let $A^\circ \subset A$ be the connected component of the identity. Then $A^\circ$ is a closed connected subgroup, and $\pi_0(A)=A/A^\circ$ is a topological group that is Hausdorff, compact, and totally disconnected, hence profinite. It is finite if and only if $A^\circ$ is open. The group $A^\circ$ is divisible (\cite[Theorem 24.25]{HRv1e2}), in fact the largest divisible subgroup of $A$.

Let $X=\tx{Hom}_\tx{cts}(A,\mb{S}^1)$ be Pontryagin dual of $A$, endowed with the usual compact open topology. Then $X$ is a discrete subgroup, and $X[\infty]$ and $A^\circ$ are mutual annihilators (\cite[Corollary 24.20]{HRv1e2}), where $X[\infty]$ denotes the torsion subgroup of $X$. Thus $X[\infty]$ is naturally identified with the Pontryagin dual of $\pi_0(A)$ and $X/X[\infty]$ with the Pontryagin dual of $A^\circ$.

\begin{lem} \label{lem:surjfin}
A continuous character $\chi : A \to \mb{S}^1$ is either surjective, or of finite order. The latter is the case precisely when $\chi$ kills $A^\circ$.
\end{lem}
\begin{proof}
The image of the restriction of $\chi$ to $A^\circ$ is a closed connected subgroup of $\mb{S}^1$. Thus it either equals $\mb{S}^1$, or $\{1\}$. In the latter case, $\chi$ descends to a continuous character of the profinite group $\pi_0(A)$, and is thus a torsion element of $X$.
\end{proof}

The rank $X/X[\infty]$ of the torsion-free group (whether it be a finite natural number, or infinity) is equal to the Lebesgue covering dimension of $A^\circ$, cf. \cite[Theorem 24.28]{HRv1e2}. For the definition of Lebesgue covering dimension, cf. \cite[\S3.11]{HRv1e2}.

\begin{lem} \label{lem:torchartor}
Let $B$ be locally compact abelian group that satisfies one of the following conditions:
\begin{enumerate}
	\item $B$ has a compact open subgroup $K$ such that $B/K$ is finitely generated.
	\item $B$ has finitely many connected components.
\end{enumerate}
Then any continuous character $\chi : B \to \mu_\infty(\C)$ has finite order.
\end{lem}

An example of such a group $B$ is provided by $B=S(F)$ where $F$ is a local field and $S$ is an $F$-torus.

\begin{proof}
Let $B' \subset B$ be the open subgroup that equals $K$ in the first case, and $B^\circ$ in the second case. According to Lemma \ref{lem:surjfin} in the first case, and preservation of connectedness under continuous maps in the second case, $\chi|_{B'}$ has finite order, say $n$. Then $\chi^n : B/B' \to \mu_\infty(\C)$ also has finite order, since $B/B'$ is finitely generated.
\end{proof}

Recall the following two basic constructions. If $A$ is a discrete abelian group, its profinite completion is a homomorphism $f : A \to B$ with $B$ profinite such that any other homomorphism $g : A \to C$ with $C$ profinite factors uniquely through $f$. The pair $(f,B)$ is unique up to unique isomorphism, and $B$ can be constructed as $\varprojlim A/N$ where $N$ runs over all finite index subgroups. The homomorphism $f$ has dense image. It is is injective if and only if $A$ is residually finite. A finitely generated abelian group is residually finite.

If $A$ is a locally compact abelian group, its Bohr compactification is a homomorphism $f : A \to B$ with $B$ compact (assumed Hausdorff), such that any other homomorphism $g : A \to C$ with $C$ compact (and Hausdorff) factors uniquely though $f$. The pair $(f,B)$ is unique up to unique isomorphism, and $B$ can be constructed as the closure of the image of the homomorphism
\[ A \to \prod_{A^*} \mb{S}^1,\qquad a \mapsto (\chi(a))_\chi. \]
The homomorphism $f$ has dense image. Note that the Bohr compactification commutes with finite products.

It follows directly from the universal properties that when $A$ is discrete the component group of the Bohr completion of $A$ is the profinite completion of $A$.

\begin{fct} \label{fct:ptfct}
Let $f : A \to B$ be continuous homomorphism of locally compact abelian groups and let $f^* : B^* \to A^*$ be its Pontryagin dual. Then
\begin{enumerate}
	\item $f$ is a closed injection if and only if $f^*$ is an open surjection.
	\item $f$ is injective if and only if $f^*$ has dense image.
	\item $f$ is the discretization of the compact group $B$ if and only if $f^*$ is the Bohr compactification of the discrete group $B^*$.
	\item $f$ is the profinite completion of the discrete group $A$ if and only if $f^*$ is the inclusion of the torsion subgroup (taken with discrete topology) of the compact group $A^*$.
\end{enumerate}	
\end{fct}
\begin{proof}[References]
(1) and (2) \cite[Theorem 23.25]{HRv1e2}. (3) \cite[Theorem 26.12]{HRv1e2}. (4) from (3) and the preceding remark.
\end{proof}

Consider a complex diagonalizable group $D$. We have the decomposition $D = D^c \times D^v$, where $D^c$ is the maximal compact subgroup, and $D^v$ is the maximal real vector space. More precisely, $D=\tx{Hom}_\Z(X^*(D),\C^\times)$ and $D^c=\tx{Hom}_\Z(X^*(D),\mb{S}^1)$, $D^v=\tx{Hom}_\Z(X^*(D),\R_{>0})$. 

\begin{lem} \label{lem:diag}
Let $D_\tx{disc}$ by the abstract group $D$ equipped with the discrete topology. Consider the homomorphism $D_\tx{disc} \to D \to D^c$. The composition of its Pontryagin dual with the projection $(D_\tx{disc})^* \to \pi_0((D_\tx{disc})^*)$ is the profinite completion of $X^*(D)$.	
\end{lem}
\begin{proof}
The Pontryagin dual of $D^c$ is $X^*(D)$. The homomorphism $D_\tx{disc} \to D^c$ induces the identity on torsion subgroups and $\pi_0((D_\tx{disc})^*)^*$ is the torsion subgroup of $D_\tx{disc}$, so the claim follows from Fact \ref{fct:ptfct}.
\end{proof}

\begin{exa}
Consider $D=\C^\times$. Then $D^c=\mb{S}^1$ and $D^v=\R_{>0} \cong \R$, and $X^*(D)=\Z$. Now $(D_\tx{disc})^*=(\mb{S}^1_\tx{disc})^* \times (\R_\tx{disc})^*$. The second factor is connected by Fact \ref{fct:ptfct}. 
We have the exact sequence
\[ 0 \to \Q/\Z \to \mb{S}^1_\tx{disc} \to \R/\Q \to 0, \]
whose dual realizes the connected group $(\R/\Q)^*$ as the identity component of $(\mb{S}^1_\tx{disc})^*$ and the profinite group $(\Q/\Z)^*=\hat\Z$ as its component group. The latter is the profinite completion of $X^*(D)=\Z$.
\end{exa}

\subsection{Review of Langlands duality for tori over local fields} \label{app:langtori}

Let $S$ be an $F$-torus and $\hat S=\tx{Hom}(X_*(S),\C^\times)$ its complex dual. Let $K/F$ be the splitting extension of $S$. Equip $\hat S$ with its analytic topology and let $X$ be a topological group acting continuously on $\hat S$. A 1-cocycle of $X \to \hat S$ is continuous if and only if it is continuous at $1 \in X$. If $X$ is locally profinite, then the set of continuous 1-cocycles $X \to \hat S$ is the same whether we equip $\hat S$ with its analytic topology or its discrete topology, since the analytic topology on $\hat S$ does not admit small subgroups. This is the case with $X=\Gamma$ or $X=\Gamma_{K/F}$. When $F$ is non-archimedean, this is also the case with $W_F$ or $X=W_{K/F}$. But when $F$ is archimedean, $W_F$ and $W_{K/F}$ are not locally profinite. We will write $Z^1(X,\hat S)$ for the set of continuous 1-cocycles $X \to \hat S$, and $H^1(X,\hat S)$ for the group of cohomology classes of continuous 1-cocycles.

Restriction provides an inclusion $Z^1(\Gamma,\hat S) \to Z^1(W_F,\hat S)$, and inflation provides an isomorphism $Z^1(W_{K/F},\hat S) \to Z^1(W_F,\hat S)$. The groups of coboundaries $B^1(\Gamma,\hat S)=B^1(W_F,\hat S)=B^1(W_{K/F},\hat S)$ are all the same, and we obtain the inclusion $H^1(\Gamma,\hat S) \to H^1(W_F,\hat S)$ and the isomorphism $H^1(W_{K/F},\hat S) \to H^1(W_F,\hat S)$.

Langlands duality is the isomorphism $H^1(W_F,\hat S) \to \tx{Hom}_\tx{cts}(S(F),\C^\times)$, functorial in $S$, cf. \cite{Lan97}, \cite[\S6]{Lab85}. If $E/F$ is a finite Galois extension, then restriction along $W_E \subset W_F$ is identified with composition with the norm map $S(E) \to S(F)$, and corestriction along $W_E \to W_F$ is identified with restriction along $S(F) \subset S(E)$. The subgroup $H^1(\Gamma,\hat S)$ is identified with the group of characters of $S(F)$ that are trivial on the image of the norm map $S(E) \to S(F)$ for some finite Galois extension $E/F$. 

The abelian Lie group $\hat S$ has a unique maximal compact subgroup $\hat S^c$. In fact, we have the canonical $\Gamma$-stable decomposition $\hat S = \hat S^c \times \hat S^v$, where $\hat S^v$ is a finite-dimensional $\R$-vector space. This decomposition is obtained from the decomposition $\C^\times = \mb{S}^1 \times \R_{>0}$ and the isomorphism $\exp : \R \to \R_{>0}$ as
\[ \hat S = \tx{Hom}(X_*(S),\C^\times) = \tx{Hom}(X_*(S),\mb{S}^1) \times \tx{Hom}(X_*(S),\R), \]
thus
\[ \hat S^c = \tx{Hom}(X^*(S),\mb{S}^1),\qquad \hat S^v=\tx{Hom}(X_*(S),\R)=X^*(S)\otimes_\Z \R. \]

On the other hand, the decomposition $\C^\times = \mb{S}^1 \times \R_{>0}$ also induces a decomposition $\tx{Hom}_\tx{cts}(S(F),\C^\times) = \tx{Hom}_\tx{cts}(S(F),\mb{S}^1) \times \tx{Hom}_\tx{cts}(S(F),\R_{>0})$. Langlands duality is compatible with these decompositions, as the following lemma shows.

\begin{lem} \label{lem:langunit}
The following are equal.
\begin{enumerate}
	\item The group of unitary characters of $S(F)$.
	\item The subgroup $H^1_u(W_F,\hat S)$ of $H^1(W_F,\hat S)$ of cohomology classes representable by a 1-cocycle whose image in $\hat S$ is bounded in the analytic topology.
	\item The subgroup of $H^1(W_F,\hat S)$ of cohomology classes all of whose representing 1-cocycles have image in $\hat S$ that is bounded in the analytic topology.
	\item The subgroup $H^1(W_F,\hat S^c)$ of $H^1(W_F,\hat S)$.
\end{enumerate}
\end{lem}
\begin{proof}
The equality of (1) and (2) is well-known from \cite{Lan97}. The equality of (2) and (3) follows from the identity $B^1(W_F,\hat S)=B^1(\Gamma_{K/F},\hat S)$. For the equality of (3) and (4) we consider $z \in Z^1(W_F,\hat S)$. Its image is bounded if and only if $z(W_K) \subset \hat S$ is bounded. Since $z(W_K)$ is a subgroup of $\hat S$, it must lie in $\hat S^c$. Composing $z$ with the projection $\hat S \to \hat S^v$ we obtain an element of $Z^1(\Gamma_{K/F},\hat S^v)$. Such an element is cohomologically trivial, since $\hat S^v$ is a finite-dimensional $\R$-vector space. Therefore $z$ can be modified by a coboundary to have trivial projection to $\hat S^v$.
\end{proof}

\subsection{Some remarks on the cohomology of the Weil group} \label{app:weil}

We collect here some facts and observations about the cohomology of the Weil group $W_F$ with coefficients in a topological $W_F$-module $M$. This is an application of the following general situation: $G$ is a topological group, $M$ is an abelian topological group, and there is a continuous action of $G$ on $M$ given by a continuous map $G \times M \to M$ that satisfies the usual axioms of an action map.

There are in fact multiple different kinds of cohomology that can be considered in this situation. One class of such cohomology theories is obtained by considering the cohomology of a complex of cochains of $G$ valued in $M$, equipped with the usual differential, but where the cochains are demanded to be measurable, or continuous, or differentiable (provided both $G$ and $M$ have differentiable structure). The continuous or differentiable case was introduced and studied by Hochschild and Mostow \cite{HM62}. Our primary interest will be the continuous case, and we will denote the corresponding cohomology groups by $H^i_{c}(G,M)$. Our interest in $H^i_{c}(G,M)$ is due to Langlands' duality theorem in the case that $G=W_F$ is the Weil group of a local or global field and $M=\hat T$ is the complex dual of an $F$-torus, cf. Appendix \ref{app:langtori}.

The measurable case was introduced and studied by Moore \cite{MooreCohIII}, where it is assumed that $G$ is locally compact second countable and $M$ is polonais, cf. \cite[\S2]{MooreCohIII}. In fact, Moore introduces two definitions -- one where cochains are measurable functions, and one where cochains are equivalence classes where two measurable functions are considered equivalent if they are equal almost everywhere. Moore shows that the two definitions lead to the same cohomology group \cite[Theorem 5]{MooreCohIII}, but the second version has the advantage of endowing the cohomology groups with a natural topology \cite[\S4,\S5]{MooreCohIII}. We will denote these cohomology groups by $H^i_{m}(G,M)$. The advantage of $H^i_{m}(G,M)$ is that, given a closed normal subgroup $H \subset G$ there is a Hochschild--Serre spectral sequence $E_r^{p,q}$ that converges to $H_{m}^*(G,M)$, and such that $E_2^{p,q}=H_{m}^p(G/H,H_{m}^q(H,M))$ for those $q$ for which $H_{m}^q(H,M)$ is Hausdorff \cite[Theorem 9]{MooreCohIII}. The latter condition is automatic when $q=0,1$, but may also hold for other $q$, in particular when $H_{m}^q(H,M)=0$.

For this reason it is useful to know when $H_{m}^i(G,M)=H_{c}^i(G,M)$. This question was studied by Wigner in his thesis \cite{Wig73}. From his work one can extract the following.

\begin{lem} \label{lem:cc_comp}
Let $M$ be one of the following:
\begin{enumerate}
	\item A discrete abelian group
	\item a finite-dimensional real or complex vector space
	\item a complex diagonalizable group.
\end{enumerate}
Then $H_{m}^i(G,M)=H_{c}^i(G,M)$ for any closed subgroup $G$ of $W_F$.
\end{lem}
\begin{proof}
We apply  \cite[Theorem 2]{Wig73} and \cite[Proposition 3]{Wig73}. It is clear that in all three cases $M$ is a locally connected complete metric topological group. It is enough to show that $G$ is locally compact, $\sigma$-compact, and of finite Lebesgue covering dimension. Local compactness and $\sigma$-compactness are immediate. When $F=\C$ or $F=\R$, then $W_F$ is a complex Lie group, and the same holds for $G$, so its Lebesgue covering dimension is equal to its dimension as a (real) Lie group, which is $2$. When $F$ is non-archimedean, then $W_F$ is topologically the countable disjoint union of profinite topological spaces, and the same holds for $G$. According to \cite[Chap. 3, Theorem 3.5]{HRv1e2}, the Lebesgue covering dimension of $G$ is zero.
\end{proof}

\begin{lem} \label{lem:cc_infres}
Let $H \subset G$ be a closed normal subgroup. Then the inflation-restriction sequence
\[ \begin{aligned}
1 \to& H_{m}^1(G/H,M^H) \to H_{m}^1(G,M) \to H_{m}^1(H,M)^{G/H}\\
 \to& H^2_{m}(G/H,M^H)\to H^2_{m}(G,M)
\end{aligned}\]
is exact. If in addition $H^2(H,M)$ is Hausdorff (in particular, if it is trivial), then the longer inflation-restriction sequence
\[ \begin{aligned}
1 \to& H_{m}^1(G/H,H^0_m(H,M)) \to H_{m}^1(G,M) \to H^0_m(G/H,H_{m}^1(H,M))\\
 \to& H^2_{m}(G/H,H^0_m(H,M)) \to H^2_{m}(G,M)_1 \to H^1_{m}(G/H,H_{m}^1(H,M))\\
 \to& H^3_{m}(G/H,H^0_m(H,M))
\end{aligned} \]
is exact, where $H^2_{m}(G,M)_1$ is the kernel of the restriction map $H^2_{m}(G,M) \to H^2_{m}(H,M)$.
\end{lem}
\begin{proof}
This is the 5-term, respectively 7-term, exact sequence associated to the spectral sequence \cite[Theorem 9]{MooreCohIII}, together with the identification of the corresponding terms $E_2^{p,q}$.
\end{proof}

\begin{lem} \label{lem:h2van}
Let $V$ be a finite-dimensional real or complex vector space on which $W_F$ acts through the map $W_F \to \Gamma$ by linear automorphisms. Then $H^2_c(W_F,V)=0$.
\end{lem}
\begin{proof}
Using Lemma \ref{lem:cc_comp} we will identify $H_c(G,V)$ with $H_m(G,V)$ for $G$ being $W_F$ or any closed subgroup of it, and simply write $H(G,V)$.

We first note that if $K$ is a compact topological group acting continuously on $V$, then $H^i(K,V)=0$. This is \cite[Proposition IX.1.12]{BorWal} in the case that $V$ is complex (cf. the convention on p. 169 of loc. cit.). When $V$ is real, we have $H^i(K,V_\C)=0$ where $V_\C=V\otimes_\R \C$ is the complexification of $V$. But $V_\C=V\oplus V$ as a $K$-module, hence $H^i(K,V)$ is a direct summand of $H^i(K,V_\C)$ and therefore also zero. 

We now take $K$ to be a closed normal compact subgroup of $W_F$ that satisfies $H^1(K,V)=0$ and $H^2(K,V)=0$. From Lemma \ref{lem:cc_infres} we obtain the surjective homomorphism $H^2(W_F/K,V^K) \to H^2(W_F,V)$. But \cite[Theorem 10]{MooreCohIII} identifies $H^2(W_F/K,V^K)$ with the group of topological extensions of $W_F/K$ by $V^K$. So it is enough to find $K$ for which any topological extension of $W_F/K$ by $V^K$ splits (as an extension of topological groups).

When $F$ is non-archimedean and take $K$ to be the inertia subgroup of $W_F$. Since it is compact, we have $H^1(K,V)=H^2(K,V)=0$ by the above paragraph. Moreover $W_F/K=\Z$ and the splitting claim is clear. 

When $F=\C$ we take $K=\mb{S}^1$. Again by the above paragraph $H^1(K,V)=H^2(K,V)=0$. Moreover $W_F/K=\R_{>0} \cong \R$ acting trivially on $V^K=V$, and the splitting claim is again clear.

When $F=\R$, we take again $K=\mb{S}^1$ and the previous argument reduces to showing $H^2(W_F/K,V)=0$. But $W_F/K=\R_{>0} \times \Gamma=\R \times \Gamma$. We now take $K=\Gamma$. Since $K=\Z/2\Z$, the groups $H^i(K,V)$ are the usual Galois cohomology groups, and their vanishing for $i>0$ is immediate from the unique divisibility of $V$. Therefore we again reduce to the obvious vanishing of $H^2(\R,V^K)$.
\end{proof}

\begin{pro}[Rajan, Karpuk] \label{pro:cc_rajan}
Let $T$ be an $F$-torus. Then $H^2_c(W_F,\hat T)=0$.
\end{pro}
\begin{proof}
This is \cite[Theorem 2]{Raj04}, who proves it for $H^2_m$, but it holds equally for $H^2_c$ by Lemma \ref{lem:cc_comp}. A different proof in the non-archimedean case was given in \cite[Theorem 3.2.2]{Karpuk}.
\end{proof}

\begin{lem} \label{lem:cc_surj_na}
Let $F$ be non-archimedean and $D$ be a complex diagonalizable group with algebraic $\Gamma$-action. Then 
\[ H^2_c(\Gamma,D^\circ) \to H^2_c(\Gamma,D) \to H^2_c(W_F,D) \to 0 \] 
is exact and $H^2_c(W_F,D) \to H^2_c(W_F,\pi_0(D))$ is an isomorphism.
\end{lem}
\begin{proof}
We consider the commutative diagram with exact rows
\[ \xymatrix{
H_c^2(\Gamma,D^\circ)\ar[r]\ar[d]&H_c^2(\Gamma,D)\ar[r]\ar[d]&H_c^2(\Gamma,\pi_0(D))\ar[r]\ar[d]&H_c^3(\Gamma,D^\circ)\ar[d]\\
H_c^2(W_F,D^\circ)\ar[r]&H_c^2(W_F,D)\ar[r]&H_c^2(W_F,\pi_0(D))\ar[r]&H_c^3(W_F,D^\circ)
}
\]
We have $H_c^2(W_F,D^\circ)=0$ by Proposition \ref{pro:cc_rajan} and $H_c^3(\Gamma,D^\circ)=0$ since $\tx{scd}(\Gamma)=2$. The map $H_c^2(\Gamma,\pi_0(D)) \to H_c^2(W_F,\pi_0(D))$ is bijective due to \cite[Theorem 3.1.2]{Karpuk}. The claim follows.
\end{proof}

\begin{rem} \label{rem:cc_nsurj_a}
The analogous claim in the archimedean case is false. When $F$ is complex, then $H^2_c(\Gamma,D)=0$. On the other hand, if we take $D=\mu_n(\C)$, then $H^2_c(W_F,D)=H^2_c(\C^\times,\mu_n(\C)) \neq 0$ by \cite[Theorem 10]{MooreCohIII}.

When $F$ is real, we take $D=\Z/3\Z$ with trivial $\Gamma$-action. Then $H_c^1(\C^\times,D)=0$, so Lemmas \ref{lem:cc_comp} and \ref{lem:cc_infres} give the exact sequence $H_c^2(\Gamma,D) \to H_c^2(W_F,D) \to H_c^2(\C^\times,D)^\Gamma \to H_c^3(\Gamma,D)$. But all $H_c^i(\Gamma,-)$ are 2-torsion groups, while $D$ is $2$-divisible, hence all these groups vanish, and we are left with the isomorphism $H_c^2(W_F,D) \to H_c^2(\C^\times,D)^\Gamma$. Again $H_c^2(\C^\times,D) \neq 0$.
\end{rem} 

\begin{lem} \label{lem:dual-k-z}
Let $S \to T$ be a homomorphism of $F$-tori with finite kernel $K$. Then
\begin{enumerate}
	\item The dual homomorphism $\hat T \to \hat S$ is surjective. 
	\item Write $\hat Z=\tx{ker}(\hat T \to \hat S)$ and $C=\tx{cok}(S \to T)$. Then
	\[ X^*(\hat Z^\circ) = X_*(C),\qquad X^*(\pi_0(\hat Z))=\tx{Hom}_\Z(X^*(K),\Q/\Z). \]
	\item The Pontryagin dual of the finite abelian group $K(F)$ is naturally identified with $H^2(W_F,\hat Z)$.
\end{enumerate}
\end{lem}
\begin{proof}
Statements (1) and (2) follow by applying $\tx{Hom}_\Z(-,\Z)$ to the exact sequence $0 \to X^*(C) \to X^*(T) \to X^*(S) \to X^*(K) \to 0$ and using the identification $\tx{Hom}_\Z(X^*(K),\Q/\Z) \to \tx{Ext}^1_\Z(X^*(K),\Z)$.

(3) By Pontryagin duality it is enough to show that in the piece of the long exact cohomology sequence $H^1_u(W_F,\hat T) \to H^1_u(W_F,\hat S) \to H^2(W_F,\hat Z)$ the second map, given by the connecting homomorphism, is surjective.

The preimage of $\hat S^c$ in $\hat T$ equals $\hat T^c$ and this leads to the exact sequence $1 \to \hat Z \to \hat T^c \to \hat S^c \to 1$. Lemma \ref{lem:langunit} reduces the problem to the vanishing of $H^2(W_F,\hat T^c)$. But this group is a direct factor of $H^2(W_F,\hat T)$, which vanishes according to Proposition \ref{pro:cc_rajan}.
\end{proof}

\subsection{Review of the algebraic fundamental group} \label{app:borovoi}

Let $F$ be a field, $F^s$ a fixed separable closure, and $\Gamma$ the Galois group of $F^s/F$. Let $G$ be a connected reductive $F$-group.

We briefly recall the concept of the universal maximal torus $T$ of $G$. Given any two Borel pair $(T_i,B_i)$ of $G_{F^s}$, for $i=1,2$, every $g \in G(F^s)$ such that $\tx{Ad}(g)(T_1,B_1)=(T_2,B_2)$ induces the same isomorphism $T_1 \to T_2$. We turn the set of Borel pairs of $G_{F^s}$ into a small category, in which there is a unique arrow between any two Borel pairs, and we have the functor from this small category into the category of tori which sends $(T_1,B_1)$ to $T_1$ and the unique arrow $(T_1,B_1) \to (T_2,B_2)$ to the isomorphism $T_1 \to T_2$. The limit of this functor is the universal maximal torus $T$. It comes equipped with the following additional structures:
\begin{enumerate}
	\item A root system $R(T,G) \subset X^*(T)$ equipped with a basis $\Delta \subset R(T,G)$.
	\item A root system $R^\vee(T,G) \subset X_*(T)$ equipped with a basis $\Delta^\vee \subset R^\vee(T,G)$.
	\item An $F$-structure for which the $\Gamma$-action leaves $\Delta$ and $\Delta^\vee$ stable. In this way, $(X^*(T),\Delta,X_*(T),\Delta^\vee)$ becomes a based root datum with $\Gamma$-action.
	\item An action of the Weyl group $\Omega(T,G)$ of the above root datum on $T$.
\end{enumerate}

Write $T_\tx{sc}$ for the torus determined by $X_*(T_\tx{sc})=Q^\vee$, where $Q^\vee \subset X_*(T)$ is the span of $R^\vee(T,G)$. Write further $T_\tx{ad}$ for the torus determined by $X^*(T_\tx{ad})=Q$, where $Q \subset X^*(T)$ is the span of $R(T,G)$. Then $T_\tx{sc}$ and $T_\tx{ad}$ are $F$-tori and we have the homomorphisms $T_\tx{sc} \to T \to T_\tx{ad}$, whose composition is an isogeny.

The Langlands dual group $\hat G$ of $G$ is the connected reductive group, defined over some fixed algebraically closed field, whose universal maximal torus $\hat T$ is equipped with the based root datum dual to that of $G$. Choosing a pinning of $\hat G$ the $\Gamma$-action on the based root datum lifts to a $\Gamma$-action on $\hat G$ by algebraic automorphisms. 

The algebraic fundamental group $\pi_1(G)$ defined by Borovoi, cf. \cite[\S1]{Brv98} is the finitely generated (discrete) abelian group with $\Gamma$-action defined as $X_*(T)/Q^\vee$. The assignment $G \mapsto \pi_1(G)$ is a covariant functor from the category of connected reductive $F$-groups to the category of $\Gamma$-modules. There is a natural identification $\pi_1(G) = X^*(Z(\hat G))$.

\end{appendices}

\bibliographystyle{amsalpha}
\bibliography{../../../TexMain/bibliography.bib}

\providecommand{\bysame}{\leavevmode\hbox to3em{\hrulefill}\thinspace}
\providecommand{\MR}{\relax\ifhmode\unskip\space\fi MR }
\providecommand{\MRhref}[2]{%
  \href{http://www.ams.org/mathscinet-getitem?mr=#1}{#2}
}
\providecommand{\href}[2]{#2}
\begin{thebibliography}{Wei18b}

\bibitem[Art17]{Art17}
James Arthur, \emph{Problems beyond endoscopy}, Representation theory, number
  theory, and invariant theory, Progr. Math., vol. 323,
  Birkh\"{a}user/Springer, Cham, 2017, pp.~23--45. \MR{3753907}

\bibitem[AV92]{AV92}
Jeffrey Adams and David~A. Vogan, Jr., \emph{{$L$}-groups, projective
  representations, and the {L}anglands classification}, Amer. J. Math.
  \textbf{114} (1992), no.~1, 45--138. \MR{1147719 (93c:22021)}

\bibitem[BD01]{BD01}
Jean-Luc Brylinski and Pierre Deligne, \emph{Central extensions of reductive
  groups by {$\bold K_2$}}, Publ. Math. Inst. Hautes \'{E}tudes Sci. (2001),
  no.~94, 5--85. \MR{1896177}

\bibitem[Bor98]{Brv98}
Mikhail Borovoi, \emph{Abelian {G}alois cohomology of reductive groups}, Mem.
  Amer. Math. Soc. \textbf{132} (1998), no.~626, viii+50. \MR{1401491}

\bibitem[BW00]{BorWal}
A.~Borel and N.~Wallach, \emph{Continuous cohomology, discrete subgroups, and
  representations of reductive groups}, second ed., Mathematical Surveys and
  Monographs, vol.~67, American Mathematical Society, Providence, RI, 2000.
  \MR{1721403}

\bibitem[Cas20]{Cas20}
W.~Casselman, \emph{A simple way to compute structure constants of semi-simple
  {L}ie algebras}, preprint, arXiv:2011.08274, 2020.

\bibitem[Del96]{Del96}
P.~Deligne, \emph{Extensions centrales de groupes alg\'{e}briques simplement
  connexes et cohomologie galoisienne}, Inst. Hautes \'{E}tudes Sci. Publ.
  Math. (1996), no.~84, 35--89 (1997). \MR{1441006}

\bibitem[Dil20]{Dillery20}
Peter Dillery, \emph{Rigid inner forms over local function fields}, 2020,
  preprint, arXiv:2008.04472.

\bibitem[FKS21]{FKS}
Jessica Fintzen, Tasho Kaletha, and Loren Spice, \emph{A twisted {Y}u
  construction, {H}arish-{C}handra characters, and endoscopy}, arXiv:2106.09120
  (2021).

\bibitem[HM62]{HM62}
G.~Hochschild and G.~D. Mostow, \emph{Cohomology of {L}ie groups}, Illinois J.
  Math. \textbf{6} (1962), 367--401. \MR{147577}

\bibitem[HR79]{HRv1e2}
Edwin Hewitt and Kenneth~A. Ross, \emph{Abstract harmonic analysis. {V}ol.
  {I}}, second ed., Grundlehren der Mathematischen Wissenschaften, vol. 115,
  Springer-Verlag, Berlin-New York, 1979, Structure of topological groups,
  integration theory, group representations. \MR{551496}

\bibitem[Kal16a]{KalSimons}
Tasho Kaletha, \emph{The local {L}anglands conjectures for non-quasi-split
  groups}, Families of automorphic forms and the trace formula, Simons Symp.,
  Springer, 2016, pp.~217--257. \MR{3675168}

\bibitem[Kal16b]{KalRI}
\bysame, \emph{Rigid inner forms of real and {$p$}-adic groups}, Ann. of Math.
  (2) \textbf{184} (2016), no.~2, 559--632. \MR{3548533}

\bibitem[Kal19a]{KalDC}
\bysame, \emph{On {$L$}-embeddings and double covers of tori over local
  fields}, arXiv:1907.05173 (2019).

\bibitem[Kal19b]{KalRSP}
\bysame, \emph{Regular supercuspidal representations}, J. Amer. Math. Soc.
  \textbf{32} (2019), no.~4, 1071--1170. \MR{4013740}

\bibitem[Kal19c]{KalSLP}
\bysame, \emph{Supercuspidal {$L$}-packets}, arXiv:1912.03274 (2019).

\bibitem[Kar13]{Karpuk}
David~A. Karpuk, \emph{Cohomology of the {W}eil group of a {$p$}-adic field},
  J. Number Theory \textbf{133} (2013), no.~4, 1270--1288. \MR{3003999}

\bibitem[Kos63]{Kos63}
Bertram Kostant, \emph{Lie group representations on polynomial rings}, Amer. J.
  Math. \textbf{85} (1963), 327--404. \MR{158024}

\bibitem[Kot84]{Kot84}
Robert~E. Kottwitz, \emph{Stable trace formula: cuspidal tempered terms}, Duke
  Math. J. \textbf{51} (1984), no.~3, 611--650. \MR{757954 (85m:11080)}

\bibitem[Kot86]{Kot86}
\bysame, \emph{Stable trace formula: elliptic singular terms}, Math. Ann.
  \textbf{275} (1986), no.~3, 365--399. \MR{858284 (88d:22027)}

\bibitem[Kot99]{Kot99}
\bysame, \emph{Transfer factors for {L}ie algebras}, Represent. Theory
  \textbf{3} (1999), 127--138 (electronic). \MR{1703328 (2000g:22028)}

\bibitem[KP22]{BTBOOK}
Tasho Kaletha and Gopal Prasad, \emph{{B}ruhat--{T}its theory: a new approach},
  New Mathematical Monographs, vol.~44, Cambridge University Press, 743pp.,
  2022.

\bibitem[KS]{KS12}
Robert~E. Kottwitz and Diana Shelstad, \emph{On splitting invariants and sign
  conventions in endoscopic transfer}, arXiv:1201.5658.

\bibitem[KS99]{KS99}
\bysame, \emph{Foundations of twisted endoscopy}, Ast\'erisque (1999), no.~255,
  vi+190. \MR{1687096 (2000k:22024)}

\bibitem[Lab85]{Lab85}
J.-P. Labesse, \emph{Cohomologie, {$L$}-groupes et fonctorialit\'{e}},
  Compositio Math. \textbf{55} (1985), no.~2, 163--184. \MR{795713}

\bibitem[Lan79]{Lan79}
R.~P. Langlands, \emph{Stable conjugacy: definitions and lemmas}, Canad. J.
  Math. \textbf{31} (1979), no.~4, 700--725. \MR{540901 (82j:10054)}

\bibitem[Lan97]{Lan97}
\bysame, \emph{Representations of abelian algebraic groups}, Pacific J. Math.
  (1997), no.~Special Issue, 231--250, Olga Taussky-Todd: in memoriam.
  \MR{1610871 (99b:11125)}

\bibitem[Lan04]{Lan04}
Robert~P. Langlands, \emph{Beyond endoscopy}, Contributions to automorphic
  forms, geometry, and number theory, Johns Hopkins Univ. Press, Baltimore, MD,
  2004, pp.~611--697. \MR{2058622}

\bibitem[Lan13]{Lan13}
\bysame, \emph{Singularit\'es et transfert}, Ann. Math. Qu\'e. \textbf{37}
  (2013), no.~2, 173--253. \MR{3117742}

\bibitem[Li11]{WenWeiLi11}
Wen-Wei Li, \emph{Transfert d'int\'{e}grales orbitales pour le groupe
  m\'{e}taplectique}, Compos. Math. \textbf{147} (2011), no.~2, 524--590.
  \MR{2776612}

\bibitem[Li19]{WenWeiLi19}
\bysame, \emph{Spectral transfer for metaplectic groups. {I}. {L}ocal character
  relations}, J. Inst. Math. Jussieu \textbf{18} (2019), no.~1, 25--123.
  \MR{3886154}

\bibitem[Li20]{WenWeiLi20}
\bysame, \emph{Stable conjugacy and epipelagic {$L$}-packets for
  {B}rylinski-{D}eligne covers of {${\rm Sp}(2n)$}}, Selecta Math. (N.S.)
  \textbf{26} (2020), no.~1, Paper No. 12, 123. \MR{4069853}

\bibitem[LS87]{LS87}
R.~P. Langlands and D.~Shelstad, \emph{On the definition of transfer factors},
  Math. Ann. \textbf{278} (1987), no.~1-4, 219--271. \MR{909227 (89c:11172)}

\bibitem[LS90]{LS90}
R.~Langlands and D.~Shelstad, \emph{Descent for transfer factors}, The
  {G}rothendieck {F}estschrift, {V}ol.\ {II}, Progr. Math., vol.~87,
  Birkh\"auser Boston, Boston, MA, 1990, pp.~485--563. \MR{1106907 (92i:22016)}

\bibitem[Moo76]{MooreCohIII}
Calvin~C. Moore, \emph{Group extensions and cohomology for locally compact
  groups. {III}}, Trans. Amer. Math. Soc. \textbf{221} (1976), no.~1, 1--33.
  \MR{414775}

\bibitem[PR94]{PR94}
Vladimir Platonov and Andrei Rapinchuk, \emph{Algebraic groups and number
  theory}, Pure and Applied Mathematics, vol. 139, Academic Press, Inc.,
  Boston, MA, 1994, Translated from the 1991 Russian original by Rachel Rowen.
  \MR{1278263 (95b:11039)}

\bibitem[Raj04]{Raj04}
C.~S. Rajan, \emph{On the vanishing of the measurable {S}chur cohomology groups
  of {W}eil groups}, Compos. Math. \textbf{140} (2004), no.~1, 84--98.
  \MR{2004125}

\bibitem[She83]{She83}
Diana Shelstad, \emph{Orbital integrals, endoscopic groups and
  {$L$}-indistinguishability for real groups}, Conference on automorphic theory
  ({D}ijon, 1981), Publ. Math. Univ. Paris VII, vol.~15, Univ. Paris VII,
  Paris, 1983, pp.~135--219. \MR{723184 (85i:22019)}

\bibitem[Wei18a]{Weissman18a}
Martin~H. Weissman, \emph{A comparison of {L}-groups for covers of split
  reductive groups}, Ast\'{e}risque (2018), no.~398, 277--286, L-groups and the
  Langlands program for covering groups. \MR{3802420}

\bibitem[Wei18b]{Weissman18b}
\bysame, \emph{L-groups and parameters for covering groups}, Ast\'{e}risque
  (2018), no.~398, 33--186, L-groups and the Langlands program for covering
  groups. \MR{3802418}

\bibitem[Wig73]{Wig73}
David Wigner, \emph{Algebraic cohomology of topological groups}, Trans. Amer.
  Math. Soc. \textbf{178} (1973), 83--93. \MR{338132}

\bibitem[Zha22]{Zhao22}
Yifei Zhao, \emph{Metaplectic group schemes}, 2022, preprint, arXiv:2204.00610.

\end{thebibliography}

\end{document}